%% file: main.tex
\documentclass[11pt,a4paper]{article}

\usepackage[a4paper, margin=2.5cm]{geometry}

\usepackage{import}
\usepackage{upgreek}
\usepackage{lipsum}
\usepackage{amsfonts}
\usepackage{graphicx}
\usepackage{epstopdf}
\usepackage{algorithmic}
\ifpdf
  \DeclareGraphicsExtensions{.eps,.pdf,.png,.jpg}
\else
  \DeclareGraphicsExtensions{.eps}
\fi

\usepackage[utf8]{inputenc}
\usepackage{amsmath,amsthm,amsfonts}
\usepackage{bm}
\usepackage{amssymb}
\usepackage{arydshln}
\usepackage{graphicx}
\usepackage[listofformat=parens]{subfig}
\usepackage{float}
\newtheorem{theorem}{Theorem}
\newtheorem{lemma}{Lemma}
\newtheorem{corollary}{Corollary}
\usepackage{hyperref}
\usepackage{stmaryrd}
\usepackage{gensymb}
\usepackage{empheq}
\usepackage[titletoc]{appendix}
\allowdisplaybreaks
\usepackage[square,numbers]{natbib}
\usepackage{chngcntr}
\usepackage{apptools}
\AtAppendix{\counterwithin{lemma}{section}}
\usepackage{caption}

\usepackage{setspace}
\usepackage{stmaryrd}
\usepackage{guit}
\usepackage{mathtools}
\usepackage{natbib}
\usepackage{graphicx}
\usepackage{psfrag}
\usepackage{pgfplots}
\usepackage{float}
\usepackage{bm}
\usepackage{xcolor}
\usepackage{tikz}
\usetikzlibrary{calc}  
\usetikzlibrary{arrows,positioning}
\usetikzlibrary{shapes.geometric,decorations.markings}
\usetikzlibrary{arrows, decorations.markings}

\usepackage{arydshln}
\usepackage[listofformat=parens]{subfig}
\usepackage{hyperref}
\usepackage{gensymb}
\usepackage{empheq}
\usepackage[titletoc]{appendix}
\allowdisplaybreaks
\usepackage[square,numbers]{natbib}
\usepackage{chngcntr}
\usepackage{apptools}
\usepackage{caption}
\usepackage{setspace}
\usepackage{stmaryrd}
\usepackage{guit}
\usepackage{mathtools}
\usepackage{natbib}
\usepackage{graphicx}
\usepackage{psfrag}
\usepackage{pgfplots}
\usepackage{float}
\usepackage{bm}
\usepackage{xcolor}
\usepackage{tikz}
\usetikzlibrary{calc}  
\usetikzlibrary{arrows,positioning}
\usetikzlibrary{shapes.geometric,decorations.markings}
\usetikzlibrary{arrows, decorations.markings}

\theoremstyle{definition}
\newtheorem{definition}{Definition}[section]
\newtheorem{remark}{Remark}[section]
\newtheorem{assumption}{Assumption}[section]
\usepackage{setspace}
\hypersetup{
	colorlinks=true,
	linkcolor=blue,
	filecolor=magenta,      
	urlcolor=cyan,
}
\newcommand{\llbrace}{\lbrace\hspace{-0.15cm}\lbrace}
\newcommand{\rrbrace}{\rbrace\hspace{-0.15cm}\rbrace}
\newcommand{\norm}[1]{\left\lVert#1\right\rVert}
\newcommand{\trinorm}[1]{{\left\vert\kern-0.25ex\left\vert\kern-0.25ex\left\vert #1 
		\right\vert\kern-0.25ex\right\vert\kern-0.25ex\right\vert}}
\newcommand{\myred}[1]{\textcolor{black}{#1}}
\newcommand{\rred}[1]{\textcolor{black}{#1}}

\newcommand*\circled[1]{\tikz[baseline=(char.base)]{
    \node[shape=circle, draw, inner sep=1pt, 
        minimum height=12pt] (char) {#1};}}

\newcommand{\MB}[1]{\textcolor{black}{#1}}
\newcommand{\RRR}[1]{\textcolor{black}{#1}}
\newcommand{\kint}{\tau}

\newcommand{\HH}{\mathbb{H}}
\newcommand{\UU}{\mathfrak{\bm U}}
\newcommand{\VV}{\mathfrak{\bm V}}

\begin{document}
\title{\sc A high-order discontinuous Galerkin method for the poro-elasto-acoustic problem on polygonal and polyhedral grids} 
\author{P.F.~Antonietti$^{1}$, M.~Botti$^{2}$, I.~Mazzieri$^{3}$, S.~Nati Poltri$^{4}$}

\maketitle 
\begin{center}
	{\small $^1$ MOX, Dipartimento di Matematica, Politecnico di Milano, Italy.\\ {\tt paola.antonietti@polimi.it}}\\
	{\small $^2$ MOX, Dipartimento di Matematica, Politecnico di Milano, Italy. \\ {\tt michele.botti@polimi.it}}\\
	{\small $^3$ MOX, Dipartimento di Matematica, Politecnico di Milano, Italy. \\ {\tt ilario.mazzieri@polimi.it}}\\
	{\small $^4$ MOX, Dipartimento di Matematica, Politecnico di Milano, Italy. \\ {\tt simone.nati@mail.polimi.it}}\\
\end{center}--
\date{}

\vspace*{0.5cm}
\noindent

\begin{abstract}
The aim of this work is to introduce and analyze a  finite element discontinuous Galerkin method on polygonal meshes for the numerical discretization of acoustic waves propagation through poroelastic materials. Wave propagation is modeled by the acoustics equations in the acoustic domain and the low-frequency Biot's equations in the poroelastic one. The coupling is \MB{realized by means of} (physically consistent) transmission conditions, imposed on the interface between the domains, modeling \MB{different pores configurations}.
%Existence and uniqueness is proven for the strong formulation based on employing the semigroup theory. 
For the space discretization we introduce and analyze a high-order discontinuous Galerkin method on polygonal and polyhedral meshes, which is then coupled with Newmark-$\beta$ time integration schemes. A stability analysis for both the continuous and  semi-discrete problem is presented and error estimates for the energy norm are derived for the semi-discrete one. A wide set of numerical results obtained on test cases with manufactured solutions are presented in order to validate the error analysis. Examples of physical interest are also presented to \MB{investigate} the capability of the proposed methods in practical \MB{scenarios}.
\end{abstract}

\section{Introduction}
The paper deals with the numerical analysis of the coupled poro-elasto-acoustic  differential problem modeling an acoustic/sound wave impacting a poroelastic medium and consequently propagating through it.  Coupled poro-elasto-acoustic problems model the combined propagation of pressure and elastic waves through a porous material. Pressure waves propagate through the saturating fluid inside pores, while acoustic ones through the porous skeleton. The theory of propagation of acoustic waves with application to poroelasticity has been developed mainly by Biot \cite{biot1941general} in 1956, by introducing general equations and proposing different ways to treat coupling between acoustic and poro-elastic domains. Pioneering advances of Biot's theory concerned with slow compressional waves, whose study carried on the analysis on fast compressional waves, introduced in 1944 by Frenkel. 
%In \cite{smeulders1992wave} it is proposed a model of seismic waves in saturated soils, distinguishing in-phase (\textit{fast}) movements between solid and fluid from out-phase (\textit{slow}) ones. 
Coupled poro-elasto-acoustic models find  application in many science and engineering fields. For example, in acoustic engineering,  for the study of sound propagation through acoustic panels, whose main intent is to intercept and absorb acoustic waves for noise reduction \cite{TKWF2010};  in civil engineering, for the study of passive control and vibroacoustics, where plastic foams and fibrous or granular materials are mainly used with this intent \cite{krishnan}; in aeronautical engineering, where air-saturated porous materials are employed \cite{castagnede1998ultrasonic}; in biomedical engineering,  for the study of ultrasound propagation throughout bones to diagnose osteoporosis and study its evolution \cite{HAIRE1999291} and to model soft tissues deformation, such as the heart tissue \cite{huyghe1991two}, the skin \cite{oomens1987mixture} and the aortic tissue \cite{jayaraman1983water}.   
Poro-elasto-acoustic models find a wide strand of literature also in computational geosciences: we refer the reader to \cite{carcione2014book} for a comprehensive review.

In order to model the poroelastic domain, the concept of \textit{pores} is necessary.  \textit{Pores} can be seen as "holes" in the material where a fluid is able to move. They can be classified into \textit{open}, \textit{sealed}, and \textit{imperfect} pores: the first ones share a part with the outer surface of the material, the second ones are totally locked in, while the latter ones represent an itermediate state between the former two, as shown in Figure~\ref{fig::pores} below. From the modeling viewpoint, the difference between them is the way in which interface conditions are formulated, as detailed later on. 

Concerning the numerical discretization of poro-elasto-acoustic models, we mention the Lagrange Multipliers  method \cite{rockafellar1993lagrange,zunino,Flemisch2006}, the finite element method \cite{BERMUDEZ200317,FKTW2010}  the spectral and pseudo-spectral element method \cite{Morency2008,Sidler2010}, the ADER scheme \cite{delapuente2008,chiavassa_lombard_2013}, the finite difference method \cite{LOMBARD200490}, and references therein.

To accurately simulate wave propagation in coupled poro-elasto-acoustic domains the numerical scheme should take into account the following observations: (i) in the low-frequency range the evolution problem become stiff \cite{delapuente2008}, and therefore, explicit time integration schemes might become computationally too demanding due to the strict stability constraint; (ii) the diffusive slow compressional waves are localized near the interfaces, and therefore, mesh refinements are needed to capture the phenomenon; (iii) an accurate geometrical description of the arbitrary complex interfaces is crucial; (iv) a proper representation of the hydraulic contact at the interfaces is also mandatory to correctly capture the physics of the problem. 

By taking into consideration the aforementioned difficulties, the aim of this paper is to propose and analyze a high-order discontinuous Galerkin method on polygonal and polyhedral grids (PolyDG) for the space discretization of a coupled poroelasto-acoustic  problem, by extending  the theory carried out in \cite{bonaldi}, where a coupled system of elasto-acoustic equations is analyzed. 
We point out that the geometric flexibility due to mild regularity requirements on the underlying computational mesh together with the arbitrary-order accuracy featured by the proposed PolyDG method are crucial within this context as they ensure at the same time a high-level of flexibility in the representation of the geometry and an intrinsic high-level of precision and scalability that are mandatory to correctly represent the solution fields. Moreover, in the proposed semi-discrete formulation, the coupling between the acoustic and the poroelastic domains is introduced by considering (physically consistent) interface conditions, naturally incorporated in the scheme. 

For early results in the field of dG methods we refer, for example, to \cite{BaBoCoDiPiTe2012,AntoniettiGianiHouston_2013,cangiani2014hp,CangianiDongGeorgoulisHouston_2016,CongreveHouston2019,cangiani2020hpversion}
for second-order elliptic problems problems, to \cite{CangianiDongGeorgoulis_2017} for 
parabolic differential equations, to \cite{AntoniettiFacciolaRussoVerani_2019} for flows in fractured porous media, to \cite{AntoniettiVeraniVergaraZonca_2019}  for fluid structure interaction problems, cf. also \cite{CangianiDongGeorgoulisHouston_2017} for a comprehensive monograph.
In the framework of dG methods for hyperbolic problems we mention \cite{riviere2003discontinuous,GrScSc06} for scalar wave equation on simplex grids, while more recent dG discretizations on polytopic meshes can be found in  
\cite{AntoniettiMazzieri2018} for elastodynamics problems, in \cite{AntoniettiMazzieriMuhrNikolicWohlmuth_2020} for non-linear sound waves and in \cite{bonaldi,AntoniettiBonaldiMazzieri_2019b} for coupled elasto-acoustic problems.
To the best of our knowledge, the present approach is proposed and analyzed here for the first time in the context of multiphysics poroelasto-acoustic problems, and it provides a flexible and accurate scheme that can be employed in real applications. 

 The remaining part of the paper is structured as follows: in Section \ref{sec::physical} we introduce the mathematical model, present the weak formulation of the problem, and prove suitable stability estimates. In Section \ref{sec::numerical} we introduce the PolyDG approximation and prove its stability. Section \ref{sec::errors} is devoted to the analysis of the semi-discrete problem and the proof of $hp-$version \textit{a-priori} error estimates. The time integration schemes are introduced in Section \ref{sec::timedis}. In Section \ref{sec::results} we present some two-dimensional numerical experiments to validate the theoretical results and show the performances of the proposed method in examples of physical interest. Finally, in Section \ref{sec::conclusions} we draw some conclusions. The existence and uniqueness for the strong formulation of the problem and additional technical results are established in Appendix \ref{appendix}.

\section{The physical model and governing equations}\label{sec::physical}

Let $\Omega\subset\mathbb{R}^d$, $d=2,3$, be an open, convex polygonal/polyhedral  domain  decomposed as the union of two \MB{disjoint}, polygonal/polyhedral subdomains: $\Omega=\Omega_p\cup\Omega_a$, representing the poroelastic and the acoustic domains, respectively, cf. Figure~\ref{fig::domain}. The two subdomains share part of their boundary, resulting in the interface $\Gamma_I=\partial\Omega_p\cap\partial\Omega_a$. 
%%%%%%
\begin{figure}
\caption{(\ref{fig::pores}) Pores classification in a poroelastic domain:  \textit{sealed} (1), \textit{open} (2) and \textit{imperfect} (3) pores.
(\ref{fig::domain}) Simplified graphic representation of the domain $\Omega=\Omega_p\cup\Omega_a$ for $d=2$.}
 \subfloat[Pores classification in a poroelastic domain. \label{fig::pores}]{\includegraphics[scale=0.3]{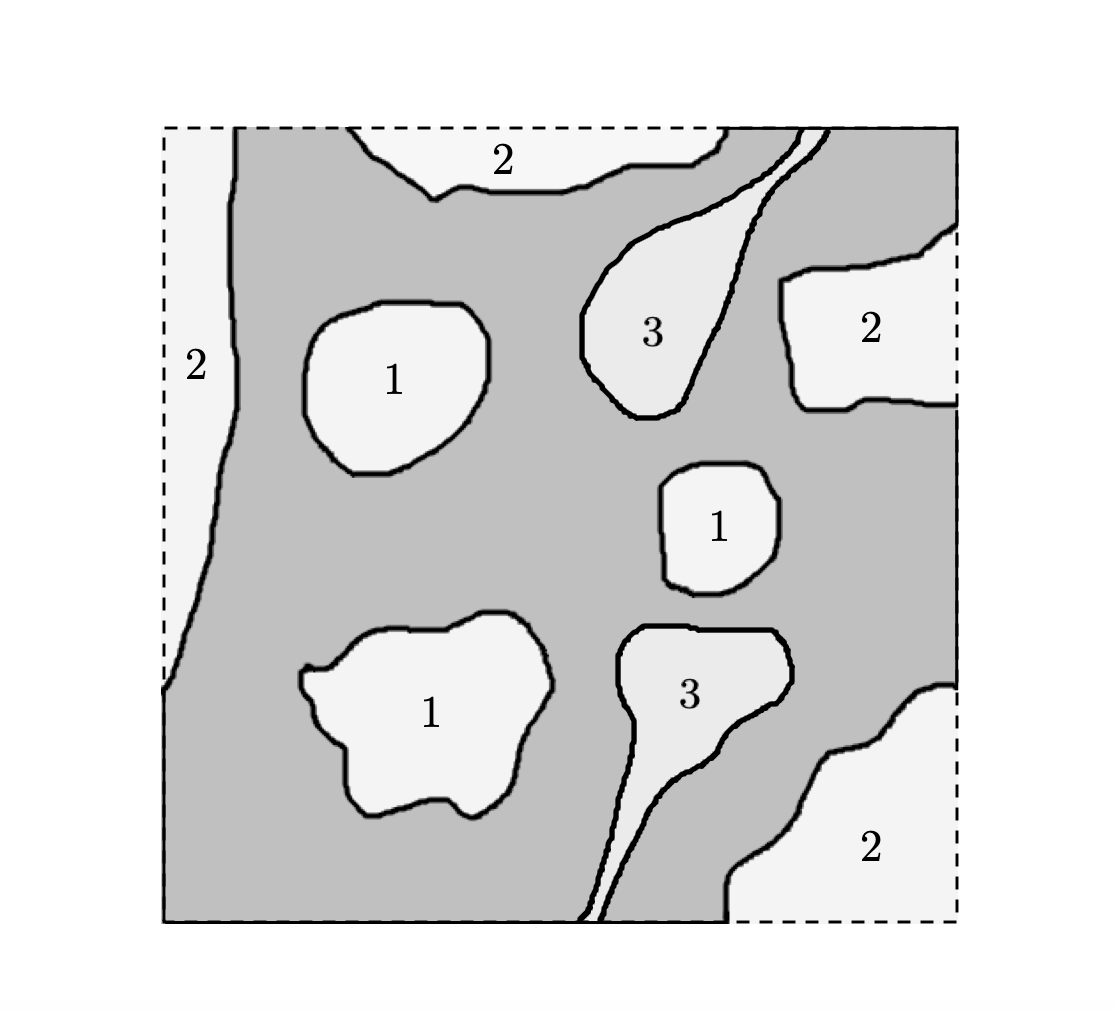}}
\subfloat[$\Omega=\Omega_p\cup\Omega_a$. \label{fig::domain}]{\includegraphics[scale=0.4]{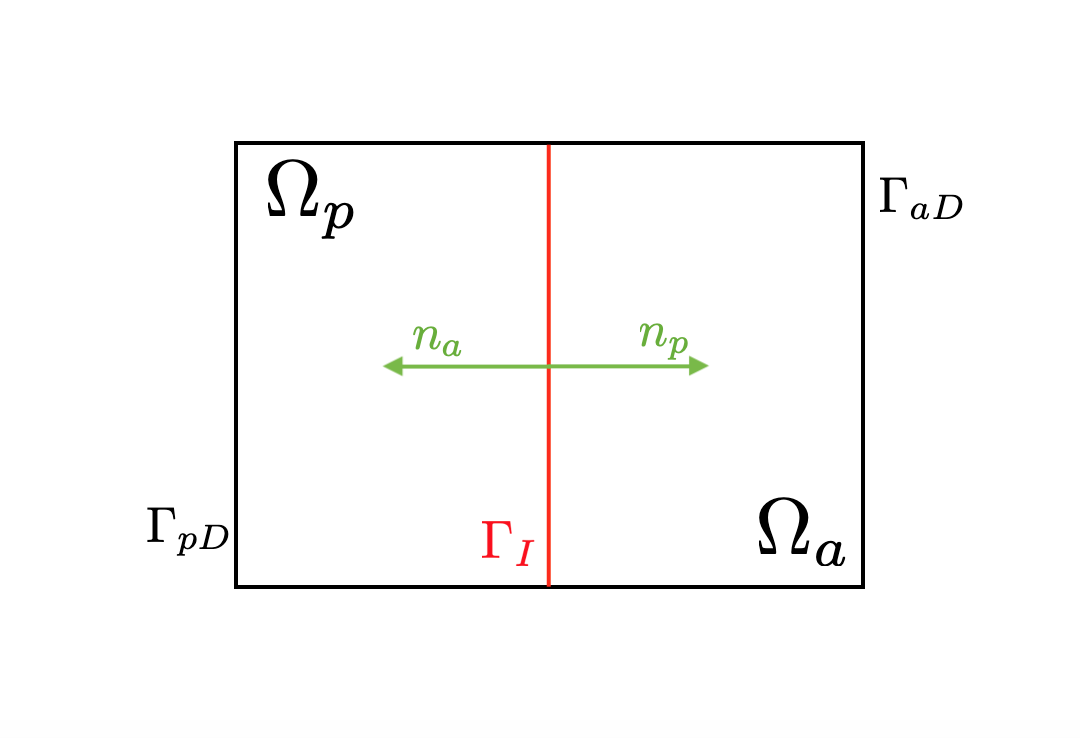}}
\end{figure}
The boundary of $\Omega$ is denoted by $\partial\Omega$, and we set $\partial\Omega_p=\Gamma_{pD}\cup\Gamma_I$ and $\partial\Omega_a=\Gamma_{aD}\cup\Gamma_I$, with $\Gamma_{pD}\cap\Gamma_I=\emptyset$ and $\Gamma_{aD}\cap\Gamma_I=\emptyset$. Surface measures of $\partial\Omega$, $\partial\Omega_p$, $\partial\Omega_a$ and $\Gamma_I$ are assumed to be strictly positive. The outer unit normal vectors to $\partial\Omega_p$ and $\partial\Omega_a$ are denoted by $\bm{n}_p$ and $\bm{n}_a$, respectively, so that $\bm{n}_p=-\bm{n}_a$ on $\Gamma_I$.
In the following, for $X\subseteq\Omega$, the notation $\bm{L}^2(X)$ is adopted in place of $[L^2(X)]^d$, with $d\in\{2,3\}$. The scalar product \MB{in $L^2(X)$} is denoted by $(\cdot,\cdot)_X$, with associated norm $\norm{\cdot}_X$.  Similarly, $\bm{H}^\ell(X)$ is defined as $[H^\ell(X)]^d$, with $\ell\geq 0$, equipped with the norm $\norm{\cdot}_{\ell,X}$, assuming conventionally that $\bm{H}^0(X)\equiv\bm{L}^2(X)$. \RRR{In addition we will use $\bm{H}(\textrm{div},X)$ to denote the space of $\bm{L}^2(X)$ functions with  square integrable divergence. In order to take into account essential boundary conditions, we also introduce the zero-trace subspaces, defined as  
$$
\begin{aligned}
H^1_0(\Omega_a) &= \{ \psi\in H^1(\Omega_a) \,|\, \psi_{|\Gamma_{aD}} = 0\}, \\
\bm{H}^1_0(\Omega_p) &= \{ \bm{v}\in \bm{H}^1(\Omega_p) \,|\, \bm{v}_{|\Gamma_{pD}} = \bm{0}\}, \\
\bm{H}_0(\textrm{div},\Omega_p) &= \{ \bm{z}\in \bm{H}(\textrm{div},\Omega_p) \,|\, (\bm{z}\cdot\bm{n}_p)_{|\Gamma_{pD}} = 0\}.
\end{aligned}
$$}
Given $k\in\mathbb{N}$ and a Hilbert space $\mathbb{H}$, the usual notation $C^k([0,T];\mathbb{H})$ is adopted for the space of \MB{$\mathbb{H}$-valued functions, $k$-times continuously differentiable in $[0,T]$}. 
The notation $x\lesssim y$ stands for $x\leq C y$, with $C>0$, independent of the discretization parameters, but possibly dependent on physical coefficients and the final time $T$.
\subsection{The poro-elasto-acoustic problem}
To model wave propagation in a poro-elastic domain $\Omega_p$ we consider the \textit{two-displacement} formulation of \cite{matuszy}, written in the solid and filtration displacements, denoted by $\bm{u}$ and $\bm{w}$, respectively. For a final observation time $T>0$, we consider the low-frequency Biot's equations:
\begin{equation}
\begin{cases}
\rho\ddot{\bm{u}} + 
\rho_f\ddot{\bm{w}}
-\nabla\cdot\bm{\sigma}=\bm{f}_p,  &\text{in }\Omega_p\times(0,T],\\[5pt]
\rho_f\ddot{\bm{u}} + 
\rho_w\ddot{\bm{w}} + 
\frac{\eta}{k}\dot{\bm{w}}
+\nabla p=\bm{g}_p,  &\text{in }\Omega_p\times(0,T].
\end{cases}
\label{eq::poroel}
\end{equation}
Here, the average density $\rho$ is given by $\rho=\phi\rho_f+(1-\phi)\rho_s$, where $\rho_s>0$ is the solid density, $\rho_f>0$ is the saturating fluid density, 
$\rho_w$ is defined as $\rho_w=\frac{a}{\phi}\rho_f$, being $\phi$ the \textit{porosity} satisfying $0<\phi_0\leq\phi\leq\phi_1<1$, and being $a>1$ the \textit{tortuosity} measuring the deviation of the fluid paths from straight streamlines, cf. \cite{souzanchi2013tortuosity}. In \eqref{eq::poroel}, $\eta>0$ represents the dynamic \textit{viscosity} of the fluid and $k>0$ is the absolute \textit{permeability}. 
\begin{remark}
As observed in \cite{chiavassa_lombard_2013}, the second equation in \eqref{eq::poroel} is valid under a constraint on frequencies, i.e. the spectrum of the waves has to lie in the low-frequency range. In what follows, we only consider frequencies lower than $f_c=\eta\phi/(2\pi a k \rho_f)$.
\end{remark}
In $\Omega_p$,  we assume the following constitutive laws for the stress $\bm{\sigma}$ and pressure $p$:
\begin{align}
& \bm{\sigma}(\bm{u},p)= \mathbb{C}:\bm{\epsilon}(\bm{u})
-\beta p \bm{I}, 
&& p(\bm{u},\bm{w}) = -m(\beta \nabla\cdot\bm{u}+\nabla\cdot \bm{w}),
\label{eq::const_sigma_press}
\end{align}
where the strain tensor $\bm \epsilon(\cdot)$ is defined as $\bm{\epsilon}(\bm{u})=\frac{1}{2}(\nabla\bm{u}+\nabla\bm{u}^T)$, and $\mathbb{C}$ is the fourth-order, symmetric and uniformly elliptic elasticity tensor \MB{defined by
 \begin{equation*}
\mathbb{C}:\bm{\tau} = 2\mu \bm{\tau} + \lambda \rm{tr}(\bm\tau), \qquad\text{for all } \bm{\tau}\in\mathbb{R}^{d\times d},
\end{equation*}
with $\rm{tr}(\bm{\tau}) = \sum_{i=1}^d \bm{\tau}_{ii}$.
}
Here, \MB{$\lambda\ge0$ and $\mu\ge\mu_0>0$} are the Lam\'e coefficients of the elastic skeleton. In \ref{eq::const_sigma_press}, the Biot--Willis coefficient $\beta$ and Biot modulus $m$ are such that $\phi<\beta\le1$ and $m\ge m_0>0$. It can be shown that the dilatation coefficients of the saturated matrix corresponds to $\lambda_f=\lambda+\beta^2m$.
By plugging the constitutive laws \eqref{eq::const_sigma_press} into \eqref{eq::poroel}, we obtain the \textit{two-displacement} formulation
\begin{equation}
\begin{cases}
\rho\ddot{\bm{u}} + 
\rho_f\ddot{\bm{w}}
-\nabla\cdot(\mathbb{C}:\bm{\epsilon}(\bm{u}))-\beta^2m\nabla(\nabla\cdot\bm{u})-\beta m \nabla(\nabla\cdot\bm{w})=\bm{f}_p, \\
\rho_f\ddot{\bm{u}} + 
\rho_w\ddot{\bm{w}} + 
\frac{\eta}{k}\dot{\bm{w}}
-\beta m\nabla(\nabla\cdot\bm{u})-m\nabla(\nabla\cdot\bm{w})=\bm{g}_p.
\end{cases}
\label{eq::poroel2}
\end{equation}
\begin{remark}
We point out that the $(\bm{u},\bm{w})$ formulation \eqref{eq::poroel2} is not the unique possible choice. For example, one could write the equations considering the velocity of the solid skeleton $\dot{\bm{u}}$ and the filtration velocity $\dot{\bm{w}}$ as unknowns, cf. \cite{chiavassa_lombard_2013}, or consider a velocity-pressure $(\bm{u},p)$ formulation, as in \cite{zunino,wohlmuth,botti,phillips2008coupling}. Here, the \textit{two-displacement} formulation turns out to be convenient in view of the coupling conditions stated below.
\end{remark}

In the fluid domain $\Omega_a$, we consider an acoustic wave with constant velocity \MB{$c>0$ and mass density $\rho_a>0$}. 
For a given source term $f_a$, the acoustic potential $\varphi$ satisfies
\begin{equation}
c^{-2}\ddot{\varphi}-\rho_a^{-1}\nabla \cdot( \rho_a \nabla\varphi) = f_a, \quad \text{in }\Omega_a\times(0,T].
\label{eq::acousticeq}
\end{equation}
Finally, we discuss the transmission conditions on $\Gamma_I$. The poro-elasto-acoustic coupling is realized through interface conditions, cf. \cite{gurevich1999interface}, expressing the continuity of normal stresses and conservation of mass. The continuity of the pressure is prescribed by writing the acoustic potential in terms of a pressure. Thus, on $\Gamma_I$ we impose
\begin{align}
-\bm{\sigma}\bm{n}_p & = \rho_a\dot{\varphi}\bm{n}_p,    \label{eq::contstress}\\
(\dot{\bm{u}}+\dot{\bm{w}})\cdot\bm{n}_p &= -\nabla\varphi\cdot\bm{n}_p, 	\label{eq::cont_vel} \\
\kint[p] & = \MB{(1-\kint)}\dot{\bm{w}}\cdot\bm{n}_p,    \label{eq::cont_pres} 
\end{align}
where  $[\cdot]$ denotes the jump operator at the interface $\Gamma_I$, i.e. \MB{$[p]=p(\bm{u},\bm{w})-p_a(\varphi)$ with $p_a(\varphi)=\rho_a\dot{\varphi}$, and $0\le\kint\le1$ is the hydraulic permeability at the interface and models both open, sealed, and imperfect pores, cf. \ref{fig::pores}. The stress tensor $\bm{\sigma}$ and  the pressure $p(\bm{u},\bm{w})$ obey the constitutive equations \eqref{eq::const_sigma_press}.
If $\kint=1$ (\textit{open} pores), equation \eqref{eq::cont_pres} reduces to the continuity of pressure at the interface, that is $p(\bm{u},\bm{w})=\rho_a\dot{\varphi}$}.  If $\kint=0$ (\textit{sealed} pores), \eqref{eq::cont_pres} simplifies to $\dot{\bm{w}}\cdot\bm{n}_p=0$, that implies that \eqref{eq::cont_vel} imposes a continuity only on the solid velocity, namely $\dot{\bm{u}}\cdot\bm{n}_p=-\nabla\varphi\cdot\bm{n}_p$. If $\kint \in (0,1)$ (\textit{imperfect} pores) then an intermediate state between \textit{open} and \textit{sealed} pores occurs.

Supplementing the constitutive equations with suitable boundary conditions (here supposed for simplicity to be of homogeneous Dirichlet type), the \textit{poro-elasto-acoustic problem} reads as: for any  $t \in (0,T]$, find $(\bm{u},\bm{w},\varphi): \Omega_p \times \Omega_p \times \Omega_a\rightarrow\mathbb{R}$ such that:
\begin{equation}\label{system}
\begin{aligned}
\rho\ddot{\bm{u}} + 
\rho_f\ddot{\bm{w}}
-\nabla\cdot(\mathbb{C}:\bm{\epsilon}(\bm{u}))-\beta m\nabla(\beta\nabla\cdot\bm{u}+\nabla\cdot\bm{w})&=\bm{f}_p, 
&& \text{in }\Omega_p,\\
\rho_f\ddot{\bm{u}} + \rho_w\ddot{\bm{w}} + \frac{\eta}{k}\dot{\bm{w}}
- m\nabla(\beta\nabla\cdot\bm{u}+\nabla\cdot\bm{w})&=\bm{g}_p,
&& \text{in }\Omega_p\\
 \rho_a c^{-2}  \ddot{\varphi} - \nabla \cdot (\rho_a\nabla\varphi)& = \rho_a  f_a 
&& \text{in }\Omega_a,\\
-(\mathbb{C}:\bm{\epsilon}(\bm{u})+\beta m(\beta \nabla\cdot\bm{u}+\nabla\cdot \bm{w}) \bm{I})\bm{n}_p  &= \rho_a\dot{\varphi}\bm{n}_p, 
&& \text{on } \Gamma_I,\\
(\dot{\bm{u}}+\dot{\bm{w}})\cdot\bm{n}_p &= -\nabla\varphi\cdot\bm{n}_p, 	
&& \text{on } \Gamma_I,\\
- m(\beta \nabla\cdot\bm{u}+\nabla\cdot \bm{w})-\MB{\kint^{-1}(1-\kint)\dot{\bm{w}}\cdot\bm{n}_p} &=
\MB{\rho_a \dot{\varphi}}, 
&& \text{on } \Gamma_I,
\end{aligned}
\end{equation}
together with initial conditions $\bm{u}(\cdot,0)=\bm{u}_0$, $\bm{w}(\cdot,0)=\bm{w}_0$, $\dot{\bm{u}}(\cdot,0)=\bm{u}_1$, $\dot{\bm{w}}(\cdot,0)=\bm{w}_1$,  in $\Omega_p$ and $\varphi(\cdot,0)=\varphi_0$, $\dot{\varphi}(\cdot,0)=\varphi_1 $ in $\Omega_a$. Notice that the acoustic equation has been multiplied by $\rho_a$. 
The existence and uniqueness of a strong solution to \eqref{system} is proved in Appendix \ref{appendix} by employing the semigroup theory.

\subsection{Weak formulation and stability estimates}
\MB{In order to derive a unified analysis for $0\le\kint\le 1$, we introduce the space
\begin{equation}
\label{eq:W_kint}
\bm{W}_{\kint} = \begin{cases}
\bm H_0(\textrm{div},\Omega_p), \qquad &\text{if }\, \kint = 1,\\
\{ \bm{z}\in \bm H_0(\textrm{div},\Omega_p) \,|\, \zeta(\kint)^{\frac12}(\bm{z}\cdot\bm{n}_p)_{|\Gamma_I} \in L^2(\Gamma_I)\},
\qquad &\text{if }\, \kint \in (0,1),\\
\{ \bm{z}\in \bm H_0(\textrm{div},\Omega_p) \,|\, (\bm{z}\cdot\bm{n}_p)_{|\Gamma_I} = 0 \},
\qquad &\text{if }\, \kint =0,
\end{cases}
\end{equation}
equipped with the norm $\norm{\cdot}_{\bm{W}_{\kint}}$ defined, for all $\bm{z}\in\bm{W}_{\kint}$, as
\begin{equation}
\label{eq:Wt_norm}
\norm{\bm{z}}_{\bm{W}_{\kint}} = \norm{\bm{z}}_{\Omega_p} + \norm{\nabla\cdot\bm{z}}_{\Omega_p} + 
\norm{\zeta(\kint)^{\frac12}\ \bm{z}\cdot\bm{n}_p}_{\Gamma_I},
\quad\text{with }\, \zeta(\kint) =
\begin{cases}
\frac{1-\kint}{\kint} \;\text{ for }\kint \in (0,1],\\
0\quad\;\;\text{ for }\kint=0.
\end{cases}
\end{equation}}
We also define the Hilbert space $ \HH = \bm H^1_0(\Omega_p) \times \RRR{\bm{W}_{\tau}} \times H^1_0(\Omega_a)$ and $\Omega_* = \Omega_p \times \Omega_p \times \Omega_a$.
The weak form  of \eqref{system} reads as: 
for any $t\in(0,T]$, find \MB{$(\bm{u},\bm{w},\varphi)(t) \in \HH $} s.t.
\begin{multline}
\MB{\mathcal{M} ((\ddot{\bm{u}},\ddot{\bm{w}},\ddot{\varphi}), (\bm{v},\bm{z},\psi))  +
\mathcal{A}((\bm{u},\bm{w},\varphi),(\bm{v},\bm{z},\psi)) + \mathcal{B} (\dot{\bm{w}},\bm{z})} \\
+ \mathcal{C}^p(\dot{\varphi},\bm{v} + \bm{z}) + \mathcal{C}^a(\dot{\bm{u}} + \dot{\bm{w}},\psi)
=  ((\bm{f}_p, \bm{g}_p, \rho_a f_a), (\bm{v},\bm{z},\psi))_{\Omega_*}\label{eq::weakform}
\end{multline}
for all $(\bm{v},\bm{z},\psi) \in \HH$, where for any \RRR{$\UU = (\bm u, \bm w, \varphi), \VV = (\bm v, \bm z, \psi) \in \HH$} we have set
\MB{\begin{equation}\label{eq:bilinear_forms}
    \begin{aligned}
    \mathcal{M}(\UU,\VV)  & =  (\rho \bm{u} + \rho_f \bm w , \bm{v})_{\Omega_p} 
    +  ( \rho_f \bm u + \rho_w \bm w, \bm z)_{\Omega_p} + (\rho_a c^{-2} \varphi, \psi)_{\Omega_a}, \\
    \mathcal{A}(\UU,\VV)  & = (\mathbb{C}:\bm{\epsilon}(\bm{u}),\bm{\epsilon}(\bm{v}))_{\Omega_p}+
    (m\nabla\cdot(\beta\bm{u}+\bm w),\nabla\cdot(\beta\bm{v}+\bm{z}))_{\Omega_p}+ 
    (\rho_a\nabla \varphi,\nabla \psi)_{\Omega_a}, \\
	\mathcal{B} (\bm{w},\bm{z}) & = 
	(\eta k^{-1} \bm{w},\bm{z})_{\Omega_p}+ (\zeta(\kint)\bm{w}\cdot\bm{n}_p, \bm{\bm{z}}\cdot\bm{n}_p)_{\Gamma_I}, \\
	\mathcal{C}^p(\varphi,\bm{z}) & =\langle \rho_a \varphi,\bm{z}\cdot\bm{n}_p\rangle_{\Gamma_I}
	= -\mathcal{C}^a(\bm{z},\varphi),
\end{aligned}
\end{equation}
\MB{with $\zeta(\kint)$ defined in \eqref{eq:Wt_norm}.}
Notice that, if $\kint = 0$,  the terms $\mathcal{C}^p(\dot\varphi,z)$ and $\mathcal{C}^a(\dot{\bm{w}},\psi)$ in \eqref{eq::weakform} are null thanks to the definition of $\bm{W}_{\kint}$ which strongly enforces condition \eqref{eq::cont_pres}}.

\MB{Before presenting a stability estimate for the solution of problem \eqref{eq::weakform} we define, for all $\UU = (\bm u, \bm w, \varphi) \in C^1([0,T];\bm{L}^2(\Omega_\star))\cap C^0([0,T];\mathbb{H})$, the energy norm
\begin{equation}\label{eq:energy}
\norm{\UU}_{\mathbb{E}}^2 = \max_{t\in[0,T]}
\norm{\UU(t)}_{\mathcal{E}}^2 =
\max_{t\in(0,T]}\left(\mathcal{M}(\dot{\UU},\dot{\UU})(t)+\mathcal{A}(\UU,\UU)(t)+\mathcal{B}(\bm{w},\bm{w})(t)\right). 
\end{equation}
As a result of the next Lemma,  $\norm{\cdot}_{\mathbb{E}}$ is a norm on $C^1([0,T];\bm{L}^2(\Omega_\star))\cap C^0([0,T];\mathbb{H})$.}
\MB{
\begin{lemma}\label{lemma:stab}
The bilinear forms $\mathcal{M}$, $\mathcal{A}$, and $\mathcal{B}$ defined in \eqref{eq:bilinear_forms} are such that
\begin{align}
\mathcal{M}(\UU,\VV) & \lesssim \norm{\UU}_{\Omega_*} \norm{\VV}_{\Omega_*}, \label{eq:M-cont}\\ 
\mathcal{M}(\UU,\UU) & \gtrsim  \norm{\UU}_{\Omega_*}^2,  \label{eq:M-coer} \\
\mathcal{A}(\UU,\VV) + \mathcal{B}(\bm{w},\bm{z}) & \lesssim 
\norm{\bm u}_{1,\Omega_p} \norm{\bm v}_{1,\Omega_p} + \norm{\bm w}_{\bm{W}_{\kint}}\norm{\bm z}_{\bm{W}_{\kint}} + \norm{\varphi}_{1,\Omega_a}\norm{\psi}_{1,\Omega_a}, \label{eq:A-cont}\\
\mathcal{A}(\UU,\UU) + \mathcal{B}(\bm{w},\bm{w}) & \gtrsim 
\norm{\bm u}_{1,\Omega_p}^2 + \norm{\bm w}_{\bm{W}_{\kint}}^2 + \norm{\varphi}_{1,\Omega_a}^2, \label{eq:A-coer}
\end{align}
for any $\UU=(\bm u, \bm w, \varphi), \VV=(\bm v, \bm z, \psi)\in \HH$.
\end{lemma}
\begin{proof}
Inequalities \eqref{eq:M-cont} and \eqref{eq:A-cont} are readily inferred by applying the Cauchy--Schwarz and triangle inequalities, while \eqref{eq:M-coer} is obtained by noting that $\rho \rho_w - \rho_f^2 > 0 $ and $\rho_a c^{-2} >0$.
The last inequality \eqref{eq:A-coer} represents the $\HH$-coercivity of $\mathcal{A}(\cdot,\cdot)+\mathcal{B}(\cdot,\cdot)$. To prove this property we apply Poincar\'e's and Korn's inequalities in $H^1_0(\Omega_a)$ and $\bm H^1_0(\Omega_p)$, respectively, to infer $\norm{\bm u}_{1,\Omega_p}^2 + \norm{\varphi}_{1,\Omega_a}^2\lesssim \mathcal{A}(\UU,\UU)$. Then, using the triangle inequality and recalling definition \eqref{eq:Wt_norm} of the $\bm{W}_{\kint}$-norm we get
$$
\norm{\bm w}_{\bm{W}_{\kint}}^2 \lesssim \norm{\nabla\cdot(\beta\bm{u}+\bm{w})}_{\Omega_p}^2 + \norm{\beta\nabla\cdot\bm{u}}_{\Omega_p}^2 + \mathcal{B}(\bm{w},\bm{w})\lesssim \mathcal{A}(\UU,\UU)+\mathcal{B}(\bm{w},\bm{w})
$$
and the conclusion follows.
\end{proof}
\begin{theorem}[Stability of the continuous weak formulation]\label{thm:stability}
Assume that the problem data satisfy $(\bm f_p, \bm g_p, \rho_a f_a)\in L^2((0,T);{\bm L}^2(\Omega_*))$, $\UU(0)=({\bm u}_0,{\bm w}_0,{\varphi}_0)\in\HH$, and $\dot{\UU}(0) = ({\bm u}_1,{\bm w}_1,{\varphi}_1)\in {\bm L}^2(\Omega_*)$. For any $t\in(0,T]$, let $\UU(t)=(\bm u,\bm w,\varphi)(t) \in \HH$ be the solution of \eqref{eq::weakform}. Then, it holds
\begin{equation*}
 \norm{\UU(t)}^2_{\mathcal{E}}\lesssim \norm{\UU(0)}_{\mathcal{E}}^2
+ \int_0^T \norm{(\bm f_p,\bm g_p, \rho_a f_a)(s)}_{\Omega_*}^2 ds,
\end{equation*}
with the hidden constant depending on the observation time $t\le T$ and on the material properties, but independent of $\kint$.
\end{theorem}
\begin{proof}
Taking $\dot{\UU} = (\dot{\bm u},\dot{\bm w},\dot{\varphi})$ as test functions in \eqref{eq::weakform}, using 
$\mathcal{C}^a(\dot{\bm u} +\dot{\bm w} ,\dot{\varphi})+\mathcal{C}^p(\dot{\varphi},\dot{\bm u} +\dot{\bm w})=0$, and integrating in time between $0$ and $t\le T$, it is inferred that
$$
\mathcal{M}(\dot{\UU},\dot{\UU})(t) + \mathcal{A}(\UU,\UU)(t)+
\int_0^t \hspace{-1mm} 2\mathcal{B}(\dot{\bm w},\dot{\bm w})\ ds
=\mathcal{M}(\dot{\UU},\dot{\UU})(0)+\mathcal{A}(\UU,\UU)(0)
+\int_0^t\hspace{-1mm} 2(\mathfrak{F},\dot{\UU})_{\Omega_*}\ ds,
$$
where we have adopted the abridged notation 
$\mathfrak{F} =  (\bm f_p,\bm g_p,\rho_a f_a)$. Hence, applying the Cauchy--Schwarz and Young inequalities to bound the third term in the right-hand side, using that $\mathcal{B}({\bm w},{\bm w})(t) \le \mathcal{B}({\bm w},{\bm w})(0)
+ \int_0^t \mathcal{B}(\dot{\bm w},\dot{\bm w})(s)\, ds$, and recalling definition \eqref{eq:energy} of the energy norm, for all $t\in (0,T]$ one has  
$$
\norm{\UU(t)}_{\mathcal{E}}^2 \lesssim \norm{\UU(0)}_{\mathcal{E}}^2
+ \int_0^t \norm{\mathfrak{F}(s)}_{\Omega_*}^2 \, ds + \int_0^t \norm{\dot{\UU}(s)}_{\Omega_*}^2 \, ds.
$$
Finally, owing to \eqref{eq:M-cont}, we obtain $\norm{\dot{\UU}}_{\Omega_*}^2\lesssim \norm{\UU}_{\mathcal{E}}^2$,
%using \eqref{eq:A-cont} and recalling the initial conditions ${\UU}(0)=({\bm u}_0,{\bm w}_0,{\varphi}_0)$, $\dot{\UU}(0)=({\bm u}_1,{\bm w}_1,{\varphi}_1)$, we obtain
%$$
%\norm{\UU(t)}_{\mathcal{E}}^2 \lesssim \norm{({\bm u}_1,{\bm w}_1,{\varphi}_1)}_{\Omega_*}^2 
%+\norm{({\bm u}_0,{\bm w}_0,{\varphi}_0)}_{\HH}^2 
%+ \int_0^t \norm{\mathfrak{F}(s)}_{\Omega_*}^2 \, ds + \int_0^t \norm{\UU(s)}_{\mathcal{E}}^2 \, ds,
%$$
so that the thesis follows by applying the Gronwall's Lemma \cite{quarteroni2014numerical}.
\end{proof}
}

%%%%%%%%%%%%%%%%%%%%%%
\section{The semi-discrete formulation and its stability analysis}\label{sec::numerical}

We introduce a \textit{polytopic} mesh $\mathcal{T}_h$ made of general polygons (in 2d) or polyhedra (in 3d) and write $\mathcal{T}_h$ as $\mathcal{T}_h=\mathcal{T}^p_h\cup\mathcal{T}^a_h$, where $\mathcal{T}^{\MB{\#}}_h=\{\kappa\in\mathcal{T}_h:\overline{\kappa}\subseteq\overline{\Omega}_{\MB{\#}}\}$, with $\MB{\#}=\{p,a\}$.
Implicit in this decomposition there is the assumption that the meshes $\mathcal{T}_h^a$ and $\mathcal{T}_h^p$ are aligned with $\Omega_a$ and $\Omega_p$, respectively. 
Polynomial degrees $p_{p,\kappa}\geq1$ and $p_{a,\kappa}\geq1$ are associated with each element of $\mathcal{T}_h^p$ and $\mathcal{T}_h^a$, respectively. 
The discrete spaces are introduced as follows: $\bm{V}_h^p=[\mathcal{P}_{p_p}(\mathcal{T}_h^p)]^d$ and $V_h^a=\mathcal{P}_{p_a}(\mathcal{T}_h^a)$, where \MB{$\mathcal{P}_{r}(\mathcal{T}^{\#}_h)$ is the space of piecewise polynomials in $\Omega_{\#}$ of degree less than or equal to $r$ in any $\kappa\in\mathcal{T}_h^{\#}$ with $\#=\{p,a\}$}.

In the following, we assume that $\mathbb{C}$, $\rho_a$ and $m$ are element-wise constant and we define 
$\overline{\mathbb{C}}_\kappa=(|\mathbb{C}^{1/2}|_2^2)_{|\kappa}$, $\overline{m}_{\kappa}=( m)_{|\kappa}$ for all $\kappa\in\mathcal{T}_h^p$ and $\overline{\rho}_{a,\kappa}=\rho_{a|\kappa}$ for all $\kappa\in\mathcal{T}_h^a$.
The symbol $|\cdot|_2$ stands for \MB{the $\ell^2$-norm on $\mathbb{R}^{n\times n}$, with $n=3$ if $d=2$ and $n=6$ if $d=3$.}
In order to deal with \myred{polygonal and polyhedral elements}, we define an \textit{interface}  as the intersection of the $(d-1)$-dimensional faces of  
any two neighboring elements of $\mathcal{T}_h$. If $d=2$, an interface/face is a line segment and the set of all interfaces/faces is  denoted by $\mathcal{F}_h$.
When $d=3$, an interface can be a general polygon that we assume could be further decomposed into a set of planar triangles collected in the set  $\mathcal{F}_h$.  
We decompose $\mathcal{F}_h$ as $\mathcal{F}_h=\mathcal{F}_{h}^I \cup \mathcal{F}_h^p \cup \mathcal{F}_h^a$, where
$ \mathcal{F}_{h}^I=\{F\in\mathcal{F}_h:F\subset\partial \kappa^p\cap\partial \kappa^a,\kappa^p\in\mathcal{T}_{h}^p,\kappa^a\in\mathcal{T}_{h}^a\}$, and 
$\mathcal{F}_h^p$ and $\mathcal{F}_h^a$ denote all the faces of $\mathcal{T}_h^p$ and $\mathcal{T}_h^a$, respectively, not laying on $\Gamma_I$.
Finally, the faces of $\mathcal{T}_h^p$ and $\mathcal{T}_h^a$ can be further written as the union of \textit{internal} ($i$) and \textit{boundary} ($b$) faces, respectively, i.e.:
$\mathcal{F}^p_h=\mathcal{F}^{p,i}_h\cup\mathcal{F}^{p,b}_h$ and 
$\mathcal{F}^a_h=\mathcal{F}^{a,i}_h\cup\mathcal{F}^{a,b}_h.$

Following \cite{CangianiDongGeorgoulisHouston_2017}, we next introduce the main assumption on $\mathcal{T}_h$. 
\begin{definition}\label{def::polytopic_regular}
A mesh $\mathcal{T}_h$ is said to be \textit{polytopic-regular} if for any $ \kappa \in \mathcal{T}_h$, there exists a set of non-overlapping $d$-dimensional simplices contained in $\kappa$, denoted by $\{S_\kappa^F\}_{F\subset{\partial \kappa}}$, such that for any face $F\subset\partial \kappa$, the following condition holds:
\begin{equation}
h_\kappa\lesssim d|S_\kappa^F| \, |F|^{-1}.
\label{eq::firstdef}
\end{equation}
\end{definition}
\begin{assumption}
The sequence of meshes $\{\mathcal{T}_h\}_h$ is assumed to be \textit{uniformly} polytopic regular in the sense of  Definition~\ref{def::polytopic_regular}.
\label{ass::regular}
\end{assumption}
\noindent
As pointed out in \cite{CangianiDongGeorgoulisHouston_2017}, this assumption does not impose any restriction on either the number of faces per element nor their measure relative to the diameter of the element they belong to.
Under Assumption~\ref{ass::regular}, the following \textit{trace-inverse inequality} holds:
\begin{align}
& ||v||_{L^2(\partial \kappa)}\lesssim ph_\kappa^{-1/2}||v||_{L^2(\kappa)}
&& \forall \ \kappa\in\mathcal{T}_h \ \forall v \in \mathcal{P}_p(\kappa).
\label{eq::traceinv}
\end{align}
\MB{In order to avoid technicalities}, we also make the following assumption.
\begin{assumption}
For any pair of neighboring elements $\kappa^\pm\in\mathcal{T}_h$. The following \textit{hp-local bounded variation property} holds: $h_{\kappa^+}\lesssim h_{\kappa^-}\lesssim h_{\kappa^+},\ \ p_{\kappa^+}\lesssim p_{\kappa^-}\lesssim p_{\kappa^+}$.
\label{ass::3}
\end{assumption}

Finally, following  \cite{Arnoldbrezzicockburnmarini2002}, for sufficiently piecewise smooth scalar-, vector- and tensor-valued fields $\psi$, $\bm{v}$ and $\bm{\tau}$, respectively, we define the averages and jumps on each \textit{interior} face $F\in\mathcal{F}_h^{p,i}\cup\mathcal{F}_h^{a,i}\cup\mathcal{F}_h^{I}$ shared by the elements $\kappa^{\pm}\in \mathcal{T}_h^p$ as follows:
\begin{align*}
 \llbracket\psi\rrbracket  &=  \psi^+\bm{n}^++\psi^-\bm{n}^-, 
 &&\llbracket\bm{v}\rrbracket  = \bm{v}^+\otimes\bm{n}^++\bm{v}^-\otimes\bm{n}^-, 
 &&\MB{\llbracket\bm{v}\rrbracket_{\bm n}  = \bm{v}^+\cdot\bm{n}^++\bm{v}^-\cdot\bm{n}^-}, \\
 \llbrace\psi \rrbrace &= \frac{\psi^++\psi^-}{2}, 
&&\llbrace\bm{v}\rrbrace   = \frac{\bm{v}^++\bm{v}^-}{2}, 
&&\llbrace\bm{\tau}\rrbrace  = \frac{\bm{\tau}^++\bm{\tau}^-}{2}, 
\end{align*}
where $\otimes$ is the tensor product in $\mathbb{R}^3$, $\cdot^{\pm}$ denotes the trace on $F$ taken within $\kappa^\pm$, and $\bm{n}^\pm$ is the outer normal vector to $\partial \kappa^\pm$. Accordingly, on \textit{boundary} faces $F\in\mathcal{F}_h^{p,b}\cup\mathcal{F}_h^{a,b}$, we set
$
\llbracket\psi\rrbracket = \psi\bm{n},\
\llbrace\psi \rrbrace = \psi,\
\llbracket\bm{v}\rrbracket= \bm{v}\otimes\bm{n},\
\MB{\llbracket\bm{v}\rrbracket_{\bm n} =\bm{v}\cdot\bm{n}},\
\llbrace\bm{v}\rrbrace= \bm{v},\
\llbrace\bm{\tau}\rrbrace= \bm{\tau}.$

\subsection{Semi-discrete PolyDG formulation}
We are now ready to introduce the semi-discrete formulation: for $t\in(0,T]$, find $(\bm{u}_h,\bm{w}_h,\varphi_h)(t)\in \bm{V}_h^p\times \bm{V}_h^p\times V_h^a$, s.t.
\begin{multline}
\MB{\mathcal{M}((\ddot{\bm{u}_h},\ddot{\bm{w}_h},\ddot{\varphi}_h), (\bm{v}_h,\bm{z}_h,\psi_h)) +
\mathcal{A}_h((\bm{u}_h,\bm{w}_h,\varphi_h),(\bm{v}_h,\bm{z}_h,\psi_h)) +
\mathcal{B} (\dot{\bm{w}}_h,{\bm z}_h)} \\
+\mathcal{C}_h^p(\dot{\varphi_h},\bm{v}_h+\bm{z}_h)
+\mathcal{C}_h^a(\dot{\bm{u}_h}+\dot{\bm{w}_h},\psi_h)
= ((\bm{f}_p, \bm{g}_p, \rho_a f_a), (\bm{v}_h,\bm{\xi}_h,\psi_h))_{\Omega_*}\label{eq::dgsystem}
\end{multline}
for all $(\bm{v}_h,\bm{\xi}_h,\psi_h)\in \bm{V}_h^p\times \bm{V}_h^p\times V_h^a$. 
\MB{As initial conditions we take the $L^2$-orthogonal projections onto $(\bm{V}_h^p\times \bm{V}_h^p\times V_h^a)^2$ of the initial data $(\bm{u}_0,\bm{w}_0,\varphi_0,\bm{u}_1, \bm{w}_1,\varphi_1)$}.
We define $\nabla_h$ and $\nabla_h \cdot$ to be the broken \MB{gradient} and divergence operators, respectively, 
set $\bm{\epsilon}_h(\bm{v})=\frac{\nabla_h \bm{v} + \nabla_h \bm{v}^T}{2}$, $\bm{\sigma}_h(\bm{v})=\mathbb{C}:\bm{\epsilon}_h(\bm{v})$, and use the short-hand notation  
\MB{$(\cdot,\cdot)_{\Omega_{\#}}=\sum_{\kappa\in\mathcal{T}_h^{\#}} \int_\kappa\cdot$ and $\langle\cdot,\cdot\rangle_{\mathcal{F}_h^{\#}}=\sum_{F\in\mathcal{F}_h^{\#}}\int_F\cdot$ for $\# = \{a,p\}$.} 
Then, for all $\bm{u},\bm{v},\bm{w},\bm{z} \in\bm{V}_h^p$ and  $ \varphi,\psi \in V_h^a$, the bilinear forms appearing in the above formulation are given by
\MB{\begin{align}
\mathcal{A}_h((\bm{u},\bm{v},\varphi),(\bm{v},\bm{z},\psi))  & = \mathcal{A}_h^e(\bm{u},\bm{v})+
\mathcal{A}_h^p(\beta\bm{u}+\bm w,\beta\bm{v}+\bm{z}) +  \mathcal{A}_h^a(\varphi,\psi), \label{eq:bilinear_Ah}\\
\mathcal{C}_h^p(\varphi,\bm{v}) &=
\langle \rho_a\varphi,\bm{v}\cdot\bm{n}_p\rangle_{\mathcal{F}_{h}^{I}}= -\mathcal{C}_h^a(\bm{v},\varphi),
\label{eq::bilineardg2}
\end{align}}
with
\begin{equation}
\label{eq::DGbilinearforms}
\begin{aligned}
\mathcal{A}_h^e(\bm{u},\bm{v})
& = (\bm{\sigma}_h(\bm{u}),\bm{\epsilon}_h(\bm{v}))_{\Omega_p}-\langle\llbrace\bm{\sigma}_h(\bm{u})\rrbrace,\llbracket\bm{v}\rrbracket
\rangle_{\mathcal{F}_h^p} \\ \nonumber
& \qquad \qquad \qquad \qquad \qquad \qquad
- \langle\llbracket\bm{u}\rrbracket, \llbrace\bm{\sigma}_h(\bm{v})\rrbrace \rangle_{\mathcal{F}_h^p}+ 
\langle\alpha\llbracket\bm{u}\rrbracket,\llbracket\bm{v}\rrbracket\rangle_{\mathcal{F}_h^p},\\
%\label{eq::dgbil1}
\mathcal{A}_h^p(\bm{w},\bm{z})&=
( m\nabla_h\cdot\bm{w},\nabla_h\cdot\bm{z})_{{\Omega}_p}
\MB{-\langle\llbrace m(\nabla_h\cdot\bm{w})\rrbrace, \llbracket\bm{z}\rrbracket_{\bm n}\rangle_{\mathcal{F}_h^{\star}}}
\\ \nonumber
& \qquad \qquad \qquad \qquad \qquad \qquad - \MB{\langle\llbracket\bm{w}\rrbracket_{\bm n},\llbrace m(\nabla_h\cdot\bm{z})\rrbrace\rangle_{\mathcal{F}_h^\star}+\langle \gamma \llbracket\bm{w}\rrbracket_{\bm n}, \llbracket\bm{z}\rrbracket_{\bm n}\rangle_{\mathcal{F}_h^{\star}},}\\
%\label{eq::dgbil2}
\mathcal{A}_h^a(\varphi,\psi) & = (\rho_a\nabla_h \varphi,\nabla_h \psi)_{\Omega_a}
-\langle\llbrace\rho_a\nabla_h \varphi\rrbrace,\llbracket\psi\rrbracket\rangle_{\mathcal{F}_h^a}
\\ \nonumber
& \qquad \qquad \qquad \qquad \qquad \qquad
- \langle\llbracket\varphi\rrbracket, \llbrace\rho_a\nabla_h \psi\rrbrace \rangle_{\mathcal{F}_h^a}
+ \langle \chi \llbracket\varphi\rrbracket,\llbracket\psi\rrbracket\rangle_{\mathcal{F}_h^a},
%\label{eq::dgbil3}
\end{aligned}
\end{equation}
\MB{and $\mathcal{F}_h^\star = \mathcal{F}_h^p$ in the case $\kint\in(0,1]$, while $\mathcal{F}_h^\star = \mathcal{F}_h^p\cup\mathcal{F}_h^I$ in the case $\kint=0$}.
The stabilization functions $\alpha\in L^\infty(\mathcal{F}_h^p)$, $\gamma\in L^\infty(\mathcal{F}_h^p)$ and $\chi\in L^\infty(\mathcal{F}_h^a)$, are defined s.t.
\begin{align}
&
\alpha|_F=
\begin{cases}
c_1 \max\limits_{\kappa\in\{\kappa^+,\kappa^-\}}\left(\overline{\mathbb{C}}_\kappa\ p_{p,\kappa}^2{h_\kappa^{-1}}\right) \hspace{7,5mm}&\forall F\in\mathcal{F}_h^{p,i},\hspace{8,5mm}F\subseteq\partial \kappa^+\cap\partial \kappa^-, \vspace{0.1cm}\label{eq::stab_1}\\
\overline{\mathbb{C}}_\kappa \ p_{p,\kappa}^2{h_\kappa^{-1}}&\forall F\in\mathcal{F}_h^{p,b},\hspace{8,5mm}F\subseteq\partial \kappa,
\end{cases}\\ \nonumber\\
&
\gamma|_F=
\begin{cases}
c_2 \max\limits_{\kappa\in\{\kappa^+,\kappa^-\}}\left(\overline{ m}_\kappa \ p_{p,\kappa}^2{h_\kappa^{-1}}\right) \hspace{7,5mm}&\forall F\in\mathcal{F}_h^{p,i},\hspace{8,5mm}F\subseteq\partial \kappa^+\cap\partial \kappa^-, \vspace{0.1cm}\label{eq::stab_2}\\
\overline{ m}_\kappa \ p_{p,\kappa}^2{h_\kappa^{-1}}&\forall F\in\MB{\mathcal{F}_h^{p,b}\cup\mathcal{F}_h^{I}},\hspace{0,5mm} F\subseteq\partial \kappa,
\end{cases} \\ \nonumber\\
&
\chi|_F=
\begin{cases}
c_3 \max\limits_{\kappa\in\{\kappa^+,\kappa^-\}}\left(\overline{\rho}_{a,\kappa}\ p_{a,\kappa}^2{h_\kappa^{-1}}\right) \hspace{7,5mm}&\forall F\in\mathcal{F}_h^{a,i},\hspace{8,5mm}F\subseteq\partial \kappa^+\cap\partial \kappa^-, \vspace{0.1cm}\label{eq::stab_3}\\
\overline{\rho}_{a,\kappa}\  p_{a,\kappa}^2{h_\kappa^{-1}}&\forall F\in\mathcal{F}_h^{a,b},\hspace{8,5mm}F\subseteq\partial \kappa,
\end{cases}
\end{align}
with $c_1,\ c_2,\ c_3>0$ positive constants, to be properly chosen.
The definition of the penalty functions \eqref{eq::stab_1}--\eqref{eq::stab_3} is based on \cite[Lemma 35]{CangianiDongGeorgoulisHouston_2017}. With this choice, the bilinear forms in \eqref{eq::DGbilinearforms} are symmetric and coercive, cf. Lemma \ref{lem::cont}.
Alternative stabilization functions can be defined in the spirit of \cite{M2AN_2013__47_3_903_0}. The analysis of the latter is however beyond the scope of this work. See also \cite{sarkis2010} for the elliptic case.

By fixing a basis for $\bm V_h^p$ and  $V_h^a$ and denoting by ($U$, $W$,$\Phi$) the vector of the expansion coefficients in the chosen basis of the unknowns $\bm{u}_h$,  $\bm{w}_h$ and $\varphi_h$, respectively, the semi-discrete formulation \eqref{eq::dgsystem} can be written equivalently as:
\begin{multline}\label{eq::algebraic}
\left[ \begin{matrix} 
\MB{\rho\bm{M}^p} & \MB{\rho_f\bm{M}^p} & 0 \\
\MB{\rho_f\bm{M}^p} & \MB{\rho_w\bm{M}^p} & 0 \\
0 & 0 & \MB{\rho_a c^{-2}\bm{M}^a}
\end{matrix} \right]  
\left[ \begin{matrix} 
\ddot{U} \\
\ddot{W} \\ 
\ddot{\Phi}
\end{matrix} \right] + 
\left[ \begin{matrix} 
0 & 0 & \bm{C}^p \\
0 & \MB{\bm{B}} & \bm{C}^p \\
\bm{C}^a & \bm{C}^a & 0
\end{matrix} \right]  
\left[ \begin{matrix} 
\dot{U} \\
\dot{W} \\ 
\dot{\Phi}
\end{matrix} \right] \\
+\left[ \MB{\begin{matrix} 
\bm{A}^e +  \beta^2\bm{A}^p &  \beta\bm{A}^p & 0 \\
\beta\bm{A}^p & \bm{A}^p & 0 \\ 
0 & 0 & \bm{A}^a
\end{matrix}} \right] 
\left[ \begin{matrix} 
{U} \\
{W} \\ 
{\Phi}
\end{matrix} \right] = 
\left[ \begin{matrix} 
\bm{F}^p \\
\bm{G}^p \\ 
\bm{F}^a
\end{matrix} \right] 
\end{multline}
with initial conditions $U(0)=U_0$,  $W(0)=W_0$, $\Phi(0)=\Phi_0$, $\dot{U}(0)=U_1$, $\dot{W}(0)=W_1$, $\dot{\Phi}(0)=\Phi_1$.
We remark that $\bm{F}^p$, $\bm{G}^p$ and $\bm{F}^a$ are the vector representations of the linear functionals $(\bm{f}_{p},\bm{v}_h)_{\Omega_p}$,
$(\bm{g}_{p},\bm{\xi}_h)_{\Omega_p}$ and $( \rho_a f_{a},\psi_h)_{\Omega_a}$, respectively. 

\subsection{Stability analysis}
To carry out the stability analysis of the semi-discrete problem, we introduce the energy norm
\MB{\begin{multline}
\label{eq::norm_def1}
\norm{(\bm v,\bm z,\psi)(t)}_{\rm E}^2=
\mathcal{M}((\dot{\bm v},\dot{\bm z},\dot{\psi}), (\dot{\bm v},\dot{\bm z},\dot{\psi}))(t)+ \mathcal{B}({\bm z},{\bm z}) (t) \\+
\norm{\bm v (t)}_{\rm dG,e}^2 + |(\beta\bm v+\bm z)(t)|_{\rm dG,p}^2 + \norm{\psi(t)}_{\rm dG,a}^2 
\end{multline}
for all  $(\bm v,\bm z,\psi) \in C^1([0,T];\bm{V}_h^p\times \bm{V}_h^p\times V_h^a)$, where
\begin{align*}
\norm{\bm{v}}_{\rm dG,e}^2 & =\norm{\mathbb{C}^{1/2}:\bm{\epsilon}_h(\bm{v})}_{\Omega_p}^2+\norm{\alpha^{1/2}\llbracket\bm{v}\rrbracket}_{\mathcal{F}_h^p}^2  & \forall \bm v\in \bm V_h^p, \\
|\bm{z}|_{\rm dG,p}^2 &=\norm{{m}^{1/2}\nabla_h\cdot\bm z}_{\Omega_p}^2+
\norm{\gamma^{1/2}\llbracket\bm{z}\rrbracket_{\bm n}}_{\MB{\mathcal{F}_h^\star}}^2 & \forall \bm z\in \bm V_h^p,\\
\norm{\psi}_{\rm dG,a}^2 &= \norm{{\rho_a}^{1/2}\nabla_h\psi}_{\Omega_a}^2+\norm{\chi^{1/2}\llbracket\psi\rrbracket}_{\mathcal{F}_h^a}^2 & \forall \psi\in V_h^a.
\end{align*}
\begin{remark}
The notation $|\cdot|_{\rm dG,p}$ is used instead of $\norm{\cdot}_{\rm dG,p}$ in order to highlight that $|\cdot|_{\rm dG,p}:\bm{V}_h^p\to\mathbb{R}^+$ is a seminorm. However, by proceeding as in the proof of \eqref{eq:A-coer}, we can show that $\norm{\bm v}_{\rm dG,e}^2+|\beta\bm v+\bm z|_{\rm dG,p}^2+\mathcal{B}({\bm z},{\bm z})$ is a norm on $\bm{V}_h^p\times\bm{V}_h^p$. 
\end{remark}}
\begin{remark}
Notice that the norm defined in \eqref{eq::norm_def1} represents the mechanical energy of the poroelasto-acoustic system.
We observe that in the case of null external forces, i.e., ${\bm f_p} = {\bm g_p} = {\bm 0}$ and $f_a = 0$, estimate \eqref{theo:energy_system} reduces to $\norm{(\bm u_h,\bm w_h,\varphi_h)(t)}_{\rm{E}}\lesssim \norm{(\bm u_h,\bm w_h,\varphi_h)(0)}_{\rm{E}}$ for any $t>0$, namely the dG formulation \eqref{eq::dgsystem} is dissipative.
\end{remark}
The main \rred{stability} result is stated in the following theorem.
%%%%
\begin{theorem}[Stability of the semi-discrete formulation]\label{teo:stability}
Let  Assumptions~\ref{ass::regular} and \ref{ass::3} be satisfied. For sufficiently large penalty parameters $c_1$, $c_2$ and $c_3$ in \eqref{eq::stab_1}, \eqref{eq::stab_2} and \eqref{eq::stab_3}, respectively, let $(\bm u_h,\bm w_h,\varphi_h)(t)\in \bm{V}_h^p\times \bm{V}_h^p\times V_h^a$ be the solution of \eqref{eq::dgsystem} for any $t\in(0,T]$. Then, it holds
\MB{\begin{equation}\label{theo:energy_system}
\norm{(\bm u_h,\bm w_h,\varphi_h)(t)}_{\rm{E}}\lesssim\norm{(\bm u_h,\bm w_h,\varphi_h)(0)}_{\rm{E}} 
+\int_0^t \norm{(\bm f_p,\bm g_p,\rho_a f_a)(s)}_{\Omega_*}^2 \, ds,
\end{equation}
where the hidden constant depends on time $t$ and on the material properties, but is independent of $\kint$.}
\end{theorem}
\begin{proof}
By taking $(\bm{v}_h,\bm{z}_h,\psi_h)=(\dot{\bm{u}}_h,\dot{\bm{w}}_h,\dot{\varphi}_h) \in \bm{V}_h^p\times \bm{V}_h^p\times V_h^a$ in \eqref{eq::dgsystem} and using the skew-symmetry of the coupling bilinear forms \eqref{eq::bilineardg2}, we obtain
\begin{multline*}
\frac{1}{2}\frac{d}{dt}\bigg[
\MB{\mathcal{M}((\dot{\bm{u}_h},\dot{\bm{w}_h},\dot{\varphi}_h), (\dot{\bm{u}_h},\dot{\bm{w}_h},\dot{\varphi}_h))  +
\mathcal{A}_h((\bm{u}_h,\bm{v}_h,\varphi_h),(\bm{u}_h,\bm{v}_h,\varphi_h))}\bigg] \\
+ \MB{\mathcal{B}(\dot{\bm w}_h, \dot{\bm w}_h)} =
((\bm{f}_p, \bm{g}_p, \rho_a f_a), (\dot{\bm{u}}_h,\dot{\bm{z}}_h,\dot{\varphi}_h))_{\Omega_*}. 
%\label{eq::stab_semi_previous}
\end{multline*}
\MB{Thus, integrating in time between $0$ and $t\le T$, recalling definition \eqref{eq:bilinear_Ah} of $\mathcal{A}_h$, using the coercivity results of Lemma~\ref{lem::cont}, and reasoning as in the proof of Theorem~\ref{thm:stability}, one can easily obtain the thesis.}
\end{proof}

\section{Error analysis for the semi-discrete formulation}\label{sec::errors}
In this section we prove an a-priori error estimate for the semi-discrete problem \eqref{eq::dgsystem}. 
We first observe that by setting, for any \MB{time $t\in (0,T]$}, $\bm e^u(t)=(\bm u-\bm u_h)(t)$, $\bm e^w(t)=(\bm w-\bm w_h)(t)$, and $e^\varphi(t)=(\varphi-\varphi_h)(t)$ and by using the \textit{strong consistency} of the semi-discrete formulation \eqref{eq::dgsystem}, 
the \textit{error equation} reads as follows
\RRR{\begin{multline}
\mathcal{M}((\ddot{\bm{e}}^{u},\ddot{\bm{e}}^{w},\ddot{e}^\varphi), (\bm{v},\bm{z},\psi))  +
\mathcal{A}_h((\bm{e}^{u},\bm{e}^{w}, e^\varphi),(\bm{v},\bm{z},\psi)) + \mathcal{B}(\dot{\bm{e}}^w,\bm{z}) \\
+\mathcal{C}_h^p(\dot{e}^\varphi,\bm{v} + \bm{z})
+\mathcal{C}_h^a(\dot{\bm{e}}^u + \dot{\bm{e}}^w,{\psi})=0
\label{eq::error_eq} 
\end{multline}
for any $(\bm{v},\bm{z},\psi) \in \bm{V}_h^p\times \bm{V}_h^p\times V_h^a$.}
Next, we introduce the following definition and a further mesh assumption; cf \cite{cangiani2014hp,CangianiDongGeorgoulisHouston_2017}.
\begin{definition}\label{def:covering}
A \textit{covering} \RRR{$\mathcal{T}_{\S}=\{\mathcal{K}\}$} of the polytopic mesh $\mathcal{T}_h$ is a set of regular shaped $d$-dimensional simplices $\mathcal{K}$, $d=2,3$, s.t. $\forall \ \kappa\in\mathcal{T}_h$, $\exists \ \mathcal{K}\in  \RRR{\mathcal{T}_{\S}}$ s.t. $\kappa \subseteq \mathcal{K}$.
\end{definition}
\begin{assumption}
Any mesh $\mathcal{T}_h$ admits a covering $\mathcal{T}_{\S}$ in the sense of \eqref{def:covering} such that \\
i) $\max_{\kappa\in\mathcal{T}_h} \textrm{card}\{\kappa^\prime\in\mathcal{T}_h:\kappa^\prime\cap\mathcal{K}\neq\emptyset,\ \mathcal{K}\in\MB{\mathcal{T}_{\S}} \text{ s.t.} \ \kappa\subset\mathcal{K}\}\lesssim 1$ and 
ii) $h_\mathcal{K}\lesssim h_\kappa $ for each pair $\kappa\in\mathcal{T}_h, \ \mathcal{K}\in\MB{\mathcal{T}_{\S}}$ with $\kappa\subset\mathcal{K}$.
\label{ass::2}
\end{assumption}
We also introduce \MB{the norm
\begin{equation}\label{eq:trinorm_tot}
\trinorm{(\bm{v},\bm z, \psi)}_{\rm E}^2 = 
\mathcal{M}((\dot{\bm v},\dot{\bm z},\dot{\psi}), (\dot{\bm v},\dot{\bm z},\dot{\psi})) 
+ \trinorm{(\bm{v},\bm z, \psi)}_{\rm dG}^2 + \mathcal{B}({\bm z},{\bm z}),
\end{equation}
where the seminorm $\trinorm{(\bm{v},\bm z, \psi)}_{\rm dG}^2 = \trinorm{\bm{v}}_{\rm dG,e}^2 + 
\trinorm{\bm{z}}_{\rm dG,p}^2 + \trinorm{\psi}_{\rm dG,a}^2$ is defined by
\begin{align*}
\trinorm{\bm{v}}_{\rm dG,e}^2 &=\norm{\bm{v}}_{\rm dG,e}^2+\norm{\alpha^{-1/2}\llbrace\mathbb{C}:\bm{\epsilon}_h(\bm{v})\rrbrace}_{\mathcal{F}_h^p}^2
&& \forall \bm v\in \bm H^2(\mathcal{T}_h^p),\\
\trinorm{\bm{z}}_{\rm dG,p}^2 &= \MB{|\bm{z}|_{\rm dG,p}^2+
\norm{\gamma^{-1/2}\llbrace(m \nabla_h\cdot\bm z)\rrbrace}_{\mathcal{F}_h^\star}^2}
&& \forall \bm z\in \bm H^2(\mathcal{T}_h^p),\\
\trinorm{\psi}_{\rm dG,a}^2 &=\norm{\psi}_{\rm dG,a}^2+\norm{\chi^{-1/2}\llbrace\rho_a\nabla_h\psi\rrbrace}_{\mathcal{F}_h^a}^2
&& \forall \psi\in H^2(\mathcal{T}_h^a).
\end{align*}
}
For an open bounded polytopic domain $\Sigma\subset\mathbb{R}^d$ and a generic polytopic mesh $\mathcal{T}_h$ over $\Sigma$ satisfying Assumption~\ref{ass::2},
as in  \cite{cangiani2014hp}, we can introduce the Stein extension operator $\tilde{\mathcal{E}}:H^m(\kappa)\rightarrow H^m(\mathbb{R}^d)$ \cite{stein1970singular}, for any $\kappa\in\mathcal{T}_h$ and $m\in\mathbb{N}_0$, such that $\tilde{\mathcal{E}}v|_\kappa=v$ and $\norm{\tilde{\mathcal{E}}v}_{m,\mathbb{R}^d}\lesssim\norm{v}_{m,\kappa}$. The corresponding vector-valued version mapping $\bm H^m(\kappa)$ onto $\bm H^m(\mathbb{R}^d)$ acts component-wise and is denoted in the same way. 
\MB{In what follows, for any $\kappa\in\mathcal{T}_h$, we will denote by $\mathcal{K}_\kappa$ the simplex belonging to $\mathcal{T}_{\S}$ such that $\kappa\subset\mathcal{K}_\kappa$.}

\MB{In order to handle the case of small interface permeability, i.e. $0<\kint<<1$, we make an additional assumption on the discretization. This requirement is consistent with the observations of \cite{chiavassa_lombard_2013}, showing that there is a threshold value $\overline{\kint}$ such that the results for $\kint\le\overline{\kint}$ cannot be distinguished from the sealed pores case $\kint=0$. 
\begin{assumption}
In the case $\kint\in(0,1)$, for each $F\in\mathcal{F}_h^I$ and $\kappa\in\mathcal{T}_h^p$ such that $F\subset\partial\kappa\cap\Gamma_I$, it holds $\zeta(\kint)=\kint^{-1}(1-\kint) \lesssim h_{\kappa}^{-1} p_{p,\kappa}^2$, with the hidden constant independent of $\kint$.
\label{ass::kint_bnd}\RRR{We point out that this assumption is used only for the following theoretical results but it is not needed in practice, cf. Section \ref{sec::results}.}
\end{assumption}
The next Lemma provides the interpolation bounds that are instrumental for the derivation of the a-priori error estimate.}
%%%
\begin{lemma}
For any $(\bm v,\bm z,\psi) \in {\bm H}^m(\mathcal{T}_h^p)\times {\bm H}^\ell(\mathcal{T}_h^p) \times H^n(\mathcal{T}_h^a)$, with $m,\ell,n\geq 2$, there exists $(\bm v_I,\bm z_I,\psi_I)\in\bm{V}_h^p\times\bm{V}_h^p\times V_h^a$ such that
\begin{align*}
\trinorm{\bm{v}-\bm{v}_I}_{\rm dG,e}^2 & 
\lesssim
\sum_{\kappa\in\mathcal{T}_h^p}
{\frac{h_\kappa^{2(s_\kappa-1)}}{p_{p,\kappa}^{2m-3}}}\norm{\widetilde{\mathcal{E}}\bm v}_{m,\mathcal{K}_\kappa}^2, \\
\trinorm{\bm{z}-\bm{z}_I}_{\rm dG,p}^2 & 
\lesssim
\sum_{\kappa\in\mathcal{T}_h^p}
{\frac{h_\kappa^{2(r_\kappa-1)}}{p_{p,\kappa}^{2\ell-3}}}\norm{\widetilde{\mathcal{E}}\bm z}_{\ell,\mathcal{K}_\kappa}^2,\\
\trinorm{\psi-\psi_I}_{\rm dG,a}^2 & 
\lesssim
\sum_{\kappa\in\mathcal{T}_h^a}
{\frac{h_\kappa^{2(q_\kappa-1)}}{p_{a,\kappa}^{2n-3}}}\norm{\widetilde{\mathcal{E}}\psi}_{n,\mathcal{K}_\kappa}^2,
\end{align*}
where $s_\kappa = \min(m,p_{p,\kappa}+1)$, $r_\kappa = \min(\ell,p_{p,\kappa}+1)$ and $q_\kappa = \min(n,p_{a,\kappa}+1)$. Moreover, if $(\bm u,\bm w,\varphi) \in  C^1([0,T];\, \bm H^m(\mathcal{T}_h^p)\times \bm H^\ell(\mathcal{T}_h^p)\times H^n(\mathcal{T}_h^a))$, with $m,\ell,n\geq 2$, there exists \MB{$(\bm u_I,\bm w_I,\varphi_I)\in C^1([0,T];\bm{V}_h^p\times \bm{V}_h^p\times V_h^a)$} s.t.:
\begin{equation}\label{eq:interp_est}
\begin{aligned}
\MB{\trinorm{(\bm u-\bm u_I,\bm w-\bm w_I,\varphi-\varphi_I)}_{\rm E}^2}\lesssim&
\sum_{\kappa\in\mathcal{T}_h^p}{\frac{h_\kappa^{2(s_\kappa-1)}}{p_{p,\kappa}^{2m-3}}}\left(\norm{\widetilde{\mathcal{E}}\dot{\bm u}}_{m,\mathcal{K}_\kappa}^2+\norm{\widetilde{\mathcal{E}}\bm u}_{m,\mathcal{K}_\kappa}^2\right)\\+&
\sum_{\kappa\in\mathcal{T}_h^p}{\frac{h_\kappa^{2(r_\kappa-1)}}{p_{p,\kappa}^{2\ell-3}}}\left(\norm{\widetilde{\mathcal{E}}\dot{\bm w}}_{\ell,\mathcal{K}_\kappa}^2+\norm{\widetilde{\mathcal{E}}\bm w}_{\ell,\mathcal{K}_\kappa}^2\right)\\+&
\sum_{\kappa\in\mathcal{T}_h^a}{\frac{h_\kappa^{2(q_\kappa-1)}}{p_{a,\kappa}^{2n-3}}}\left(\norm{\widetilde{\mathcal{E}}\dot{\varphi}}_{n,\mathcal{K}_\kappa}^2+\norm{\widetilde{\mathcal{E}}\varphi}_{n,\mathcal{K}_\kappa}^2\right).
\end{aligned}
\end{equation}
\label{lem::interp2}
\end{lemma}
\begin{proof}
\MB{The first part of the proof readily follows by reasoning as in \cite[Lemma 5.1]{bonaldi} and observing that 
$\trinorm{\cdot}_{\rm dG,p}\lesssim\trinorm{\cdot}_{\rm dG,e}$. 
To infer estimate \eqref{eq:interp_est}, we resort to the $hp$-approximation properties stated in \cite[Lemmas 23 and 33]{CangianiDongGeorgoulisHouston_2017}, implying 
\begin{multline*}
\mathcal{M}((\dot{\bm u}-\dot{\bm u}_I,\dot{\bm w}-\dot{\bm w}_I,\dot{\varphi}-\dot{\varphi}_I),(\dot{\bm u}-\dot{\bm u}_I,\dot{\bm w}-\dot{\bm w}_I,\dot{\varphi}-\dot{\varphi}_I))
\\
\lesssim  \sum_{\kappa\in\mathcal{T}_h^p} \left(
{\frac{h_\kappa^{2s_\kappa}}{p_{p,\kappa}^{2m}}}\norm{\widetilde{\mathcal{E}}\dot{\bm u}}_{m,\mathcal{K}_\kappa}^2 
+ {\frac{h_\kappa^{2r_\kappa}}{p_{p,\kappa}^{2\ell}}}\norm{\widetilde{\mathcal{E}}\dot{\bm w}}_{\ell,\mathcal{K}_\kappa}^2\right)
+\sum_{\kappa\in\mathcal{T}_h^a}{\frac{h_\kappa^{2q_\kappa}}{p_{a,\kappa}^{2n}}}\norm{\widetilde{\mathcal{E}}\dot{\varphi}}_{n,\mathcal{K}_\kappa}^2, 
\end{multline*}
and, owing to \eqref{ass::kint_bnd},
$$
\mathcal{B}(\bm{w}-\bm{w}_I,\bm{w}-\bm{w}_I) \lesssim
\sum_{\kappa_p\in \mathcal{T}_{h,p}^I} \frac{p_{p,\kappa_p}^2}{h_{\kappa_p}} \norm{(\bm{w}-\bm{w}_I)\cdot\bm n}_{\partial\kappa_p}^2
\lesssim
\sum_{\kappa\in\mathcal{T}_h^p}{\frac{h_\kappa^{2r_\kappa-2}}{p_{p,\kappa}^{2\ell-3}}}\norm{\widetilde{\mathcal{E}}\bm w}_{\ell,\mathcal{K}_\kappa}^2.
$$}
\end{proof}
We are now ready to state the main result of this section.
\begin{theorem}[A-priori error estimates] \label{thm::error-estimate}
Let Assumptions~\ref{ass::regular},  \ref{ass::3}, \ref{ass::2}, and \ref{ass::kint_bnd} hold and let the exact solution \RRR{$\mathfrak{U}=(\bm u,\bm w,\varphi)$} of problem \eqref{system} be such that $$\mathfrak{U}\in C^2([0,T];\bm H^m(\mathcal{T}_h^p)\times\bm H^\ell(\mathcal{T}_h^p))\times H^n(\mathcal{T}_h^a)) \MB{\,\cap\, C^1([0,T]; \bm{H}^1_0(\Omega_p)\times\bm W_{\kint}\times H^1_0(\Omega_a)),}$$ with $m,n,\ell\geq 2$ and let $(\bm u_h,\bm w_h,\varphi_h)\in C^2([0,T];\bm V_h^p\times \bm V_h^p\times V_h^a)$ be the solution of the semi-discrete problem \eqref{eq::dgsystem}, with sufficiently large penalty parameters $c_1$, $c_2$ and $c_3$. 
Then, \MB{for any $t\in (0,t]$, the discretization error $\bm E(t)=(\bm e^u, \bm e^w, e^\varphi)(t)$ satisfies}
\begin{align*}
\norm{\bm E(t)}_{\rm E}\lesssim&
\sum_{\kappa\in\mathcal{T}_h^p}{\frac{h_\kappa^{s_\kappa-1}}{p_{p,\kappa}^{m-3/2}}}
\left(
\norm{\widetilde{\mathcal{E}}\dot{\bm u}}_{m,\mathcal{K}_\kappa}\hspace{-1mm}+
\norm{\widetilde{\mathcal{E}}{\bm u}}_{m,\mathcal{K}_\kappa}\hspace{-1mm}+\hspace{-0.5mm}
\int_{0}^{t}\hspace{-0.5mm}{\left[\norm{\widetilde{\mathcal{E}}\ddot{\bm u}}_{m,\mathcal{K}_\kappa}+
\norm{\widetilde{\mathcal{E}}\dot{\bm u}}_{m,\mathcal{K}_\kappa}\right]\hspace{-1mm}(s)\, ds}
\hspace{-0.5mm}\right)\nonumber\\+&
\sum_{\kappa\in\mathcal{T}_h^p}{\frac{h_\kappa^{r_\kappa-1}}{p_{p,\kappa}^{\ell-3/2}}}
\left(
\norm{\widetilde{\mathcal{E}}\dot{\bm w}}_{\ell,\mathcal{K}_\kappa}+
\norm{\widetilde{\mathcal{E}}{\bm w}}_{\ell,\mathcal{K}_\kappa}+
\int_{0}^{t}{\left[\norm{\widetilde{\mathcal{E}}\ddot{\bm w}}_{\ell,\mathcal{K}_\kappa}+
\norm{\widetilde{\mathcal{E}}\dot{\bm w}}_{\ell,\mathcal{K}_\kappa}\right]\hspace{-1mm}(s)\, ds}
\hspace{-0.5mm}\right)\nonumber\\+&
\sum_{\kappa\in\mathcal{T}_h^a}{\frac{h_\kappa^{q_\kappa-1}}{p_{a,\kappa}^{n-3/2}}}
\left(
\norm{\widetilde{\mathcal{E}}\dot{\varphi}}_{n,\mathcal{K}_\kappa}+
\norm{\widetilde{\mathcal{E}}{\varphi}}_{n,\mathcal{K}_\kappa}+
\int_{0}^{t}{\left[\norm{\widetilde{\mathcal{E}}\ddot{\varphi}}_{n,\mathcal{K}_\kappa}+
\norm{\widetilde{\mathcal{E}}\dot{\varphi}}_{n,\mathcal{K}_\kappa} \right]\hspace{-1mm}(s)\, ds}\hspace{-0.5mm}\right)\hspace{-0.5mm},
\end{align*}
\MB{where the hidden constant depends on time $t$ and on the material properties, but is independent of the discretization parameters and of $\kint$.}
\end{theorem}
\begin{proof}
For any time \MB{$t\in (0,T]$}, let  $(\bm u_I,\bm w_I,\varphi_I)(t)\in\bm{V}_h^p\times \bm{V}_h^p\times V_h^a$  be the interpolants defined in \eqref{lem::interp2}. We split the error as $\bm E(t)=\bm E_I(t)-\bm E_h(t)$, where 
\begin{align*}
\bm E_I(t)&=(\bm e_I^u,\bm e_I^w,e_I^\varphi)(t)=(\bm u-\bm u_I,\bm w-\bm w_I,\varphi-\varphi_I)(t),\\
 \bm E_h(t)&=(\bm e_h^u,\bm e_h^w,e_h^\varphi)(t)=(\bm u_h-\bm u_I,\bm w_h-\bm w_I,\varphi_h-\varphi_I)(t).
\end{align*}
From the triangle inequality  we have
\RRR{
$\norm{\bm{E}(t)}_{ \rm{E}}^2\leq
\norm{\bm{E}_h(t)}_{\rm{E}}^2+
\norm{\bm{E}_I(t)}_{\rm{E}}^2,
$}
and Lemma~\ref{lem::interp2} can be used to bound the term \RRR{$\norm{\bm{E}_I(t)}_{\rm{E}}$}.  As for the term \RRR{$\norm{\bm{E}_h(t)}_{\rm{E}}$},
by taking $(\bm v,\bm \xi,\psi)=(\dot{\bm e}_h^u,\dot{\bm e}_h^w,\dot{e}_h^\varphi) \in \bm{V}_h^p\times \bm{V}_h^p\times V_h^a$ as test functions in \eqref{eq::error_eq},
taking into account that $\bm E=\bm E_I-\bm E_h$, neglecting the coupling terms thanks to skew-symmetry and collecting a first time derivative, identity \eqref{eq::error_eq} can be rewritten as
\MB{\begin{multline}\label{eq::errors_main}
\frac{1}{2}\frac{d}{dt}\left( \mathcal{M}(\dot{\bm{E}}_h, \dot{\bm{E}}_h)  +
\mathcal{A}_h(\bm{E}_h,\bm{E}_h) \right)
+\mathcal{B}(\dot{\bm e}_h^w,\dot{\bm e}_h^w)
=\mathcal{M}(\ddot{\bm{E}}_I, \dot{\bm{E}}_h) - \mathcal{A}_h(\dot{\bm{E}}_I,\bm{E}_h)  \\ 
+\frac{d}{dt}\mathcal{A}_h(\bm{E}_I,\bm{E}_h) +\mathcal{B}(\dot{\bm e}_I^w,\dot{\bm e}_h^w)
+\mathcal{C}_h^p(\dot{e}_I^\varphi,\dot{\bm e}_h^u+ \dot{\bm e}_h^w) 
+\mathcal{C}_h^a(\dot{\bm{e}}_I^u + \dot{\bm{e}}_I^w,\dot{e}_h^\varphi),
\end{multline}
where we have used Leibniz's rule on the term $\mathcal{A}_h(\bm{E}_I,\dot{\bm{E}}_h)$. Integrating  \eqref{eq::errors_main} between $0$ and $t\le T$ and observing that 
$\bm E_h(0) = (\bm e_h^u(0),\bm{e}_h^w(0),e_h^\varphi(0))=\bm{0}$, it is inferred that
\begin{multline*}
\mathcal{M}(\dot{\bm{E}}_h, \dot{\bm{E}}_h) (t) +
\mathcal{A}_h(\bm{E}_h,\bm{E}_h) (t) +2 \int_{0}^{t}\mathcal{B}(\dot{\bm e}_h^w,\dot{\bm e}_h^w)(s) \,ds \\ 
= 2\int_{0}^t\mathcal{M}(\ddot{\bm{E}}_I,\dot{\bm{E}}_h)(s) \,ds
- 2\int_{0}^t\mathcal{A}_h(\dot{\bm{E}}_I,\bm{E}_h)(s)\,ds
+ 2\int_0^t \mathcal{B}(\dot{\bm e}_I^w,\dot{\bm e}_h^w)(s) \, ds \\ 
+ 2 \mathcal{A}_h(\bm{E}_I,\bm{E}_h) (t)  
+ 2 \int_0^t  \left(\mathcal{C}_h^p(\dot{e}_I^\varphi,\dot{\bm e}_h^u+ \dot{\bm e}_h^w)(s) 
+\mathcal{C}_h^a(\dot{\bm{e}}_I^u + \dot{\bm{e}}_I^w,\dot{e}_h^\varphi)(s) \right)\, ds.
\end{multline*}
Applying the Cauchy--Schwarz and Young inequalities on the third and fourth terms in the right-hand side of the previous identity, we obtain
\begin{equation}\label{eq::errors_main2}
\begin{aligned}
\circled{1} &= 
\mathcal{M}(\dot{\bm{E}}_h, \dot{\bm{E}}_h) (t) +\mathcal{A}_h(\bm{E}_h,\bm{E}_h) (t)
+ \int_{0}^{t}\mathcal{B}(\dot{\bm e}_h^w,\dot{\bm e}_h^w)(s) \,ds \\ 
&\le 4\int_{0}^t\mathcal{M}(\ddot{\bm{E}}_I,\dot{\bm{E}}_h)(s) \,ds 
- 4\int_{0}^t\mathcal{A}_h(\dot{\bm{E}}_I,\bm{E}_h)(s)\,ds 
+ 2\int_0^t \mathcal{B}(\dot{\bm e}_I^w,\dot{\bm e}_I^w)(s) \, ds\\ 
&+ 4 \mathcal{A}_h(\bm{E}_I,\bm{E}_I) (t)   
+ 4 \int_0^t  \hspace{-1mm}\left(\mathcal{C}_h^p(\dot{e}_I^\varphi,\dot{\bm e}_h^u+ \dot{\bm e}_h^w)
+\mathcal{C}_h^a(\dot{\bm{e}}_I^u + \dot{\bm{e}}_I^w,\dot{e}_h^\varphi) \right) (s)\, ds =  \circled{2}.
\end{aligned}
\end{equation}
Now, using Lemma \ref{lem::cont} together with the fundamental theorem of calculus we estimate the left hand side as 
$\circled{1} \geq \left(\mathcal{M}(\dot{\bm{E}}_h, \dot{\bm{E}}_h)  +
\mathcal{A}_h(\bm{E}_h,\bm{E}_h) + \mathcal{B}({\bm e}_h^w,{\bm e}_h^w)\right)(t) = \norm{\bm{E}_h(t)}^2_{\rm E}$. 
Plugging this into \eqref{eq::errors_main2}, using again the Young inequality and Lemma \ref{lem::cont} to bound the second and fourth terms in $\circled{2}$, and recalling definition \eqref{eq:trinorm_tot}, yields 
\begin{equation}\label{eq::errors_main3}
\begin{aligned}
\norm{\bm{E}_h(t)}^2_{\rm E} &\le 
2 \int_0^t \hspace{-1mm}\norm{{\bm{E}}_h(s)}^2_{\rm E} 
+ \overbrace{( \mathcal{M}(\ddot{\bm{E}}_I, \ddot{\bm{E}}_I) + \trinorm{\dot{\bm E}_I}_{\rm dG}^2 
+ \mathcal{B}(\dot{\bm e}_I^w,\dot{\bm e}_I^w))(s)}^{\trinorm{{\dot{\bm{E}}}_I(s)}^2_{\rm E}} ds\\
&+ 4\, \trinorm{\bm{E}_I (t)}_{\rm dG}^2
+ 4 \int_0^t \hspace{-1mm}\left(\mathcal{C}_h^p(\dot{e}_I^\varphi,\dot{\bm e}_h^u+ \dot{\bm e}_h^w)
+\mathcal{C}_h^a(\dot{\bm{e}}_I^u + \dot{\bm{e}}_I^w,\dot{e}_h^\varphi) \right) (s)\, ds.
\end{aligned}
\end{equation}
Now, recalling the definitions of the coupling bilinear forms $\mathcal{C}_h^p$ and $\mathcal{C}_h^a$ and using the Cauchy-Schwarz inequality followed by the trace-inverse inequality \eqref{eq::traceinv}, we infer}
\begin{align*}
\mathcal{C}_h^p(\MB{\dot{e}_I^\varphi},\MB{\dot{\bm e}_h^u+\dot{\bm e}_h^w})
& \lesssim \sum_{F\in\mathcal{F}_h^I}{\norm{\rho_a\dot{e}_I^\varphi}_F\norm{\dot{\bm e}_h^u+\dot{\bm e}_h^w}_F}  \lesssim
\sum_{\kappa_p\in\mathcal{T}_{h,p}^I,\,\kappa_a\in\mathcal{T}_{h,a}^I}{\norm{\dot{e}_I^\varphi}_{\partial \kappa_a}\norm{\dot{\bm e}_h^u+\dot{\bm e}_h^w}_{\partial \kappa_p}}  \\
&\lesssim \sum_{\kappa_p\in\mathcal{T}_{h,p}^I,\,
\kappa_a\in\mathcal{T}_{h,a}^I}{p_{p,\kappa_p}h_{\kappa_p}^{-1/2}\norm{\dot{e}_I^\varphi}_{\partial \kappa_a}}
\RRR{(\norm{\dot{{\bm e}}_h^u}_{\Omega_p} + \norm{\dot{\bm e}_h^w}_{\Omega_p})}
\end{align*}
where, to infer the last bound, we have also used Assumption~\ref{ass::3}. Therefore, we have
\begin{align*}
\int_0^t\hspace{-.5mm} \mathcal{C}_h^p(\MB{\dot{e}_I^\varphi},\dot{\bm e}_h^u+\dot{\bm e}_h^w)(s) \,ds & \lesssim \int_0^t\hspace{-1mm} \left( \sum_{\kappa\in\mathcal{T}_{h,a}^I}{p_{a,\kappa}h_{\kappa}^{-1/2}\norm{\dot{e}_I^\varphi(s)}_{\partial \kappa}}
   \right)\RRR{(\norm{\dot{{\bm e}}_h^u}_{\Omega_p} + \norm{\dot{\bm e}_h^w}_{\Omega_p})(s)}\, ds  \\ & \stackrel{\text{def}}{=} \int_0^t \MB{\mathcal{I}_h^a}(\dot{e}_I^\varphi(s))\,
   \RRR{(\norm{\dot{{\bm e}}_h^u(s)}_{\Omega_p} + \norm{\dot{\bm e}_h^w(s)}_{\Omega_p})}\, ds. \nonumber 
\end{align*}
Proceeding in the same way, we can conclude that
\begin{align*}
\int_0^t\hspace{-.5mm} \MB{\mathcal{C}_h^a(\dot{\bm e}_I^u+\dot{\bm e}_I^w}, \dot{e}_h^\varphi)(s) \,ds & 
\lesssim \int_0^t \hspace{-1mm}\left( \sum_{\kappa\in\mathcal{T}_{h,p}^I}p_{p,\kappa}h_{\kappa}^{-1/2}
 (\norm{\dot{{\bm e}}_I^u}_{\partial \kappa}+ \norm{\dot{{\bm e}}_I^w}_{\partial \kappa})(s)
 \right)
\RRR{\norm{\dot{e}_h^\varphi(s)}_{\Omega_a}}
   \, ds  \\ & \stackrel{\text{def}}{=} \MB{\int_0^t (\mathcal{I}_h^p(\dot{{\bm e}}_I^u(s)) + \mathcal{I}_h^p(\dot{{\bm e}}_I^w(s)) )\norm{\dot{e}_h^\varphi(s)}_{\Omega_a}} \, ds,
\end{align*}
Collecting the two previous bounds and applying Young's inequality \MB{together with inequality \eqref{eq:M-coer}, it is inferred that
\begin{multline*}
\int_0^t \left(\mathcal{C}_h^p(\dot{e}_I^\varphi,\dot{\bm e}_h^u+ \dot{\bm e}_h^w)
+\mathcal{C}_h^a(\dot{\bm{e}}_I^u + \dot{\bm{e}}_I^w,\dot{e}_h^\varphi) \right) (s)\, ds \\
\lesssim
\int_0^t\trinorm{{\bm{E}}_h(s)}^2_{\rm E}\, ds 
+ \int_0^t \hspace{-1mm}\left( \mathcal{I}_h^a(\dot{e}_I^\varphi)^2  + \mathcal{I}_h^p(\dot{{\bm e}}_I^u)^2 + \mathcal{I}_h^p(\dot{{\bm e}}_I^w)^2 \right)(s)\, ds.
\end{multline*}
Hence, plugging the previous bound into \eqref{eq::errors_main3} and using Gronwall's Lemma, we get
$$
\norm{\bm{E}_h(t)}^2_{\rm E} \lesssim 
\trinorm{\bm{E}_I(t)}^2_{\rm E} +\int_0^t\trinorm{\dot{\bm{E}}_I(s)}^2_{\rm E}\, ds 
+ \int_0^t \hspace{-1mm}\left( \mathcal{I}_h^a(\dot{e}_I^\varphi)^2  + \mathcal{I}_h^p(\dot{{\bm e}}_I^u)^2 + \mathcal{I}_h^p(\dot{{\bm e}}_I^w)^2 \right)(s)\, ds,
$$
To estimate the terms on the right hand side, we make use of Lemma~\ref{lem::interp2} and the following bounds inferred from \cite[Lemma 33]{CangianiDongGeorgoulisHouston_2017}:}
\begin{gather*}
\mathcal{I}_h^a(\dot{e}_I^\varphi)^2\lesssim
\sum_{\kappa\in\mathcal{T}^a_{h,I}}\frac{h_\kappa^{2q_\kappa-2}}{p_{a,\kappa}^{2n-3}}
\norm{\widetilde{\mathcal{E}}\dot{\varphi}}^2_{n,\mathcal{K}_\kappa}\\
\mathcal{I}_h^p(\dot{\bm e}_I^u)^2+
\mathcal{I}_h^p(\dot{\bm e}_I^w)^2
\lesssim
\sum_{\kappa\in\mathcal{T}^p_{h,I}}\frac{h_\kappa^{2s_\kappa-2}}{p_{p,\kappa}^{2m-3}}
\norm{\widetilde{\mathcal{E}}\dot{\bm u}}_{m,\mathcal{K}_\kappa}^2	+
\sum_{\kappa\in\mathcal{T}^p_{h,I}}\frac{h_\kappa^{2r_\kappa-2}}{p_{p,\kappa}^{2\ell-3}}
\norm{\widetilde{\mathcal{E}}\dot{\bm w}}^2_{\ell,\mathcal{K}_\kappa}.
\end{gather*}
As a result, the thesis follows.
\end{proof}

\begin{corollary}
Under the hypotheses of Theorem \ref{thm::error-estimate},
assume that $h\approx h_\kappa$ for any $\kappa \in \mathcal{T}_h^p \cup \mathcal{T}_h^a$, $p_{p,\kappa} = p$ for any $\kappa \in \mathcal{T}_h^p$ and 
$p_{a,\kappa} = q$ for any $\kappa \in \mathcal{T}_h^a$. Then, if 
 $\bm u\in C^2([0,T];\bm H^m(\Omega_p))$, $\bm w\in C^2([0,T];\bm H^\ell(\Omega_p))$ and $\varphi\in C^2([0,T]; H^n(\Omega_a))$, with $m,\ell\geq p +1$, $n\geq q+1$ the error estimate of Theorem \ref{thm::error-estimate} reads
\begin{align*}
\norm{\bm E(t)}_{\rm E}\lesssim&
\frac{h^{p}}{p^{m-3/2}}
\left(
\norm{\widetilde{\mathcal{E}}\dot{\bm u}}_{m,\mathcal{K}_\kappa}+
\norm{\widetilde{\mathcal{E}}{\bm u}}_{m,\mathcal{K}_\kappa}+
\int_{0}^{t}{\left[\norm{\widetilde{\mathcal{E}}\ddot{\bm u}}_{m,\mathcal{K}_\kappa}+
\norm{\widetilde{\mathcal{E}}\dot{\bm u}}_{m,\mathcal{K}_\kappa}\right](s)\, ds}
\right)\nonumber\\+&
\frac{h^{p}}{p^{\ell-3/2}}
\left(
\norm{\widetilde{\mathcal{E}}\dot{\bm w}}_{\ell,\mathcal{K}_\kappa}+
\norm{\widetilde{\mathcal{E}}{\bm w}}_{\ell,\mathcal{K}_\kappa}+
\int_{0}^{t}{\left[\norm{\widetilde{\mathcal{E}}\ddot{\bm w}}_{\ell,\mathcal{K}_\kappa}+
\norm{\widetilde{\mathcal{E}}\dot{\bm w}}_{\ell,\mathcal{K}_\kappa}\right](s)\, ds}
\right)\nonumber\\+&
{\frac{h^{q}}{q^{n-3/2}}}
\left(
\norm{\widetilde{\mathcal{E}}\dot{\varphi}}_{n,\mathcal{K}_\kappa}+
\norm{\widetilde{\mathcal{E}}{\varphi}}_{n,\mathcal{K}_\kappa}+
\int_{0}^{t}{\left[\norm{\widetilde{\mathcal{E}}\ddot{\varphi}}_{n,\mathcal{K}_\kappa}+
\norm{\widetilde{\mathcal{E}}\dot{\varphi}}_{n,\mathcal{K}_\kappa}\right](s)\, ds} \right),
\end{align*}
\MB{where the hidden constant depends on time $t$ and on the material properties, but is independent of the discretization parameters and $\kint$.}
The above bounds are optimal in $h$ and suboptimal in $p$ and $q$ by a factor $\frac12$, see \cite{PerugiaLDG}.
\end{corollary}

\section{Time discretization}\label{sec::timedis}
\RRR{To integrate in time equation \eqref{eq::algebraic}, we first discretize the interval} $[0,T]$  by introducing a timestep $\Delta t>0$, such that $\forall \ k\in\mathbb{N}$, $t_{k+1}-t_k=\Delta t$ and define $\bm X^k$ as  $\bm X^k=\bm X(t^k)$, with $\bm{X}=
[U, W, \Phi]^T$. 
\RRR{Next, we rewrite} equation \eqref{eq::algebraic} in compact form as
$\bm{A}\ddot{\bm{X}}+\bm{B}\dot{\bm{X}}+\bm{C}\bm{X}=\bm{F}$ and get
\begin{equation}\label{eq:2ndorder_time}
\ddot{\bm{X}}=\bm{A}^{-1}(\bm{F}-\bm{B}\dot{\bm{X}}-\bm{C}\bm{X})=
\bm{A}^{-1}\bm{F}-\bm{A}^{-1}\bm{B}\dot{\bm{X}}-\bm{A}^{-1}\bm{C}\bm{X}=\mathcal{L}(t,\bm{X},\dot{\bm{X}}),
\end{equation}
\RRR{Finally, to integrate in time \eqref{eq:2ndorder_time} we can apply the Newmark$-\beta$ or the leap-frog scheme as follows.} 
The Newmark$-\beta$ scheme is defined by introducing a Taylor expansion for displacement and velocity, respectively:
\begin{equation}
\begin{cases}
\bm{X}^{k+1}=\bm{X}^k+\Delta t \bm{Z}^k+\Delta t^2(\beta_N\mathcal{L}^{k+1}+(\frac{1}{2}-\beta_N)\mathcal{L}^{k}),\\[5pt]
\bm{Z}^{k+1}=\bm{Z}^k+\Delta t(\gamma_N\mathcal{L}^{k+1}+(1-\gamma_N)\mathcal{L}^k),
\end{cases}
\label{eq::newmark_taylor}
\end{equation}
where $\bm{Z}^k=\dot{\bm{X}}(t^k)$, $\mathcal{L}^k=\mathcal{L}(t^k,\bm X^k,\bm Z^k)$ and the Newmark parameters $\beta_N$ and $\gamma_N$  satisfy, the following constraints $0\leq\gamma_N\leq 1$, $0\leq2\beta_N\leq1$.
The typical choices of parameters are $\gamma_N=1/2$ and $\beta_N=1/4$, for which the scheme is unconditionally stable and second order accurate. 
Finally, by plugging the definition of $\mathcal{L}$ into \eqref{eq::newmark_taylor}, for $k\geq 0$, the time integration reduces to:
\begin{align*}
\begin{bmatrix}
\bm{A}+\Delta t^2\beta_N\bm{C}&\Delta t^2\beta_N\bm{B}\\
\gamma_N\Delta t\bm{C}&\bm{A}+\gamma_N\Delta t\bm{B}
\end{bmatrix}
\begin{bmatrix}
\bm{X}^{k+1}\\ \bm{Z}^{k+1}
\end{bmatrix}   = 
  \begin{bmatrix}
\bm{A}-\Delta t^2\widetilde{\beta}_N\bm{C}&\Delta t\bm{A}-\Delta t^2\widetilde{\beta}_N\bm{B}\\
-\widetilde{\gamma}_N\Delta t\bm{C}&\bm{A}-\widetilde{\gamma}_N\Delta t\bm{B}
\end{bmatrix}
\begin{bmatrix}
\bm{X}^{k}\\\bm{Z}^{k}
\end{bmatrix} \\  +
\begin{bmatrix}
\Delta t^2\beta_N\bm{F}^{k+1}+ \Delta t^2\widetilde{\beta}_N\bm{F}^{k}\\\gamma_N\Delta t \bm{F}^{k+1} + \widetilde{\gamma}_N\Delta t \bm{F}^{k} \nonumber
\end{bmatrix},
\end{align*}
where $\widetilde{\beta}_N=(\frac{1}{2}-\beta_N)$ and $\widetilde{\gamma}_N = (1-\gamma_N)$.
\RRR{By applying the leap-frog scheme to  \eqref{eq:2ndorder_time} we get
\begin{equation}\label{eq::leapfrog1}
(\bm{A}+\frac{\Delta t^2}{2}\bm{B}) \bm{X}^{k+1} = \Delta t^2  \bm{F}^{k} + (2\bm{A} - \Delta t^2 \bm{C})\bm{X}^k + (\frac{\Delta t}{2}\bm{B} - \bm{A})\bm{X}^{k-1}, 
\end{equation}
for $k\geq 1$ with initial step 
\begin{equation}\label{eq::leapfrog2}
\bm{A}\bm{X}^{1} = (\bm{A} - \frac{\Delta t^2}{2} C)\bm{X}^0 + (\Delta t \bm{A} - \frac{\Delta t^2}{2}\bm{B})\bm{Z}^0 + \frac{\Delta t^2}{2}\bm{F}^0.
\end{equation}
Recall that \eqref{eq::leapfrog1}--\eqref{eq::leapfrog2} is explicit and second order accurate.
\begin{remark}
The leap-frog method is often applied to wave propagation problems due to its ease of implementation, the reduced size of the system (compared to a Newmark-type scheme), and because typically the matrix of the linear system to be solved is easily invertible. The latter in fact turns out to be diagonal or block-diagonal when using a dG method for the approximation in space. We note that in equation \eqref{eq::leapfrog1} this does not occur due to the coupling conditions at the interface between the poro-elastic and acoustic domains.
As a further constraint, the fact that in poroelastic-acoustic  materials there is an additional compressional wave of second kind (slow P-wave) to be correctly propagated has an impact on the time integration scheme. Indeed, as a further outcome of the model, the amplitudes of the wavefield are attenuated because of energy loss due to the presence of a viscous fluid. In the case of low frequencies and a viscous fluid, the wave equations become stiff. In other words, the slow P-wave becomes the diffusive
mode, which dominates the character of the equation and drastically restricts the stability condition for  explicit methods. For these reasons we prefer to use an implicit time scheme, cf. also  \cite{chiavassa_lombard_2013,delapuente2008}.
\end{remark}
}

\section{Numerical results}\label{sec::results}
Numerical implementation has been carried out with \textsc{Matlab}. Meshes have been generated through the \texttt{polymesher} software, cf. \cite{paulino}.
\subsection*{Test case 1}
The model problem is solved in  $\Omega=(-1,1)\times(0,1)$, on a sequence of \textit{polygonal meshes} as the one shown in Figure~\ref{fig::mesh}, and with physical parameters shown in Table~\ref{param}.  For the first test case, we choose as exact solution
 \begin{align*}
& \bm{u}(x,y;t)=
\begin{pmatrix}
x^2 \cos(\frac{\pi x}{2}) \sin(\pi x)
 \\[5pt]
x^2 \cos(\frac{\pi x}{2}) \sin(\pi x)
\end{pmatrix}
\cos(\sqrt{2}\pi t),
&&
\bm{w}(x,y;t)=-\bm{u}(x,y;t), \\
&  \varphi(x,y;t)=
 (x^2 \sin(\pi x)\sin(\pi y)) \sin(\sqrt{2}\pi t), &&
\end{align*}
in order to have a null pressure in the whole poroelastic domain. Since the solution together with its first $x-$, $y-$ and $t-$ derivatives are identically zero at the interface $\Gamma={0}\times (0,1)$,  interface coupling conditions are consequently null. This  suggests to test the \textit{sealed pores} ($\kint=0$), the \textit{imperfect pores} $(\kint \in(0,1)$) and the \textit{open pores} ($\kint=1$) cases with the same manufactured solution.
\begin{figure}
\begin{minipage}{\textwidth}
\begin{minipage}{0.45\textwidth}
\centering
\includegraphics[scale=0.45]{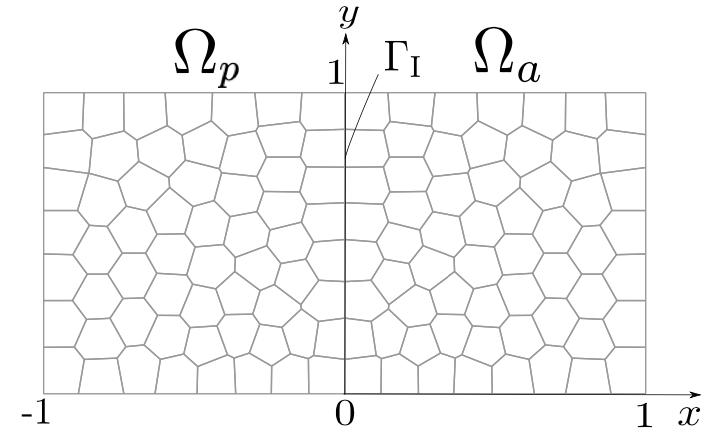}
\captionof{figure}{Test case 1.  Polygonal mesh, with $N = 100$ polygons.}
 \label{fig::mesh}
\end{minipage}
\hfill
\begin{minipage}{0.5\textwidth}
\centering
\begin{tabular}{|c|c|}
\hline
\textbf{Field}      & \textbf{Value} \\ \hline
$\rho_f$, $\rho_s$ & 1                           \\ \hline
$\lambda$, $\mu$   & 1                           \\ \hline
$a$   & 1                           \\ \hline
$\phi$   & 0.5                           \\ \hline
$\eta$             & 0                       \\ \hline
$\rho_w$           & 2                       \\ \hline
$\beta$, m         & 1                       \\ \hline
$c$, $\MB{\rho_a}$ & 1                       \\ \hline
\end{tabular}
\captionof{table}{Test case 1.  Physical parameters.}
\label{param}
\end{minipage}
\end{minipage}
\end{figure}
%%%%%%%%%
A sequence of uniformly refined polygonal meshes have been considered, with uniform polynomial degree $p_{p,\kappa}=p_{a,\kappa}=p = 1,2,3$. The final time $T$ has been set equal to $0.25$, considering a timestep of $\Delta t=10^{-4}$ for the Newmark-$\beta$ scheme,  $\gamma_N=1/2$ and $\beta_N=1/4$.
The penalty parameters $c_1, c_2$ and $c_3$ appearing in the definition \eqref{eq::stab_1}--\eqref{eq::stab_3} have been chosen equal to 10.
In Figure \ref{L2errorsvsh} (left) we report the computed errors as a function of the inverse of the mesh-size (log-log scale), for the case $p=3$. As predicted by Theorem \ref{thm::error-estimate} the errors decays proportionally to $h^{3}$.
\begin{figure}[htbp]
\centering
   \input{sealed_pores_h_3.tikz}
   \input{sealed_pores_N.tikz}
   \input{intermediate_pores_h_3.tikz}
   \input{intermediate_pores_N.tikz}
   \input{open_pores_h_3.tikz}
   \input{open_pores_N.tikz}
\caption{Test case 1. Left: computed errors in the energy norm, at the final time $T$,  as a function of $h$ ($p=3$). Right: Computed errors \MB{in the $L^2$-norm}, at final time T, as a function of the polynomial degree $p$ on a  computational mesh of $N=100$ polygons.} 
\label{L2errorsvsh}
\end{figure}
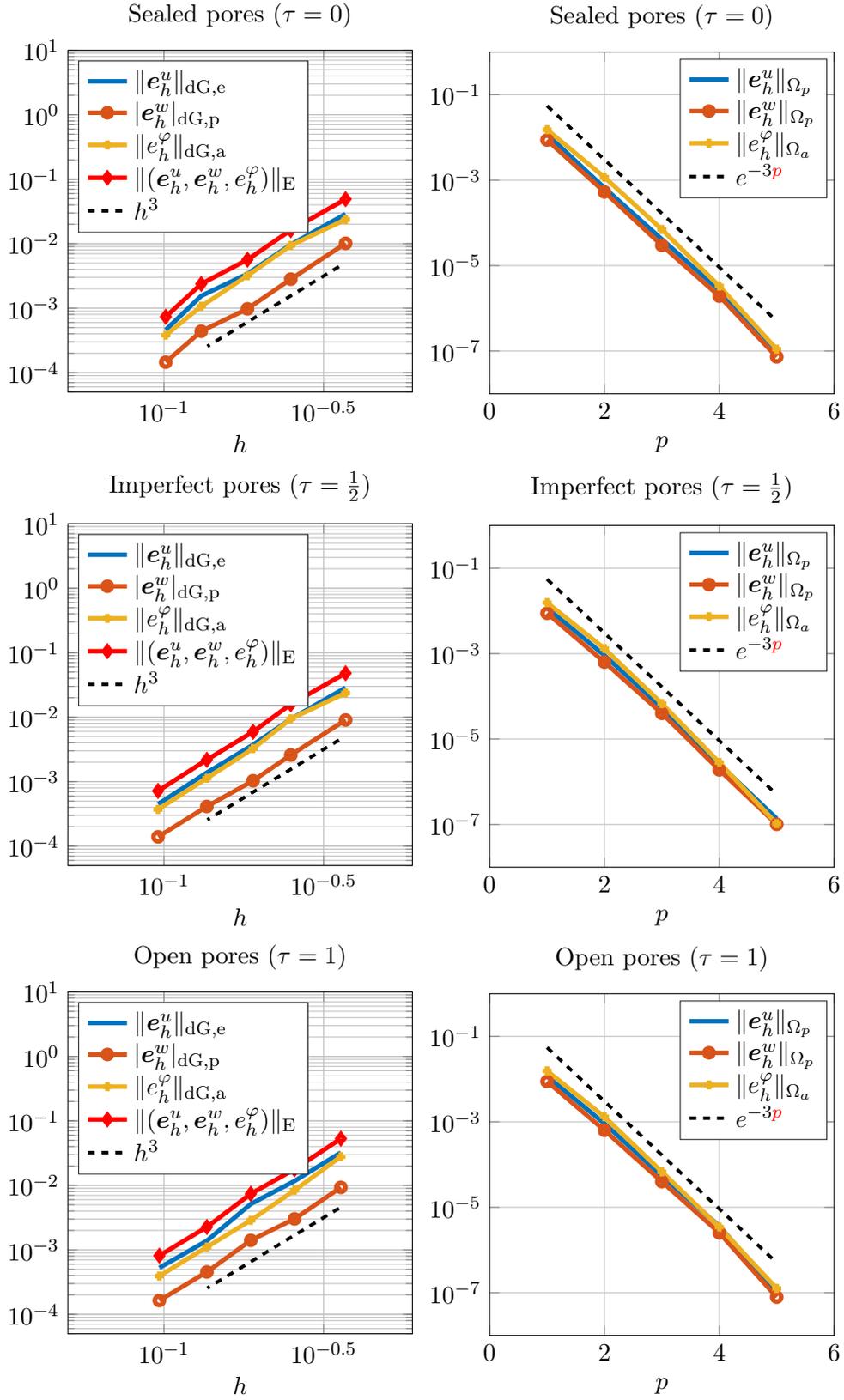
%%%%%%%%%
\noindent
Moreover, we have also computed the $L^2$-errors on the pressure field $p$. These results are reported Figure~\ref{fig::L2pressure} and show a convergence rate proportional to $h^3$, as expected. We point out the that discrete pressure has been computed through equation \eqref{eq::const_sigma_press}.
\begin{figure}[htbp]
\centering
   \input{sealed_pores_pressure.tikz}
\caption{Test case 1.  Computed errors $||p-p_h||_{\Omega}$, at the final time $T$,  as a function of $h$ ($p=3$).} 
\label{fig::L2pressure}
\end{figure}
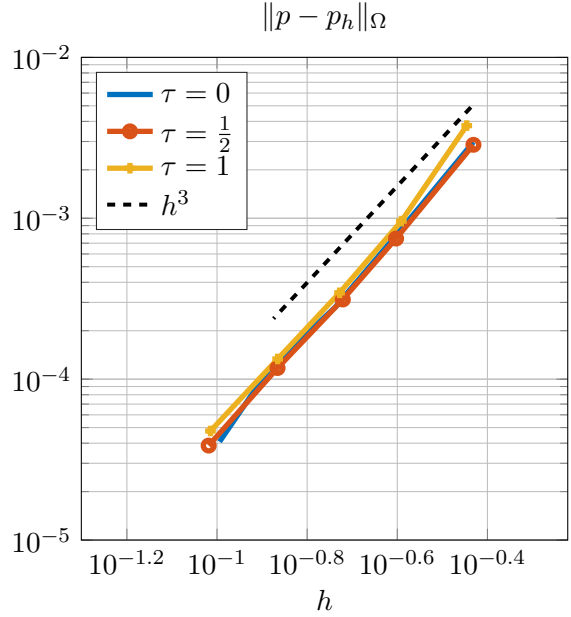
%%%%%
Finally, we compute the $L^2$ norm of the error fixing a computational mesh of $N=100$ polygons and varying the polynomial degree $p=1,2,\ldots, 5$. The computed errors are reported in Figure \ref{L2errorsvsh} (right) (semi-log scale), and an exponential decay of the error is clearly attained.
%\begin{figure}[htbp]
%\centering
%\input{sealed_pores_N.tikz}
%\input{intermediate_pores_N.tikz}
%   \input{open_pores_N.tikz}
%\caption{Test case 1.  Computed errors, at the final time $T$, as a function of the polynomial degree $p$ on a  computational mesh of $N=400$ polygons.}
%\label{fig::perrors}
%\end{figure}
%
%
\subsection*{Test case 2.   Oblique interface}
The second test cases consider a domain $\Omega=(0,400)\times(0,400)\ {\rm m^2}$, with a  straight interface with slope 60$^\circ$, cf. Figure \ref{fig::test_2_space}. Physical and dimensional parameters have been chosen as in \cite{chiavassa_lombard_2013} and listed in  Table \ref{tab::test1}.
%%%%%%%%%%
\begin{table}[htbp]
\centering
\begin{tabular}{llllll}
\cline{1-5}
\textbf{Fluid}  & Fluid density      & $\rho_f, \MB{\rho_a}$    & 1000                 & $\rm kg/m^3$    &  \\
& Wave velocity      & $c$         & 1500                 & $\rm m/s$       &  \\
& Dynamic viscosity  & $\eta$      & 0                    & $\rm Pa\cdot s$ &  \\ \cline{1-5}
\textbf{Grain}  & Solid density      & $\rho_s$    & 2690                 & $\rm kg/m^3$    &  \\
& Shear modulus      & $\mu$       & 1.86$\cdot 10^9$     & $\rm Pa$        &  \\ \cline{1-5}
\textbf{Matrix} & Porosity           & $\phi$      & 0.38                 &             &  \\
& Tortuosity         & $a$         & 1.8                  &             &  \\
& Permeability       & $k$         & $2.79\cdot 10^{-11}$ & $\rm m^2$       &  \\
& Lam\'e coefficient   & $\MB{\lambda}$ & $1.2\cdot 10^{8}$    & $\rm Pa$        &  \\
& Biot's coefficient & $m$         & $5.34\cdot 10^{9}$   & $\rm Pa$        &  \\
& Biot's coefficient & $\beta$     & 0.95                 &             &  \\  \cline{1-5}
\textbf{\MB{Interface}} & \MB{Interface permeability}           & $\tau$      & \MB{\{0; $10^{-8}$; 1\}}
&             &  \\
\cline{1-5}
\end{tabular}
\caption{Test case 2.   Physical parameters.}
\label{tab::test1}
\end{table}
%%%%%%%%%%

\noindent
Boundary and initial conditions have been set equal to zero both for the poroelastic and the acoustic domain. Forcing terms are null in $\Omega_p$, while in $\Omega_a$ a forcing term is imposed until $t=0.05\ \rm s$, by considering the following load: \RRR{$f_a=r(x,y)h(t),$}
where
\begin{equation}
h(t)=
\begin{cases}
\sum_{k=1}^{4}{\alpha_k\sin(\gamma_k\omega_0 t)}, &\text{if } 0<t<\frac{1}{f_0}\\[5pt]
0,&\text{otherwise,}
\end{cases}
\label{eq::forcing_time}
\end{equation}
with coefficients defined as: $\alpha_1=1,$ $\alpha_2=-21/32,$ $\alpha_3=63/768,$ $\alpha_4=-1/512$, $\gamma_k=2^{k-1}$, $\omega_0=2\pi f_0\ \rm Hz$, $f_0=20\ \rm Hz$.  The function \RRR{ $r(x,y)$ is  defined as $ r(x,y)= 1$,  if $(x,y) \in \bigcup_{i=1}^4 B({\bf x}_i,R)$, while $r(x,y)=0$, otherwise, where $B({\bf x}_i,R)$ is the circle centered in ${\bf x}_i$ and with radius $R$}. Here, we set ${\bf x}_1 = (250,100)$ m, ${\bf x}_2 = (250,150)$ m, ${\bf x}_3 = (250,200)$ m, ${\bf x}_4 = (250,250)$ m and \RRR{$R=10$} m. 
Notice that, the support of the function $r(x,y)$ has been reported in Figure \ref{fig::test_2_space}, superimposed with a sample of one of the computational meshes employed.
%%%%%%%%%%%
\begin{figure}
 \subfloat[Test case 2. \label{fig::test_2_space}]{\includegraphics[scale=0.5]{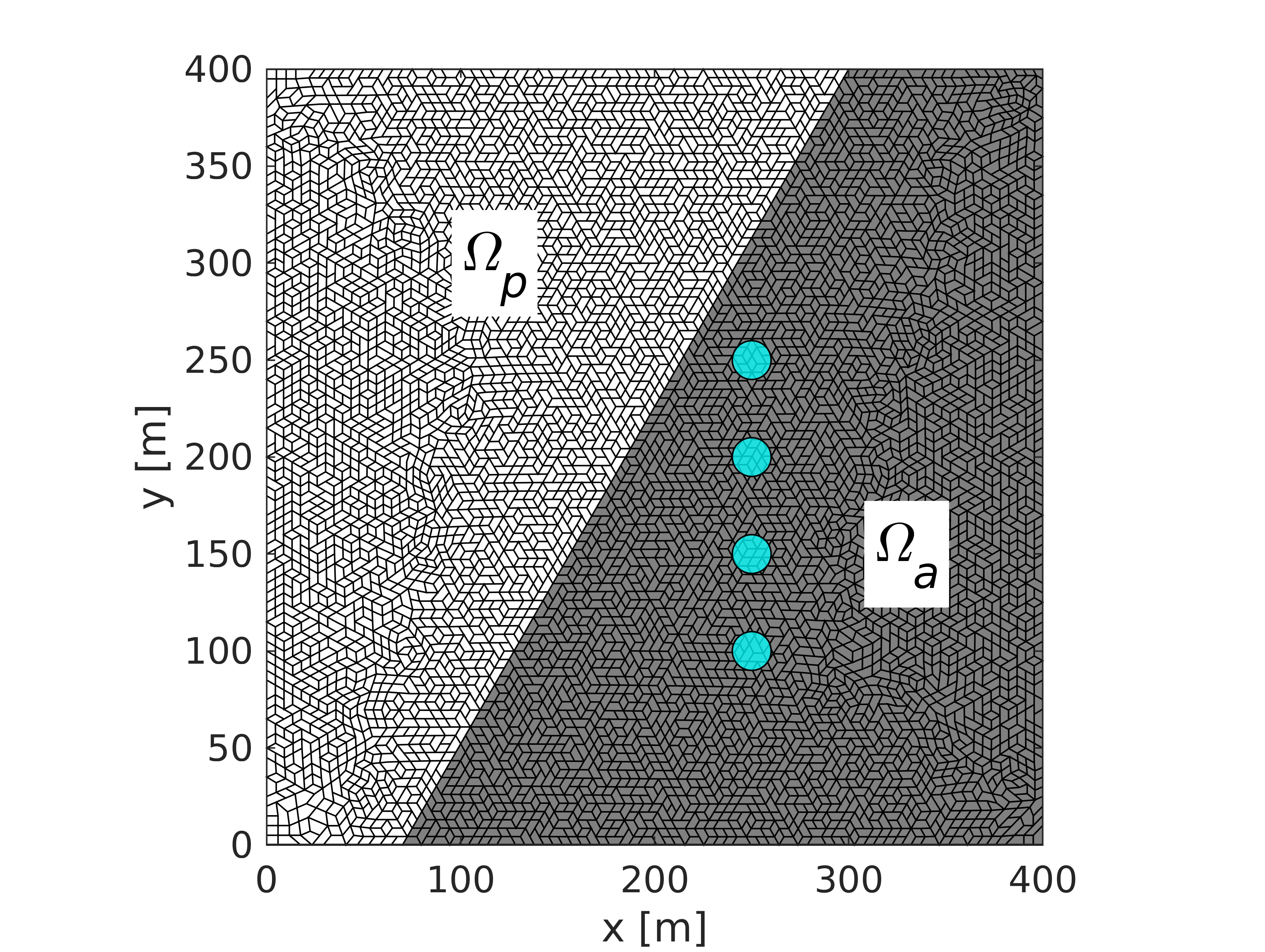}}
 \subfloat[Test case 3. \label{fig::test_4_space}]{\includegraphics[scale=0.5]{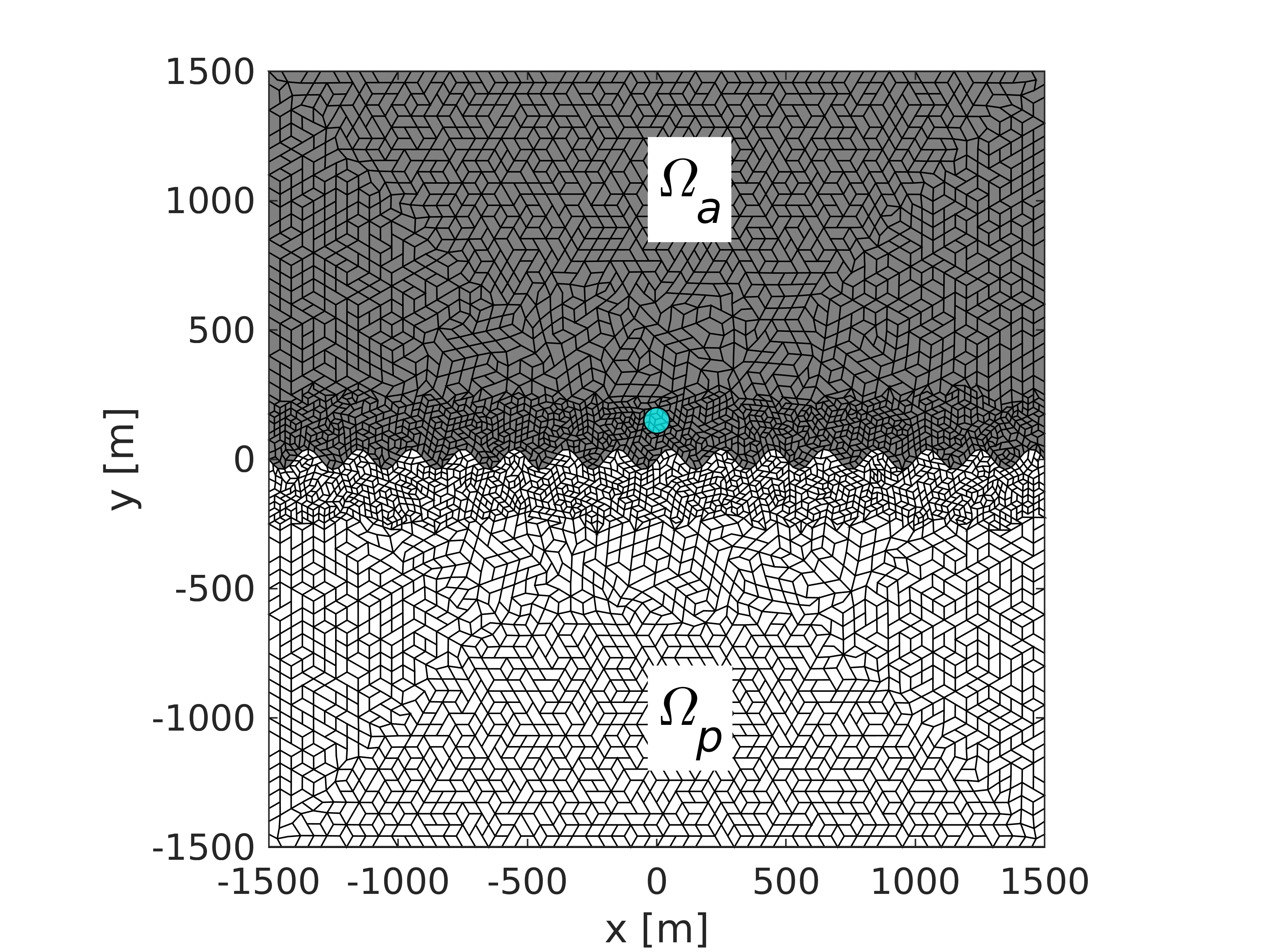}} 
 \caption{Test cases 2 and 3. Computational domains and computational grids. The support of $\bm{r}(x,y)$ is also superimposed in cyan over the mesh.}
\end{figure}
%%%%%%%%%%%
%%%%%%%%%%%
%\begin{figure}[htbp]
%	\centering
%	\includegraphics[scale=0.5]{test_obliquo_spaceforce.png}
%	\caption[Test 2: $\bm{r}(x,y)$ function over the mesh]{}
%	\label{fig::test_2_space}
%\end{figure}
%%%%%%%%%%%
Simulations have been  carried out  by considering: a polygonal mesh consisting in $N=6586$ polygons, subdivided into $N_a=3564$ and $N_p=3022$ polygons for the acoustic and poroelastic domain, respectively; a Newmark scheme with time step $\Delta t=10^{-3}\ \rm s$ and $\gamma_N=1/2$ and $\beta_N=1/4$ in a time interval $[0,0.15]\ \rm s$; a polynomial degree $p_{p,\kappa}=p_{a,\kappa}=p = 4$. 
\RRR{In Figure \ref{fig::test2}, we show the computed pressure $p_h$ 
considering the interface permeability $\tau=0,10^{-8}$ and $\tau=1$, respectively.  The latter values aim at modeling \textit{sealed}, \textit{imperfect} and \textit{open} pores condition at the interface.} Remark that $p_h = \MB{\rho_a \dot{\varphi}_h}$ in the acoustic domain while $p_h =  \MB{-m(\beta \nabla\cdot\bm{u}_h+\nabla\cdot \bm{w}_h)}$ in the poroelastic one. As one can see, \RRR{the pressure wave correctly propagates from the acoustic domain to the poroelastic one: the continuity at the interface boundary can be appreciated for the case $\tau=1$ (open pores).} 
%%%%%%%%%%
\begin{figure}
\centering
\includegraphics[scale=0.19]{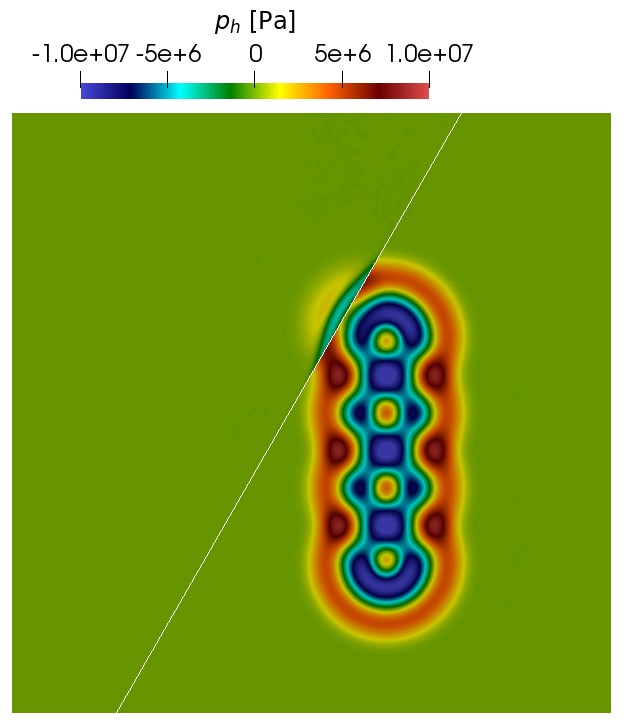}
\includegraphics[scale=0.19]{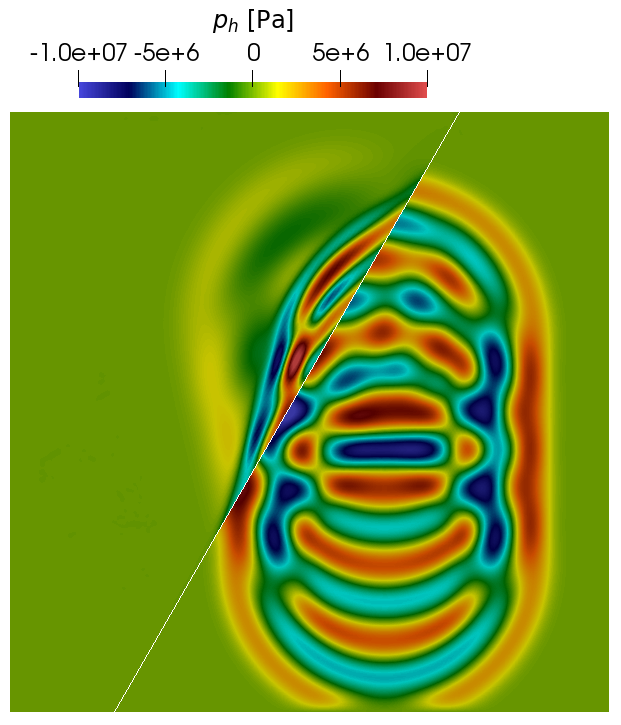}
\includegraphics[scale=0.19]{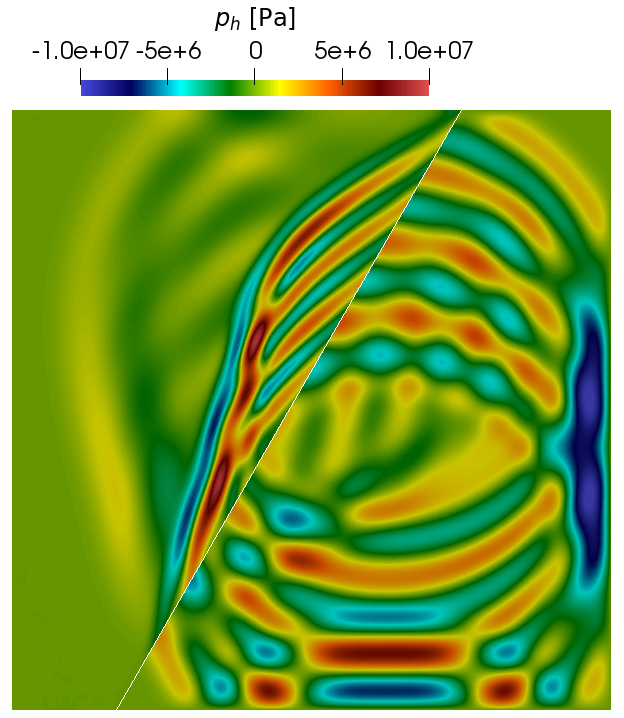}
%\subfigure[imperfect pores]{
\includegraphics[scale=0.19]{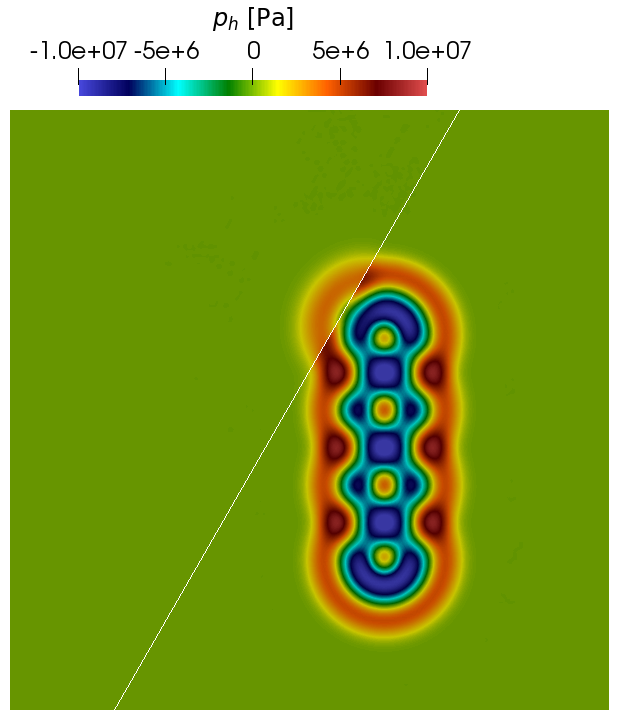}
\includegraphics[scale=0.19]{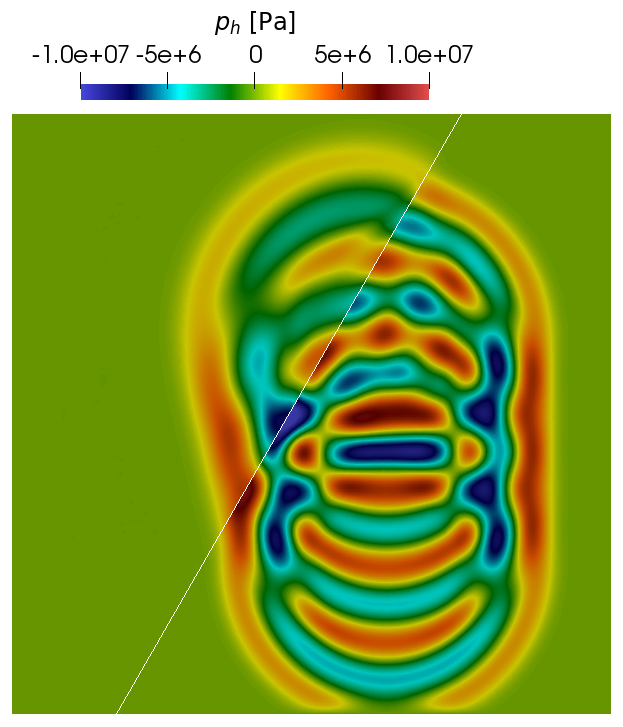}
\includegraphics[scale=0.19]{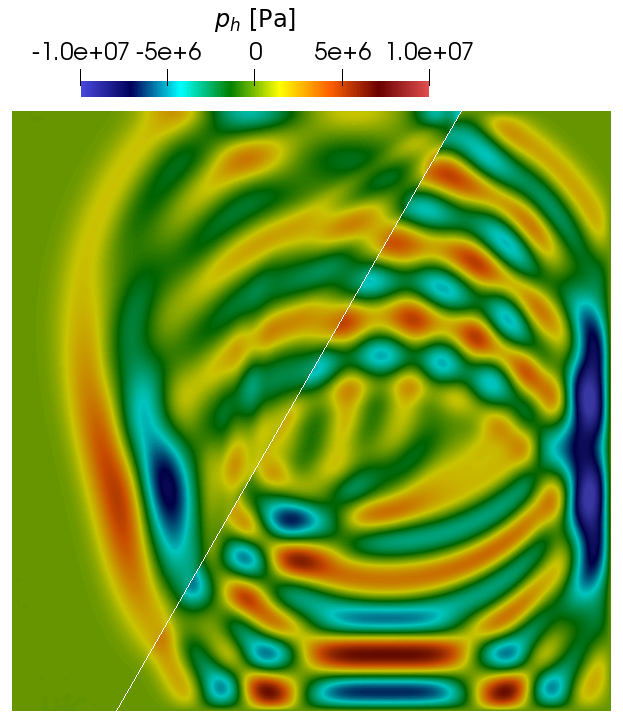}
%}
%\subfigure[open pores]{
\includegraphics[scale=0.19]{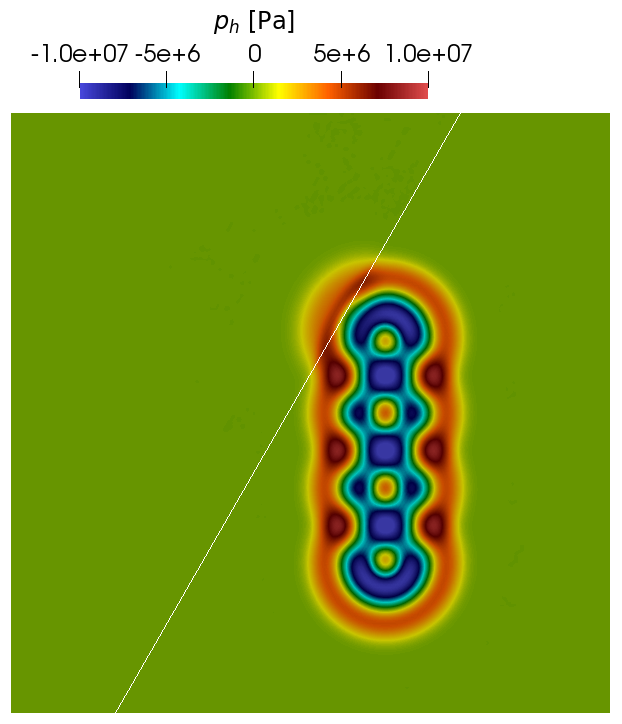}
\includegraphics[scale=0.19]{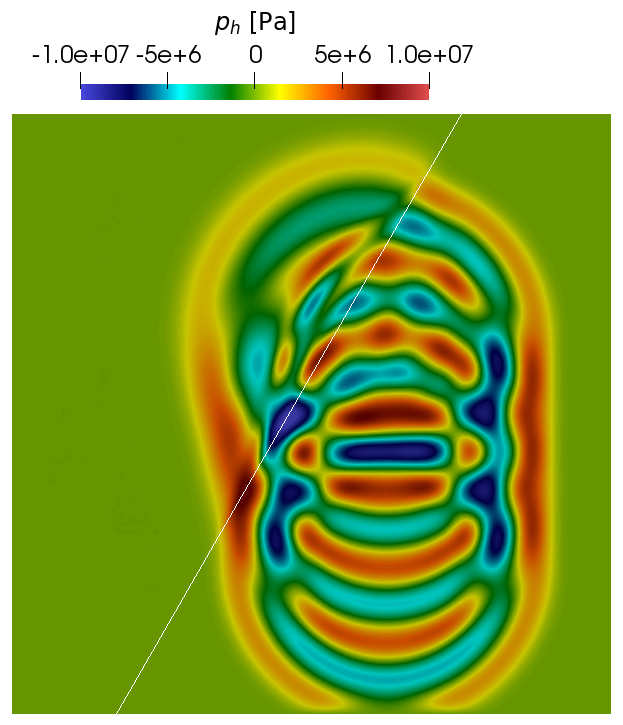}
\includegraphics[scale=0.19]{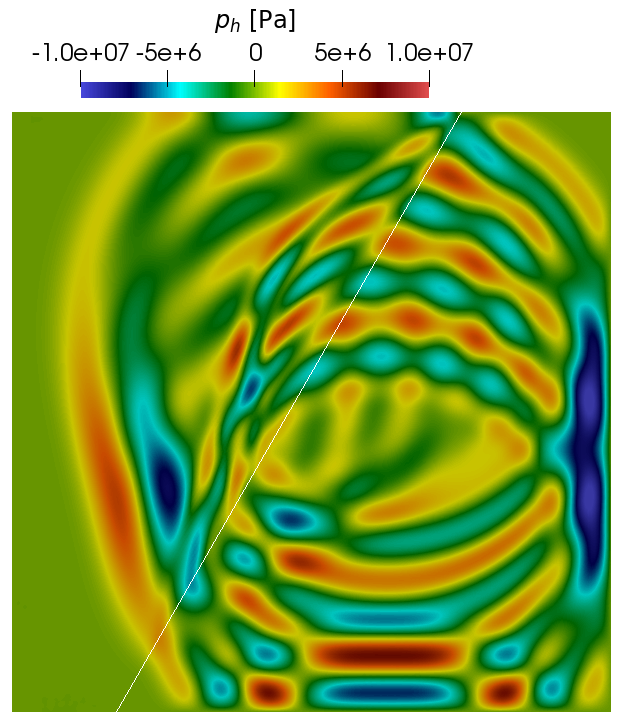}
%}
\caption[Test 2: Pressure solution in the poroelastic-acoustic domain at three time instants]{Test case 2.   Oblique interface. Computed pressure $p_h$ in the poroelastic-acoustic domain at three time instants (from left to right $t=0.04,\ 0.08,\ 0.12 \rm s$), with $\Delta t = 10^{-3}\ \rm s$. \RRR{First line: $\tau=0$ (sealed pores). Second line:  $\tau=10^{-8}$ (imperfect pores). Third line:  $\tau=1$ (open pores).}}
\label{fig::test2}
\end{figure}
\subsection*{Test case 3: Sinusoidal interface}
Finally, with the same data of test case 2, we consider a square domain $\Omega=[-1500,1500]^2\rm m^2$ and a sinusoidal interface $\Gamma$ defined through the relation $\Gamma(x)=40\sin\left(\frac{\pi}{100}x\right)$, cf. Figure \ref{fig::test_4_space}. For this numerical experiment we consider the dynamic viscosity $\eta=0$ and $\eta=0.0015$.
 The number of polygons composing the mesh is $N=5441$, subdivided into $N_a=2713$ and  $N_p=2728$ polygons for the acoustic and poro-elastic subdomains, respectively. Moreover, as shown in Figure \ref{fig::test_4_space}, we have set the initial conditions on the acoustic domain, by defining $h(t)$ as before and \RRR{$r(x,y) = 1/\rho_a$}, if $(x,y) \in B({\bf x}_1,R)$, and equal to $0$, otherwise, with ${\bf x}_1 = (0,150)$ m and $R=50$ m. \RRR{Here we consider the interface permeability $\tau=1$}.
%%%%%%%%%%%%%%%%%%%%
%\begin{figure}[htbp]
%	\centering
%	\includegraphics[width=0.5\textwidth]{mesh_test3_seno.png}
%	\caption[Test case 4.   $\bm{r}(x,y)$ function over the mesh]{Test case 4.   Computational domain and sample of computational grid. The support of $\bm{r}(x,y)$ is also superimposed in cyan over the mesh.}
%	\label{fig::test_4_space}
%\end{figure}
%%%%%%%%%%%%%%%%%%%%
\begin{figure}[htbp]
\centering
\includegraphics[width=0.3\textwidth]{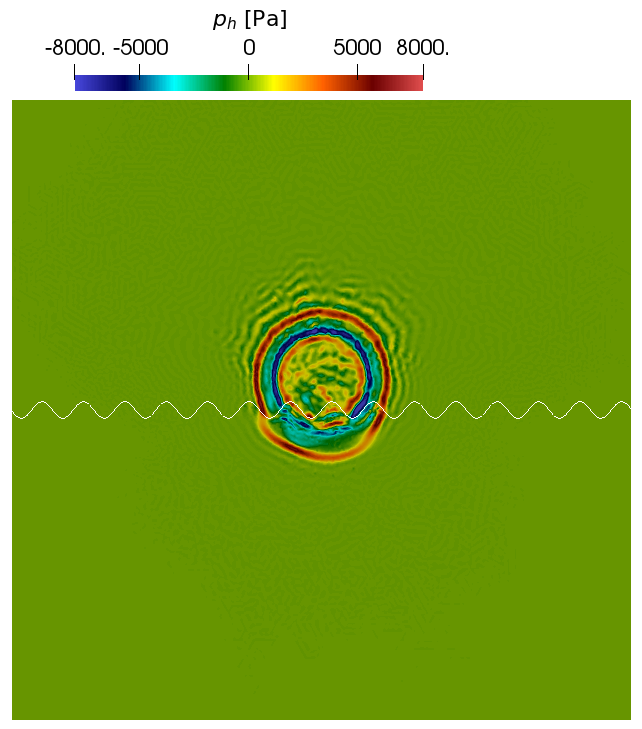}
\includegraphics[width=0.3\textwidth]{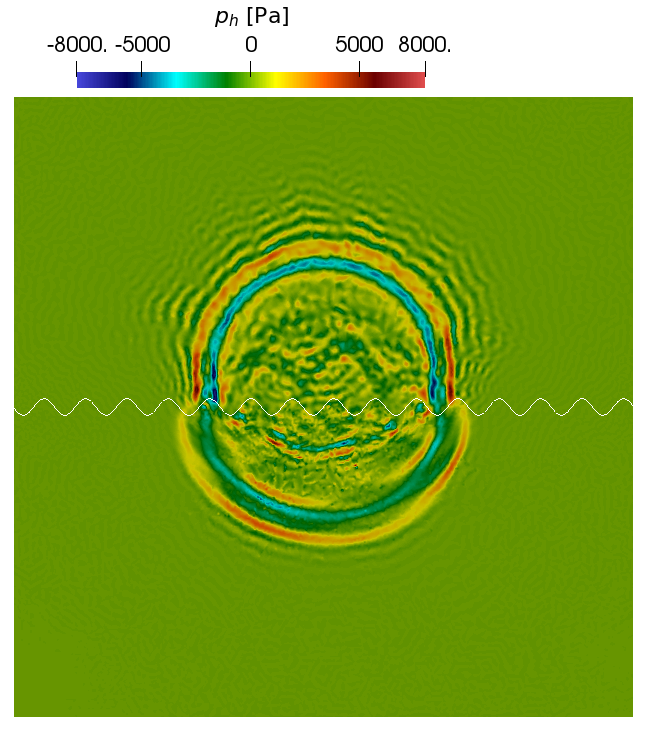}
\includegraphics[width=0.3\textwidth]{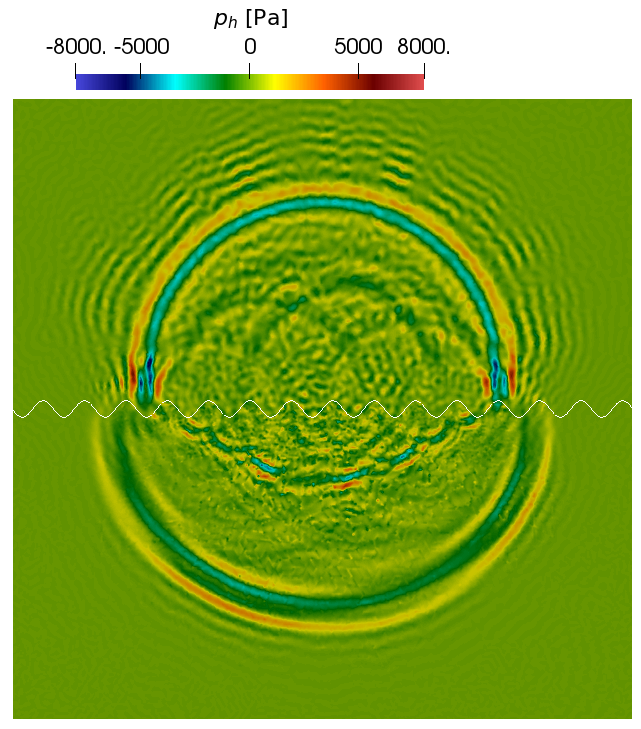}
\includegraphics[width=0.3\textwidth]{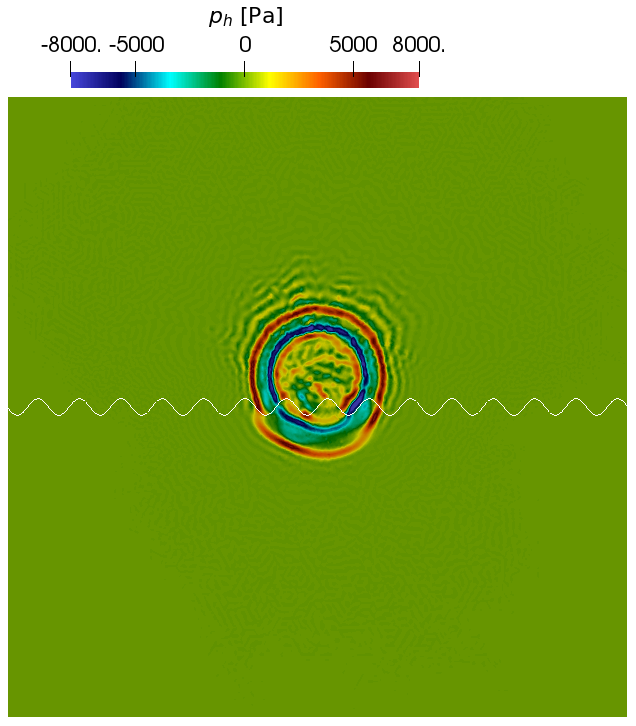}
\includegraphics[width=0.3\textwidth]{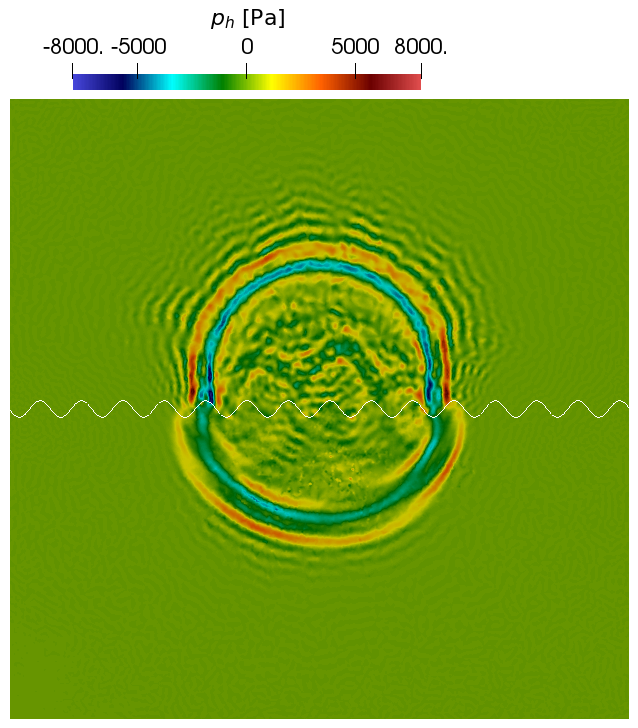}
\includegraphics[width=0.3\textwidth]{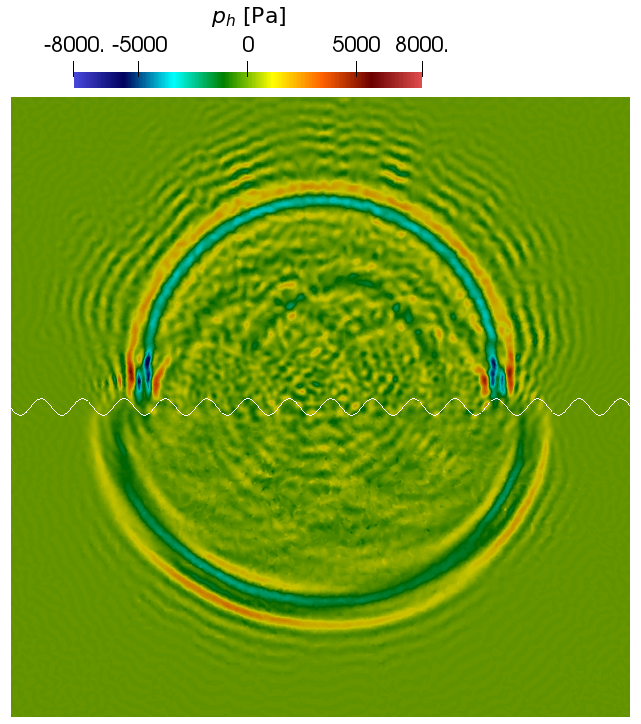}
\caption[Test case 3.   Pressure solutions at four time instants]{Test case 3.   Computed pressure $p_h$ at the time instants $t=0.2$~s (left), $t=0.4$~s (center) and $t=0.6$~s (right) with $\Delta t = 10^{-3}\ \rm s$:  $\eta=0$ (top line) $\eta=0.0015$ (bottom line).}
\label{fig::test4}
\end{figure}
%%%%%%%%%%%%%%%%%%%%
\noindent
In Figure \ref{fig::test4} we show the propagation of the discrete pressure at the time instants $t=0.2, 0.4$~s and $t=0.6$~s. Observe how the sinusoidal interface contributes to the diffraction of the acoustic wave in the poroelastic domain. This effect is more relevant when the viscosity is null while for $\eta=0.0015$ the diffracted waves are attenuated in the poroelastic domain.  In particular, we can observe the main wave front traveling towards the rigid walls of the domain followed by waves having smaller amplitude originated by the sinusoidal shape of the contact boundary. 
%Moreover, at time instant $t=1\ \rm s$, it can be noticed that the wavefront has reached the boundary of the poroelastic domain and it is reflected back because of the Dirichlet conditions.
\section{Conclusions}\label{sec::conclusions}
In this work we have presented and analyzed a PolyDG approximation to the coupled poro-elasto-acoustic problem on polygonal and polyhedral grids. \MB{Well-posedness of the continuous problem has been established by employing the semigroup theory.} We a have proved a stability result for both the continuous and the semi-discrete formulations together with a priori $hp$-version error estimates for the semi-discrete solution in a suitable energy norm. Finally, a wide set of two-dimensional numerical simulations have been carried out.

\section*{Acknowledgments}
PFA, MB and IM are members of the INdAM Research group GNCS and this work is partially funded by INdAM-GNCS. 
PFA has been partially funded by the research project PRIN17, n.201744KLJL funded by MIUR.
The work of MB has been funded by the European Commission through the H2020-MSCA-IF-EF project PDGeoFF (Grant no. 896616). PA acknowledges the H2020-MSCA-IF-EF European Commission research grant no. 896616   (project PDGeoFF).

\appendix
\section{Theoretical results}\label{appendix}
\MB{The existence and uniqueness of the solution to problem \eqref{system} as well as some technical results instrumental for the stability and error analysis are presented below.}

\MB{We establish the existence and uniqueness result in the framework of the Hille--Yosida theory by combining and adapting the arguments of \cite[Theorem 3.1]{bonaldi} and \cite[Section 5.2]{ezziani} where the elasto-acoustic coupling and the poroelastic problem were analyzed, respectively.
To do so, we additionally define the spaces 
$\bm H_\mathbb{C}^\Delta(\Omega_p)=\{\bm v\in\bm L^2(\Omega_p):\nabla\cdot(\mathbb{C}:\bm\epsilon(\bm v))\in\bm L^2(\Omega_p)\}$,
$\bm H^{\nabla}(\Omega_p)=\{\bm v\in\bm L^2(\Omega_p):\nabla (\nabla \cdot \bm  v) \in\bm L^2(\Omega_p)\}$, and $H^\Delta(\Omega_a)=\{v\in L^2(\Omega_a):\Delta v\in L^2(\Omega_a)\}$.}
\begin{theorem}[Existence and uniqueness of \eqref{system}]\label{thm:existenceuniqueness}
Assume that the initial data have the following regularity:
$\bm u_0\in\bm H^\Delta_\mathbb{C}(\Omega_p)\cap\bm H^1_0(\Omega_p)$, 
$ \bm u_1\in \bm H^1_0(\Omega_p)$,
$\bm w_0\in \MB{\bm{W}_{\kint}}\cap\bm H^{\nabla}(\Omega_p)$, 
$\bm w_1\in \MB{\bm{W}_{\kint}}$,
$\varphi_0 \in  H^\Delta(\Omega_a)\cap H_0^1(\Omega_a)$,
$\varphi_1\in H^1_0(\Omega_a)$, 
and that the source terms are such that $\bm f_p\in C^1([0,T];\bm L^2(\Omega_p))$, $\bm g_p\in C^1([0,T];\bm L^2(\Omega_p))$ and $f_a\in C^1([0,T];L^2(\Omega_a))$. Then, problem \eqref{system} admits a unique strong solution $(\bm u,\bm w,\varphi)$ s.t.
\begin{alignat*}{5}
&\bm u&&\in C^2([0,T];\bm L^2(\Omega_p)) \cap
C^1([0,T];\bm H^1_0(\Omega_p)) \cap
C^0([0,T];\bm H^\Delta_\mathbb{C}(\Omega_p) \cap\bm H^1_0(\Omega_p)),\\
&\bm w&&\in C^2([0,T];\bm L^2(\Omega_p))\cap
C^1([0,T];\MB{\bm{W}_{\kint}}) \cap
C^0([0,T];\bm H^{\nabla}(\Omega_p) \cap\MB{\bm{W}_{\kint}}),\\
&\varphi&&\in C^2([0,T]; L^2(\Omega_a)) \cap
C^1([0,T]; H^1_0(\Omega_a)) \cap
C^0([0,T]; H^\Delta(\Omega_a) \cap H^1_0(\Omega_a)).
\end{alignat*}
\end{theorem}
\begin{proof}%{(\ref{thm:existenceuniqueness})}
Let $\bm v=\dot{\bm u}$, $\bm z=\dot{\bm w}$, $\lambda=\dot{\varphi}$, and $\mathcal{U}=(\bm u,\bm v,\bm w,\bm z,\varphi,\lambda).$ We introduce the Hilbert space $\RRR{\mathbb{V}}=\bm H^1_0(\Omega_p)\times\bm L^2(\Omega_p)\times\MB{{\bm W}_{\kint}}\times\bm L^2(\Omega_p)\times H^1_0(\Omega_a)\times L^2(\Omega_a)$,
equipped with the scalar product
\RRR{
\begin{multline*}
(\mathcal{U}_1,\mathcal{U}_2)_{\mathbb{V}}=
(\rho \bm v_1 + \rho_f \bm z_1,\bm v_2)_{\Omega_p}+
(\rho_f \bm v_1 + \rho_w \bm z_1,\bm z_2)_{\Omega_p}+
(\rho_ac^{-2}\lambda_1,\lambda_2)_{\Omega_a} 
\\+
\left(\mathbb{C}:\bm\epsilon(\bm u_1),\bm\epsilon(\bm u_2)\right)_{\Omega_p}+
\left(m\nabla\cdot(\beta\bm u_1+\bm w_1),\nabla\cdot(\beta\bm u_2+\bm w_2)\right)_{\Omega_p}\\+ (\rho_a\nabla\varphi_1,\nabla\varphi_2)_{\Omega_a} + (\eta k^{-1} \bm w_1, \bm w_2)_{\Omega_p} + (\zeta(\tau) \bm w_1 \cdot \bm n_p, \bm w_2 \cdot \bm n_p)_{\Gamma_I},
\end{multline*}
}
%
%\begin{multline*}
%(\mathcal{U}_1,\mathcal{U}_2)_{\mathbb{H}}=(\rho\bm v_1,\bm v_2)_{\Omega_p}+
%(\rho_f\bm z_1,\bm v_2)_{\Omega_p}+
%(\rho_f\bm v_1,\bm z_2)_{\Omega_p}+
%(\rho_w\bm z_1,\bm z_2)_{\Omega_p}\\+
%\left(\mathbb{C}:\bm\epsilon(\bm u_1),\bm\epsilon(\bm u_2)\right)_{\Omega_p}+
%\left(m\nabla\cdot(\beta\bm u_1+\bm w_1),\nabla\cdot(\beta\bm u_2+\bm w_2)\right)_{\Omega_p}\\+
%(\rho_ac^{-2}\lambda_1,\lambda_2)_{\Omega_a}+
%(\rho_a\nabla\varphi_1,\nabla\varphi_2)_{\Omega_a}.
%\end{multline*}
%Notice that the sum of the first four terms on the right hand side of the above equation is equal to
%		\begin{gather*}
%			(\rho\bm v_1,\bm v_2)_{\Omega_p}+
%			(\rho_f\bm z_1,\bm v_2)_{\Omega_p}+
%			(\rho_f\bm v_1,\bm z_2)_{\Omega_p}+
%			(\rho_w\bm z_1,\bm z_2)_{\Omega_p}=\\
%			(\widetilde{\rho}_s\bm v_1,\bm v_2)_{\Omega_p}+
%			(\widetilde{\rho}_w\bm z_1,\bm z_2)_{\Omega_p}+
%			\left(\rho_f[\phi^{1/2}\bm v_1+\phi^{-1/2}\bm z_1],\phi^{1/2}\bm v_2+\phi^{-1/2}\bm z_2\right),
%		\end{gather*}
%$			(\widetilde{\rho}_s\bm v_1,\bm v_2)_{\Omega_p}+
%			(\widetilde{\rho}_w\bm z_1,\bm z_2)_{\Omega_p}+
%			\left(\rho_f[\phi^{1/2}\bm v_1+\phi^{-1/2}\bm z_1],\phi^{1/2}\bm v_2+\phi^{-1/2}\bm z_2\right)$,
where \MB{${\bm W}_{\kint}$ is defined in \eqref{eq:W_kint}}.
We remark that the scalar product is positive definite in \RRR{$\mathbb{V}\times \mathbb{V}$}, \RRR{cf. \cite{ezziani}}. 
We define the operator 
\begin{align*}
& A:\mathcal{D}(A)\subset\mathbb{V}\rightarrow\mathbb{V}
&& 
A\mathcal{U}=
\begin{pmatrix}
-\bm v \\
-\frac{1}{\rho_T}\left(\rho_w\nabla\cdot\bm\sigma+\frac{\rho_f\eta}{k}\bm z+\rho_f\nabla p\right)\\
-\bm z\\
\frac{1}{\rho_T}\left(\rho_f\nabla\cdot\bm\sigma+\frac{\rho\eta}{k}\bm z+\rho\nabla p\right)\\
-\lambda\\
-c^2 \MB{\rho_a^{-1}\nabla\cdot(\rho_a\nabla\varphi)}\\
\end{pmatrix},
\end{align*}
with $\rho_T=\rho\rho_w-\rho_f^2>0$, and 
\begin{align*}
\mathcal{D}(A)=\{&\mathcal{U}\in\mathbb{V}: \bm u\in\bm H^\Delta_\mathbb{C}(\Omega_p), \bm v\in\bm H^1_0(\Omega_p), \bm w\in \bm H^{\nabla}(\Omega_p),
\bm z\in \MB{{\bm W}_{\kint}}, \\
&\varphi\in H^\Delta(\Omega_a), \lambda\in H_0^1(\Omega_a);\;
(\bm \sigma+\rho_a\lambda\bm I)\cdot\bm n_p=\bm 0,\hspace{0.1cm}\text{on }\Gamma_I,\\
&\MB{\kint(p-\rho_a\lambda) -(1-\kint){\bm z}\cdot{\bm n}_p = 0}, \hspace{0.1cm}\text{on }\Gamma_I,\; (\nabla\varphi+\bm v+\bm z)\cdot\bm n_p=0,\hspace{0.1cm}\text{on }\Gamma_I\}.
\end{align*}
With the above notation, problem \eqref{system} can be reformulated as follows:
given $\mathcal{F}\in C^1([0,T];\mathbb{V})$ defined as $\mathcal{F}(t)= (\bm 0, (\rho_w\bm f_p-\rho_f\bm g_p)/\rho_T,	\bm 0, (\rho\bm g_p-\rho_f\bm f_p)/\rho_T, 0, c^2 f_a)$  and $\mathcal{U}_0\in\mathcal{D}(A)$, find $\mathcal{U}\in C^1([0,T];\mathbb{V})\cap C^0([0,T];\mathcal{D}(A))$ such that
\[
\begin{cases}
\dfrac{\rm d\mathcal{U}}{\rm dt}+A\mathcal{U}(t)=\mathcal{F}(t),\hspace{1cm}t\in(0,T],\\
\mathcal{U}(0)=\mathcal{U}_0.
\end{cases}
\]
Owing to the Hille--Yosida theorem, the above problem is well-posed \MB{provided  the existence of $\mu > 0$ such that $A+\mu I$ is maximal monotone}, \RRR{i.e. $(A\mathcal{U},\mathcal{U})_{\mathbb{V}} + \mu \|\mathcal{U}\|^2_{\mathbb{V}} \geq 0$ $\forall\ \mathcal{U}\in\mathcal{D}(A)$} \MB{and $A+\mu I:\mathcal{D}(A)\to\mathbb{V}$ is onto}.
The first condition follows from the definition of the scalar product in $\mathbb{V}$, the definition of $\mathcal{D}(A)$ and integration by parts:
\begin{align*}
(A\mathcal{U},\mathcal{U})_\mathbb{V}=&
-\left(\frac{\rho\rho_w}{\rho_T}\nabla\cdot\bm\sigma+
\frac{\rho\rho_f}{\rho_T}\frac{\eta}{k}\bm z+
\frac{\rho\rho_f}{\rho_T}\nabla p,\bm v\right)_{\Omega_p}
-\left(\mathbb{C}:\bm\epsilon(\bm v),\bm\epsilon(\bm u)\right)_{\Omega_p}\\
&+\left(\frac{\rho_f^2}{\rho_T}\nabla\cdot\bm\sigma+
\frac{\rho\rho_f}{\rho_T}\frac{\eta}{k}\bm z+
\frac{\rho\rho_f}{\rho_T}\nabla p,\bm v\right)_{\Omega_p}
-\MB{(\nabla\cdot\rho_a\nabla\varphi,\lambda)_{\Omega_a}}
\\&
-\left(\frac{\rho_f\rho_w}{\rho_T}\nabla\cdot\bm\sigma+
\frac{\rho_f^2}{\rho_T}\frac{\eta}{k}\bm z+
\frac{\rho_f^2}{\rho_T}\nabla p,\bm z\right)_{\Omega_p}
-(\rho_a\nabla\lambda,\nabla\varphi)_{\Omega_a} \\&
+\left(\frac{\rho_w\rho_f}{\rho_T}\nabla\cdot\bm\sigma+
\frac{\rho_w\rho}{\rho_T}\frac{\eta}{k}\bm z+
\frac{\rho_w\rho}{\rho_T}\nabla p,\bm z\right)_{\Omega_p} \RRR{-(\eta k^{-1}\bm z, \bm w)_{\Omega_p}}\\
&-\left(m\nabla\cdot(\beta\bm v+\bm z),\nabla\cdot(\beta\bm u+\bm w)\right)_{\Omega_p} 
\RRR{-(\zeta(\tau)\bm z\cdot \bm n_p, \bm w\cdot \bm n_p)_{\Gamma_I}}\\
&=\norm{(\eta/k)^{\frac12}\bm z}_{\Omega_p}^2  
\MB{+\norm{\zeta(\tau)^{\frac12}\bm z\cdot \bm n_p}_{\Gamma_I}^2}
 \RRR{-((\eta/k) \bm z, \bm w)_{\Omega_p}\hspace{-1mm}
 -\hspace{-0.5mm}(\zeta(\tau)\bm z\cdot \bm n_p, \bm w\cdot \bm n_p)_{\Gamma_I}},
\end{align*}
where we have also used that all the terms on $\Gamma_I$ (\MB{except $\norm{\zeta(\tau)^{1/2}\bm z\cdot \bm n_p}_{\Gamma_I}^2$ for $\kint\in(0,1)$) vanish}. Thus, by choosing \RRR{$\mu \geq 1/2$,}
and applying the Young's inequality, we obtain \RRR{$(A\mathcal{U},\mathcal{U})_{\mathbb{V}} + \mu \|\mathcal{U}\|^2_{\mathbb{V}} \geq 0$.}
Now, we prove that $A+\nu I$ is surjective for all $\nu>0$. 
%Remark that  to obtain the proof for $A+\mu I$ one can reason for $\nu = \mu +1$. 
\MB{The surjectivity of $A+\nu I$ is equivalent to verify that for any $\mathcal{F}\in\mathbb{V}$, there exists  $\mathcal{U}\in\mathcal{D}(A)$ s.t. $A\mathcal{U} +\nu \mathcal{U}=\mathcal{F}$, i.e.}
\begin{subequations} \label{eq::cont_stab}
\begin{align}
&\RRR{\nu}\bm u-\bm v=\bm{\mathcal{F}}_1,\label{eq::cont_stab_1}\\
&\RRR{\nu} \bm v-\dfrac{\rho_w}{\rho_T}\nabla\cdot\bm\sigma-\dfrac{\rho_f}{\rho_T}\dfrac{\eta}{k}\bm z-\dfrac{\rho_f}{\rho_T}\nabla p=\bm{\mathcal{F}}_2,\label{eq::cont_stab_2}\\
&\RRR{\nu} \bm w-\bm z=\bm{\mathcal{F}}_3,\label{eq::cont_stab_3}\\
&\RRR{\nu} \bm z+\dfrac{\rho_f}{\rho_T}\nabla\cdot\bm\sigma+\dfrac{\rho}{\rho_T}\dfrac{\eta}{k}\bm z+\dfrac{\rho}{\rho_T}\nabla p=\bm{\mathcal{F}}_4,\label{eq::cont_stab_4}\\
&\RRR{\nu} \varphi-\lambda= \mathcal{F}_5,\label{eq::cont_stab_5}\\
&\RRR{\nu} \lambda-\MB{c^2\rho_a^{-1}\nabla\cdot(\rho_a\nabla\varphi)}= \mathcal{F}_6.\label{eq::cont_stab_6}
\end{align}
\end{subequations}
\MB{Hence, by plugging $\bm v=\RRR{\nu} \bm u-\bm{\mathcal{F}}_1$, $\bm z= \RRR{\nu} \bm w-\bm{\mathcal{F}}_3$, and $\lambda=\RRR{\nu} \varphi-\mathcal{F}_5$ respectively in \eqref{eq::cont_stab_2}, \eqref{eq::cont_stab_4}, and \eqref{eq::cont_stab_6} and rearranging, we rewrite the previous system as}
\begin{equation*}
\begin{cases}
\RRR{\nu^2}(\rho\bm u+\rho_f\bm w) -\nabla\cdot\bm\sigma =\rho(\MB{\nu}\bm{\mathcal{F}}_1+\bm{\mathcal{F}}_2)
+\rho_f(\RRR{\nu}\bm{\mathcal{F}}_3+\bm{\mathcal{F}}_4)=\bm G_1,\\
\RRR{\nu^2}( \rho_f\bm u+\rho_w\bm w)+\dfrac{\RRR{\nu}\eta}{k}\bm w+\nabla p=\rho_f(\MB{\nu}\bm{\mathcal{F}}_1+\bm{\mathcal{F}}_2)+\rho_w(\RRR{\nu}\bm{\mathcal{F}}_3+\bm{\mathcal{F}}_4)+\dfrac{\eta}{k}\bm{\mathcal{F}}_3=\bm G_2,\\
\MB{\nu^2}\rho_a c^{-2}\varphi-\MB{\nabla\cdot(\rho_a\nabla\varphi)}=
\rho_a c^{-2}(\MB{\nu}\mathcal{F}_5+\mathcal{F}_6)=G_3.
\end{cases}
\end{equation*}
Owing to $\bm n_p=-\bm n_a$ on $\Gamma_{I}$, equations \eqref{eq::cont_stab_1}, \eqref{eq::cont_stab_3} and \eqref{eq::cont_stab_5}, and the transmission conditions on $\Gamma_{I}$ embedded in the definition of $\mathcal{D}(A)$, the variational formulation of the above problem reads:
find $(\bm u,\bm w,\varphi)\in\bm H_0^1(\Omega_p)\times \MB{{\bm W}_{\kint}}\times H_0^1(\Omega_a)$ s.t. 
\begin{equation*}
\mathcal{A}((\bm u,\bm w,\varphi),(\bm v,\bm z,\lambda))=\mathcal{L}(\bm v,\bm z,\lambda),
\quad\text{for all }\; (\bm v,\bm z,\lambda)\in\bm H_0^1(\Omega_p)\times\MB{{\bm W}_{\kint}} \times H_0^1(\Omega_a),
\end{equation*}
with
\begin{align*}
\mathcal{A}((\bm u,\bm w,\varphi),(\bm v,\bm z,\lambda))&=\RRR{\nu^2}(\rho\bm u+\rho_f\bm w,\bm v)_{\Omega_p}+
\left(\mathbb{C}\bm\epsilon(\bm u),\bm\epsilon(\bm v)\right)_{\Omega_p}
+ \RRR{\nu^2}(\rho_f\bm u+\rho_w\bm w,\bm z)_{\Omega_p}\\&+
\left(m\nabla\cdot(\beta\bm u+\bm w),\nabla\cdot(\beta\bm v+\bm z)\right)_{\Omega_p}+
\RRR{\nu}\left(\eta k^{-1}\bm w,\bm z\right)_{\Omega_p}\\&+
\MB{\nu\left(\zeta(\tau)\bm w\cdot{\bm n}_p,\bm z\cdot{\bm n}_p\right)_{\Gamma_I}}+
\MB{\nu^2}(\rho_a c^{-2}\varphi,\lambda)_{\Omega_a}\\&
+(\rho_a\nabla\varphi,\nabla\lambda)_{\Omega_a}+
\MB{\nu(\rho_a\varphi, \bm v\cdot\bm n_p)_{\Gamma_I}}- \MB{\nu(\bm u\cdot\bm n_p,\rho_a\lambda)_{\Gamma_I}},\\ \text{and}\quad
\mathcal{L}(\bm v,\bm z,\lambda)&=(\bm G_1,\bm v)_{\Omega_p}+
(\bm G_2,\bm z)_{\Omega_p}+( G_3,\lambda)_{\Omega_a}
\MB{-(\mathcal{\bm F}_1\cdot{\bm n}_p, \rho_a\lambda)_{\Gamma_I}}\\&
+\MB{\left(\zeta(\tau)\mathcal{\bm F}_3\cdot{\bm n}_p,\bm z\cdot{\bm n}_p\right)_{\Gamma_I}
+(\rho_a\mathcal{F}_5,\bm v\cdot{\bm n}_p)_{\Gamma_I}}.
\end{align*}
\MB{The well-posedness of the previous problem follows from the Lax-Milgram Lemma, since $\mathcal{A}$ is coercive for all $\nu>0$}. In addition, \MB{owing to \eqref{eq::cont_stab_2}, \eqref{eq::cont_stab_4}, and \eqref{eq::cont_stab_6}, we infer that}  $\bm u\in\bm H_\mathbb{C}^\Delta(\Omega_p)\cap\bm H_0^1(\Omega_p)$, $\bm w\in {\bm H^{\nabla}} (\Omega_p)\cap\MB{{\bm W}_{\kint}}$, and $\varphi\in H^\Delta(\Omega_a)\cap H_0^1(\Omega_a)$. Moreover, this gives $(\bm v,\bm z,\lambda)\in\bm H_0^1(\Omega_p)\times\MB{{\bm W}_{\kint}}\times H_0^1(\Omega_a)$ due to \eqref{eq::cont_stab_1}, \eqref{eq::cont_stab_3}, and \eqref{eq::cont_stab_5}. Then $\mathcal{U}\in\mathcal{D}(A)$ and the proof is complete.
\end{proof}

%%%%%%%%%%%%%
We conclude the Appendix with some technical results needed in the analysis.
The first Lemma \MB{hinges on} Assumption \ref{ass::regular} and the trace inverse inequality \eqref{eq::traceinv}. 
\begin{lemma}
The following bounds hold:
\begin{align}
\norm{ \alpha^{-1/2} \llbrace \bm \sigma_h ( \bm v)\rrbrace}_{\mathcal{F}_h^p}\lesssim&
\frac{1}{\sqrt{c_1}} \norm{\mathbb{C}^{1/2}\bm\epsilon_h(\bm v)}_{\Omega_p} && \forall\bm v\in \bm V_h^p,\label{eq::f2o1}\\
\norm{\chi^{-1/2}\llbrace\rho_a\nabla_h\psi\rrbrace}_{\mathcal{F}_h^a}\lesssim &  \frac{1}{\sqrt{c_2}} \norm{\rho_a^{1/2}\nabla_h\psi}_{\Omega_a}&&\forall\psi\in V_h^a,\label{eq::f2o2} \\ 
\norm{ \gamma^{-1/2} \llbrace m \nabla_h\cdot\bm z  \rrbrace }_{\mathcal{F}_h^\star} \lesssim &
\frac{1}{\sqrt{c_3}} \norm{m^{1/2}\nabla_h\cdot\bm z}_{\Omega_p}&&\forall\bm z\in\bm V_h^p,\label{eq::f2o3}
\end{align}
where $c_1$, $c_2$ and $c_3$ are the constants appearing in \eqref{eq::stab_1}, \eqref{eq::stab_2} and \eqref{eq::stab_3}, respectively.
\label{lem::stab_fun}
\end{lemma}
\MB{The following Lemma establishes the coercivity and boundedness of the discrete bilinear form $\mathcal{A}_h$ defined in \eqref{eq:bilinear_Ah}}.
\begin{lemma}
Let Assumptions \ref{ass::regular} and \ref{ass::3} be satisfied. Then,
\MB{\begin{align*}
& \mathcal{A}_h^e(\bm u,\bm v)\lesssim 
\norm{\bm u}_{\rm dG,e}
\norm{\bm v}_{\rm dG,e}
&&  \mathcal{A}_h^e(\bm u,\bm u)\gtrsim 
\norm{\bm u}_{\rm dG,e}^2
&& \forall \bm u,\bm v\in\bm V_h^p,\\
& \mathcal{A}_h^p(\bm u,\bm v)\lesssim 
|\bm u|_{\rm dG,p}
|\bm v|_{\rm dG,p}
&& 
\mathcal{A}_h^p(\bm u,\bm u)\gtrsim 
|\bm u|_{\rm dG,p}^2
&& \forall \bm u,\bm v\in\bm V_h^p,\\
& \mathcal{A}_h^a(\varphi,\psi)\lesssim
\norm{\varphi}_{\rm dG,a}
\norm{\psi}_{\rm dG,a}
&&\mathcal{A}_h^a(\varphi,\varphi)\gtrsim 
\norm{\varphi}_{\rm dG,a}^2
&& \forall \varphi,\psi\in V_h^a,\\
&\mathcal{A}_h^e(\bm u,\bm v)\lesssim
\trinorm{\bm u}_{\rm dG,e}
\norm{\bm v}_{\rm dG,e}
&&\qquad\forall \bm u \in\bm H^2(\mathcal{T}_h^p) 
&&\forall \bm v \in \bm{V}_h^p,\\
&\mathcal{A}_h^a(\varphi,\psi)\lesssim
\trinorm{\varphi}_{\rm dG,a}
\norm{\psi}_{\rm dG,a}
&&\qquad\forall \varphi \in H^2(\mathcal{T}_h^a) \quad
&&\forall \psi \in \bm{V}_h^a,\\
&\mathcal{A}_h^p(\bm w,\bm z)\lesssim
\trinorm{\bm w}_{\rm dG,p}
|\bm z|_{\rm dG,p}
&&\qquad\forall \bm w \in\bm H^2(\mathcal{T}_h^p)
&&\forall \bm z \in \bm{V}_h^p.
\end{align*}}
\label{lem::cont}
The coercivity bounds hold provided that the stability parameters $c_1$, $c_2$ and $c_3$ appearing in \eqref{eq::stab_1},\eqref{eq::stab_2} and \eqref{eq::stab_3}, respectively, are chosen sufficiently large.
\end{lemma}
\begin{proof}
\RRR{The proof is based on employing Lemma \ref{lem::stab_fun} and standard arguments. See also \cite{AntoniettiMazzieri2018} and \cite[Lemma A.2]{bonaldi}.}
\end{proof}

\bibliographystyle{abbrv}

\end{document}

%% file: sealed_pores_h_3.tikz
\definecolor{mycolor1}{rgb}{0.00000,0.44700,0.74100}%
\definecolor{mycolor2}{rgb}{0.85000,0.32500,0.09800}%
\definecolor{mycolor3}{rgb}{0.92900,0.69400,0.12500}%
\begin{tikzpicture}

\begin{axis}[%
width=0.33\textwidth,
height=0.33\textwidth,
scale only axis,
xmode=log,
xmin=0.05,
xmax=0.6,
xminorticks=true,
xlabel={$h$},
xmajorgrids,
xminorgrids,
ymode=log,
ymin=0.00005,
ymax=10,
yminorticks=true,
ymajorgrids,
yminorgrids,
title={Sealed pores ($\kint=0$)},
axis background/.style={fill=white},
legend style={at={(0.03,0.97)},anchor=north west,legend cell align=left,align=left,draw=white!15!black}
]
\addplot [color=mycolor1,solid,line width=2.0pt,mark=,mark options={solid}]
  table[row sep=crcr]{%
     3.706747546840842e-01     2.887944277271775e-02\\
     2.496245525522082e-01     9.672085562071669e-03\\
     1.825964801200949e-01     3.411462066956723e-03\\
     1.308540523483373e-01     1.550759125008310e-03\\
     1.012898589583347e-01     4.599299098164105e-04\\
};
\addlegendentry{$\|{\bm e}_h^u\|_{\rm dG,e}$};

\addplot [color=mycolor2,solid,line width=2.0pt,mark=o,mark options={solid}]
  table[row sep=crcr]{%
     3.706747546840842e-01     1.013626396945255e-02\\
     2.496245525522082e-01     2.820175535879977e-03\\
     1.825964801200949e-01     9.763779975987207e-04\\
     1.308540523483373e-01     4.386931797954194e-04\\
     1.012898589583347e-01     1.453838558086529e-04\\
};
\addlegendentry{$|{\bm e}_h^w |_{\rm dG,p}$};

\addplot [color=mycolor3,solid,line width=2.0pt,mark=+,mark options={solid}]
  table[row sep=crcr]{%
     3.706747546840842e-01     2.368199794649497e-02\\
     2.496245525522082e-01     9.385048409630980e-03\\
     1.825964801200949e-01     3.194418382690740e-03\\
     1.308540523483373e-01     1.070184966507638e-03\\
     1.012898589583347e-01     3.806153214534772e-04\\
};
\addlegendentry{$\|{e}_h^\varphi\|_{\rm dG,a}$};

\addplot [color=red,solid,line width=2.0pt,mark=diamond,mark options={solid}]
  table[row sep=crcr]{%
     3.706747546840842e-01     4.911320070131794e-02\\
     2.496245525522082e-01     1.619696269749003e-02\\
     1.825964801200949e-01     5.654209292461577e-03\\
     1.308540523483373e-01     2.381899400157443e-03\\
     1.012898589583347e-01     7.357987262702295e-04\\
};
\addlegendentry{$\| (\bm e_h^u,\bm e_h^w,e_h^\varphi)\|_{\rm E}$};

\addplot [color=black,dashed,line width=1.5pt]
  table[row sep=crcr]{%
0.35329917294355	0.00440989107359586\\
0.25283535049205	0.00161626804205538\\
0.187501752721944	0.00065919817341206\\
0.136755485644054	0.000255760968810901\\
};
\addlegendentry{$h^{3}$};

\end{axis}

\end{tikzpicture}%

%% file: sealed_pores_N.tikz
\definecolor{mycolor1}{rgb}{0.00000,0.44700,0.74100}%
\definecolor{mycolor2}{rgb}{0.85000,0.32500,0.09800}%
\definecolor{mycolor3}{rgb}{0.92900,0.69400,0.12500}%
\begin{tikzpicture}

\begin{axis}[%
width=0.33\textwidth,
height=0.33\textwidth,
scale only axis,
xmin=0,
xmax=6,
xlabel={$p$},
xmajorgrids,
ymode=log,
ymin=1e-8,
ymax=1,
yminorticks=true,
ymajorgrids,
yminorgrids,
axis background/.style={fill=white},
title={Sealed pores ($\kint=0$)},
legend style={legend cell align=left,align=left,draw=white!15!black}
]
\addplot [color=mycolor1,solid,line width=2.0pt,mark=,mark options={solid}]
  table[row sep=crcr]{%
     1.000000000000000e+00     1.177991278200456e-02\\
     2.000000000000000e+00     6.585008750236634e-04\\
     3.000000000000000e+00     4.127934245185949e-05\\
     4.000000000000000e+00     2.643189290988839e-06\\
     5.000000000000000e+00     1.044565400333243e-07\\
};
\addlegendentry{$\|\bm{e}_h^u \|_{\Omega_p}$};

\addplot [color=mycolor2,solid,line width=2.0pt,mark=o,mark options={solid}]
  table[row sep=crcr]{%
     1.000000000000000e+00     8.712743534229246e-03\\
     2.000000000000000e+00     5.314005447541475e-04\\
     3.000000000000000e+00     2.972507420533812e-05\\
     4.000000000000000e+00     1.929916490371516e-06\\
     5.000000000000000e+00     7.319645410146761e-08\\
};
\addlegendentry{$\|\bm{e}_h^w \|_{\Omega_p}$};

\addplot [color=mycolor3,solid,line width=2.0pt,mark=+,mark options={solid}]
  table[row sep=crcr]{%
     1.000000000000000e+00     1.519635533476007e-02\\
     2.000000000000000e+00     1.172626301186908e-03\\
     3.000000000000000e+00     7.010299726268865e-05\\
     4.000000000000000e+00     3.312892162770180e-06\\
     5.000000000000000e+00     1.105832075393239e-07\\
};
\addlegendentry{$\|e_h^\varphi  \|_{\Omega_a}$};

\addplot [color=black,dashed,line width=1.5pt]
  table[row sep=crcr]{%
  1   5.502322005640724e-02\\
  2   3.027554745375815e-03\\
  3   1.665858109876335e-04\\
  4   9.166087736247617e-06\\
  5   5.043476625678880e-07\\
};
\addlegendentry{$e^{-3\textcolor{red}{p}}$};   

\end{axis}

\end{tikzpicture}%

%% file: intermediate_pores_h_3.tikz
\definecolor{mycolor1}{rgb}{0.00000,0.44700,0.74100}%
\definecolor{mycolor2}{rgb}{0.85000,0.32500,0.09800}%
\definecolor{mycolor3}{rgb}{0.92900,0.69400,0.12500}%
\begin{tikzpicture}

\begin{axis}[%
width=0.33\textwidth,
height=0.33\textwidth,
scale only axis,
xmode=log,
xmin=0.05,
xmax=0.6,
xminorticks=true,
xlabel={$h$},
xmajorgrids,
xminorgrids,
ymode=log,
ymin=0.00005,
ymax=10,
yminorticks=true,
ymajorgrids,
yminorgrids,
title={Imperfect pores ($\kint=\frac{1}{2}$)},
axis background/.style={fill=white},
legend style={at={(0.03,0.97)},anchor=north west,legend cell align=left,align=left,draw=white!15!black}
]
\addplot [color=mycolor1,solid,line width=2.0pt,mark=,mark options={solid}]
  table[row sep=crcr]{%
     3.706747546840842e-01     2.862199823762073e-02\\
     2.496245525522082e-01     9.324317669188271e-03\\
     1.902011899904271e-01     3.728047015204185e-03\\
     1.362344921556706e-01     1.377796066427069e-03\\
     9.582420607016751e-02     4.468751507371053e-04\\
     };
\addlegendentry{$\|{\bm e}_h^u\|_{\rm dG,e}$};

\addplot [color=mycolor2,solid,line width=2.0pt,mark=o,mark options={solid}]
  table[row sep=crcr]{%
     3.706747546840842e-01     9.023705782478268e-03\\
     2.496245525522082e-01     2.589406957430076e-03\\
     1.902011899904271e-01     1.033030983576455e-03\\
     1.362344921556706e-01     4.114393850741680e-04\\
     9.582420607016751e-02     1.391784201366490e-04\\
};
\addlegendentry{$|{\bm e}_h^w |_{\rm dG,p}$};

\addplot [color=mycolor3,solid,line width=2.0pt,mark=+,mark options={solid}]
  table[row sep=crcr]{%
     3.706747546840842e-01     2.368199794646264e-02\\
     2.496245525522082e-01     9.385048409628954e-03\\
     1.902011899904271e-01     3.252547390191706e-03\\
     1.362344921556706e-01     1.129214642458998e-03\\
     9.582420607016751e-02     3.704348089739504e-04\\
};
\addlegendentry{$\|{e}_h^\varphi\|_{\rm dG,a}$};

\addplot [color=red,solid,line width=2.0pt,mark=diamond,mark options={solid}]
  table[row sep=crcr]{%
     3.706747546840842e-01     4.799049989790460e-02\\
     2.496245525522082e-01     1.583766682551031e-02\\
     1.902011899904271e-01     5.926005769797757e-03\\
     1.362344921556706e-01     2.180147933235778e-03\\
     9.582420607016751e-02     7.178847165737432e-04\\
};
\addlegendentry{$\| (\bm e_h^u,\bm e_h^w,e_h^\varphi)\|_{\rm E}$};

\addplot [color=black,dashed,line width=1.5pt]
  table[row sep=crcr]{%
0.35329917294355	0.00440989107359586\\
0.25283535049205	0.00161626804205538\\
0.187501752721944	0.00065919817341206\\
0.136755485644054	0.000255760968810901\\
};
\addlegendentry{$h^{3}$};

\end{axis}

\end{tikzpicture}%

%% file: intermediate_pores_N.tikz
\definecolor{mycolor1}{rgb}{0.00000,0.44700,0.74100}%
\definecolor{mycolor2}{rgb}{0.85000,0.32500,0.09800}%
\definecolor{mycolor3}{rgb}{0.92900,0.69400,0.12500}%
\begin{tikzpicture}

\begin{axis}[%
width=0.33\textwidth,
height=0.33\textwidth,
scale only axis,
xmin=0,
xmax=6,
xlabel={$p$},
xmajorgrids,
ymode=log,
ymin=1e-8,
ymax=1,
yminorticks=true,
ymajorgrids,
yminorgrids,
axis background/.style={fill=white},
title={Imperfect pores ($\kint=\frac{1}{2}$)},
legend style={legend cell align=left,align=left,draw=white!15!black}
]
\addplot [color=mycolor1,solid,line width=2.0pt,mark=,mark options={solid}]
  table[row sep=crcr]{%
     1.000000000000000e+00     1.168413773164724e-02\\
     2.000000000000000e+00     9.110036141239667e-04\\
     3.000000000000000e+00     5.409251244231194e-05\\
     4.000000000000000e+00     2.577289301717289e-06\\
     5.000000000000000e+00     1.372053461737597e-07\\
};
\addlegendentry{$\|\bm{e}_h^u \|_{\Omega_p}$};

\addplot [color=mycolor2,solid,line width=2.0pt,mark=o,mark options={solid}]
  table[row sep=crcr]{%
     1.000000000000000e+00     8.817449338351227e-03\\
     2.000000000000000e+00     6.330435375557821e-04\\
     3.000000000000000e+00     3.993770183520082e-05\\
     4.000000000000000e+00     1.899681832045270e-06\\
     5.000000000000000e+00     1.010544545661184e-07\\
};
\addlegendentry{$\|\bm{e}_h^w \|_{\Omega_p}$};

\addplot [color=mycolor3,solid,line width=2.0pt,mark=+,mark options={solid}]
  table[row sep=crcr]{%
     1.000000000000000e+00     1.558506121581035e-02\\
     2.000000000000000e+00     1.313285276986313e-03\\
     3.000000000000000e+00     6.764476445360933e-05\\
     4.000000000000000e+00     2.829817020328395e-06\\
     5.000000000000000e+00     1.045590434349055e-07\\
};
\addlegendentry{$\|e_h^\varphi  \|_{\Omega_a}$};

\addplot [color=black,dashed,line width=1.5pt]
  table[row sep=crcr]{%
  1   5.502322005640724e-02\\
  2   3.027554745375815e-03\\
  3   1.665858109876335e-04\\
  4   9.166087736247617e-06\\
  5   5.043476625678880e-07\\
};
\addlegendentry{$e^{-3\textcolor{red}{p}}$};   

\end{axis}

\end{tikzpicture}%

%% file: open_pores_h_3.tikz
\definecolor{mycolor1}{rgb}{0.00000,0.44700,0.74100}%
\definecolor{mycolor2}{rgb}{0.85000,0.32500,0.09800}%
\definecolor{mycolor3}{rgb}{0.92900,0.69400,0.12500}%
\begin{tikzpicture}

\begin{axis}[%
width=0.33\textwidth,
height=0.33\textwidth,
scale only axis,
xmode=log,
xmin=0.05,
xmax=0.6,
xminorticks=true,
xlabel={$h$},
xmajorgrids,
xminorgrids,
ymode=log,
ymin=0.00005,
ymax=10,
yminorticks=true,
ymajorgrids,
yminorgrids,
title={Open pores ($\kint=1$)},
axis background/.style={fill=white},
legend style={at={(0.03,0.97)},anchor=north west,legend cell align=left,align=left,draw=white!15!black}
]
\addplot [color=mycolor1,solid,line width=2.0pt,mark=,mark options={solid}]
  table[row sep=crcr]{%
     3.581723141739799e-01     3.229608608138437e-02\\
     2.564569543970479e-01     1.167903320077093e-02\\
     1.870455127132988e-01     5.142054723261010e-03\\
     1.363219521552116e-01     1.383648437495561e-03\\
     9.671538020774045e-02     5.268878053014214e-04\\
     };
\addlegendentry{$\|{\bm e}_h^u\|_{\rm dG,e}$};

\addplot [color=mycolor2,solid,line width=2.0pt,mark=o,mark options={solid}]
  table[row sep=crcr]{%
     3.581723141739799e-01     9.313651958443980e-03\\
     2.564569543970479e-01     3.018565879117654e-03\\
     1.870455127132988e-01     1.402815181954107e-03\\
     1.363219521552116e-01     4.545727277168591e-04\\
     9.671538020774045e-02     1.637368956095560e-04\\
};
\addlegendentry{$|{\bm e}_h^w |_{\rm dG,p}$};

\addplot [color=mycolor3,solid,line width=2.0pt,mark=+,mark options={solid}]
  table[row sep=crcr]{%
     3.581723141739799e-01     2.791810699918568e-02\\
     2.564569543970479e-01     8.400050390138643e-03\\
     1.870455127132988e-01     2.868330026486586e-03\\
     1.363219521552116e-01     1.107045433246193e-03\\
     9.671538020774045e-02     3.918972446118939e-04\\
};
\addlegendentry{$\|{e}_h^\varphi\|_{\rm{dG,a}}$};

\addplot [color=red,solid,line width=2.0pt,mark=diamond,mark options={solid}]
  table[row sep=crcr]{%
     3.581723141739799e-01     5.296398731810582e-02\\
     2.564569543970479e-01     1.779300073946707e-02\\
     1.870455127132988e-01     7.399345160889286e-03\\
     1.363219521552116e-01     2.260784754113515e-03\\
     9.671538020774045e-02     8.123111760372912e-04\\
     };
\addlegendentry{$\| (\bm e_h^u,\bm e_h^w,e_h^\varphi)\|_{\rm E}$};

\addplot [color=black,dashed,line width=1.5pt]
  table[row sep=crcr]{%
0.35329917294355	0.00440989107359586\\
0.25283535049205	0.00161626804205538\\
0.187501752721944	0.00065919817341206\\
0.136755485644054	0.000255760968810901\\
};
\addlegendentry{$h^{3}$};

\end{axis}

\end{tikzpicture}%

%% file: open_pores_N.tikz
\definecolor{mycolor1}{rgb}{0.00000,0.44700,0.74100}%
\definecolor{mycolor2}{rgb}{0.85000,0.32500,0.09800}%
\definecolor{mycolor3}{rgb}{0.92900,0.69400,0.12500}%
\begin{tikzpicture}

\begin{axis}[%
width=0.33\textwidth,
height=0.33\textwidth,
scale only axis,
xmin=0,
xmax=6,
xlabel={$p$},
xmajorgrids,
ymode=log,
ymin=1e-8,
ymax=1,
yminorticks=true,
ymajorgrids,
yminorgrids,
axis background/.style={fill=white},
title={Open pores ($\kint=1$)},
legend style={legend cell align=left,align=left,draw=white!15!black}
]
\addplot [color=mycolor1,solid,line width=2.0pt,mark=,mark options={solid}]
  table[row sep=crcr]{%
     1.000000000000000e+00     1.168413773164724e-02\\
     2.000000000000000e+00     9.110036141239667e-04\\
     3.000000000000000e+00     5.409251244231194e-05\\
     4.000000000000000e+00     3.626905401376469e-06\\
     5.000000000000000e+00     1.022982160706756e-07\\
};
\addlegendentry{$\|\bm{e}_h^u \|_{\Omega_p}$};

\addplot [color=mycolor2,solid,line width=2.0pt,mark=o,mark options={solid}]
  table[row sep=crcr]{%
     1.000000000000000e+00     8.817449338351227e-03\\
     2.000000000000000e+00     6.330435375557821e-04\\
     3.000000000000000e+00     3.993770183520082e-05\\
     4.000000000000000e+00     2.528389032976360e-06\\
     5.000000000000000e+00     7.898219493166855e-08\\
};
\addlegendentry{$\|\bm{e}_h^w \|_{\Omega_p}$};

\addplot [color=mycolor3,solid,line width=2.0pt,mark=+,mark options={solid}]
  table[row sep=crcr]{%
     1.000000000000000e+00     1.558506121581035e-02\\
     2.000000000000000e+00     1.313285276986313e-03\\
     3.000000000000000e+00     6.764476445360933e-05\\
     4.000000000000000e+00     3.449258055338023e-06\\
     5.000000000000000e+00     1.263029241203747e-07\\
};
\addlegendentry{$\|e_h^\varphi  \|_{\Omega_a}$};

\addplot [color=black,dashed,line width=1.5pt]
  table[row sep=crcr]{%
  1   5.502322005640724e-02\\
  2   3.027554745375815e-03\\
  3   1.665858109876335e-04\\
  4   9.166087736247617e-06\\
  5   5.043476625678880e-07\\
};
\addlegendentry{$e^{-3\textcolor{red}{p}}$};   

\end{axis}

\end{tikzpicture}%

%% file: sealed_pores_pressure.tikz
\definecolor{mycolor1}{rgb}{0.00000,0.44700,0.74100}%
\definecolor{mycolor2}{rgb}{0.85000,0.32500,0.09800}%
\definecolor{mycolor3}{rgb}{0.92900,0.69400,0.12500}%
\begin{tikzpicture}
\begin{axis}[%
width=0.4\textwidth,
height=0.4\textwidth,
scale only axis,
xmode=log,
xmin=0.05,
xmax=0.6,
xminorticks=true,
xlabel={$h$},
xmajorgrids,
xminorgrids,
ymode=log,
ymin=0.00001,
ymax=0.01,
yminorticks=true,
ymajorgrids,
yminorgrids,
title={$\| p -p_h \|_{\Omega}$},
axis background/.style={fill=white},
legend style={at={(0.03,0.97)},anchor=north west,legend cell align=left,align=left,draw=white!15!black}
]

\addplot [color=mycolor1,solid,line width=2.0pt,mark=,mark options={solid}]
  table[row sep=crcr]{%
     3.706747546840842e-01     2.978634524372338e-03\\
     2.496245525522082e-01     7.832732334696137e-04\\
     1.825964801200949e-01     2.785986547973479e-04\\
     1.308540523483373e-01     1.165413564678970e-04\\
     1.012898589583347e-01     4.085028729717422e-05\\
};
\addlegendentry{$\kint=0$};

\addplot [color=mycolor2,solid,line width=2.0pt,mark=o,mark options={solid}]
  table[row sep=crcr]{%
     3.706747546840842e-01     2.860982475272565e-03\\
     2.496245525522082e-01     7.479888791365233e-04\\
     1.902011899904271e-01     3.126808453864335e-04\\
     1.362344921556706e-01     1.172607895255750e-04\\
     9.582420607016751e-02     3.865722000141522e-05\\
};
\addlegendentry{$\kint=\frac{1}{2}$};

\addplot [color=mycolor3,solid,line width=2.0pt,mark=+,mark options={solid}]
  table[row sep=crcr]{%
     3.581723141739799e-01     3.766520317319817e-03\\
     2.564569543970479e-01     9.596233680281444e-04\\
     1.870455127132988e-01     3.440569263845758e-04\\
     1.363219521552116e-01     1.334768713465238e-04\\
     9.671538020774045e-02     4.762310388546280e-05\\
};
\addlegendentry{$\kint=1$};

\addplot [color=black,dashed, line width=1.5pt]
  table[row sep=crcr]{%
0.365752072770996	0.00489283294568612\\
0.272105667490156	0.00201471102230995\\
0.180496321976427	0.000588037563954273\\
0.133499390655323	0.000237923779546922\\
};
\addlegendentry{$h^{3}$};

\end{axis}

\end{tikzpicture}%